\documentclass[12pt, reqno]{amsart}
\usepackage[margin=0.75in]{geometry}
\usepackage[
  bookmarks=true]{hyperref}
\usepackage[OT2, T1]{fontenc}
\usepackage[english]{babel}
\usepackage[utf8]{inputenc}
\usepackage{csquotes}
\usepackage[final]{microtype}
\usepackage{lmodern}
\usepackage{amsthm}
\usepackage{amssymb}
\usepackage{mathrsfs}
\usepackage{enumerate}
\usepackage{tikz-cd} 
\usepackage{tikz}
\usepackage{multicol}
\usepackage{multirow}
\usepackage{tikz-qtree}
\usepackage{rotating}
\usepackage{graphicx}
\usepackage{comment}
\usepackage{enumitem}
\usepackage{manfnt}
\usepackage{mathtools}
\usepackage{diagbox}
\usepackage{multirow}
\usepackage{booktabs}
\usepackage{listings} 
\usetikzlibrary{arrows,calc,matrix,trees,arrows.meta,positioning,decorations.pathreplacing,bending}

\newtheorem{theorem}{Theorem}[section]
\newtheorem*{theorem*}{Theorem}
\newtheorem*{goal*}{Goal}
\newtheorem*{question*}{Question}

\newtheorem{proposition}[theorem]{Proposition}
\newtheorem{lemma}[theorem]{Lemma}

\theoremstyle{definition}

\newtheorem{remark}[theorem]{Remark}
\newtheorem{definition}[theorem]{Definition}
\newtheorem{example}[theorem]{Example}

\newtheorem{condition}[theorem]{Condition}

\newcommand{\Oh}{\mathcal{O}} 
\newcommand{\Sel}{\mathrm{Sel}} 
\newcommand{\Gal}{\mathrm{Gal}} 

\newcommand{\genlegendre}[4]{%
  \genfrac{(}{)}{}{#1}{#3}{#4}%
  \if\relax\detokenize{#2}\relax\else_{\!#2}\fi
}

\DeclareSymbolFont{cyrletters}{OT2}{wncyr}{m}{n}
\DeclareMathSymbol{\Sha}{\mathalpha}{cyrletters}{"58}

\newcommand{\DK}[1]{\textcolor{purple}{$\clubsuit$ DK: #1 $\clubsuit$ }}

\usepackage{graphicx}
\title{Distribution of Selmer ranks in prime cyclic extensions}
\author{Daniel Keliher}
\address{Daniel Keliher, Department of Mathematics and Statistics, Middlebury College, Middlebury, VT 05735, USA.}
\email{dkeliher@middlebury.edu}

\author{Sun Woo Park}
\address{Sun Woo Park, Max Planck Institute for Mathematics,  Vivatsgasse 7, 53111 Bonn, Germany}
\email{s.park@mpim-bonn.mpg.de}

\begin{document}

\begin{abstract}

    Using modifications to work of Klagsbrun, Mazur, and Rubin, we study (assuming the Extended Riemann Hypothesis) the distribution of Selmer ranks of twist families of some given even-dimensional Galois modules satisfying some mild technical conditions. As a corollary, we study the probability with which a fixed elliptic curve gains (or does not gain) rank in $p$-cyclic extensions, obtaining bounds for this distribution. Likewise, for some superelliptic curves $C$, we bound the average size of $C(L)$ as $L$ ranges over $p$-cyclic extensions over a number field $K$ containing primitive $p$-th roots of unity. Lastly, we study the probability with which a fixed hyperelliptic curve gains (or does not gain) rank in quadratic extensions, also obtaining bounds for this distribution. In all three cases, the extensions under consideration are ordered by the product of ramified primes. 
\end{abstract}
\maketitle
\tableofcontents

\section{Introduction}


The average size of 2-Selmer groups in twist families of an elliptic curve in various regimes has been determined in works of Swinnerton-Dyer \cite{SDtwists}, Heath-Brown \cite{HeathBrown}, Kane \cite{Kane2013}, Klagsbrun, Mazur and Rubin \cite{KMR13, KMR14} and, most recently, Pan and Tian \cite{PT25}. Work of Smith \cite{Smith1, Smith2} computes the average size of $2^\infty$-Selmer groups in such twist families over number fields. Work of Poonen and Rains \cite{PR12} obtains the same distribution in a heuristic model obtained by modeling Selmer groups as random intersections of isotropic subspaces of a suitable quadratic space

In each of \cite{SDtwists, Kane2013, KMR13, KMR14}, the distribution of 2-Selmer groups is characterized as the stationary distribution of an irreducible, aperiodic Markov process over a countable state space. In \cite{KMR13, KMR14}, the average sizes of $p$-Selmer ranks of prime twists of 2-dimensional Galois modules are also computed. This latter distribution is with respect to a ``non-standard'' order given by the \textit{fan structure} (see  Section \ref{sec:KMRfan}). Our first goal is to obtain the same distribution in an alternative ordering; this is the content of Theorem \ref{thm:fanB}.

A natural application of the distribution of $p$-Selmer ranks is studying how often the rank of the fixed elliptic curve gains rank upon base extension to cyclic $p$-extensions of a number field $K$. Given such an extension $L/K$, there is a $(p-1)$-dimensional abelian variety $A_{L/K}$ such that 
$$\mathrm{rk}(E/L) - \mathrm{rk}(E/K) \leq (p-1)\dim_{\mathbb{F}_p}\mathrm{Sel}_{1-\sigma}(A_{L/K}/K);$$
see Example \ref{ex:Twists}.

Theorem \ref{thm:fanB} gives us access to the distribution of $\dim_{\mathbb{F}_p}\mathrm{Sel}_{1-\sigma_p}(A_{L/K}/K)$ as $L/K$ varies.   Let $\mathcal{F}_{p}(K;X)$ denote the set of isomorphism classes of cyclic $p$-extensions $F/K$ with relative discriminant ideal $\Delta_{F/K}$, for which $\prod_{\mathfrak{p} \mid \Delta_{F/K}} \mathrm{Nm}_{\mathbb{Q}}^K \mathfrak{p} \leq X$. Set $N_p(K; X) = \# \mathcal{F}_{p}(K;X)$. Then, as a consequence of Theorem \ref{thm:fanB}, we obtain the following.
\begin{theorem}\label{thm:intro}
    Assume the Extended Riemann Hypothesis. Fix a prime $p$, an elliptic curve $E/K$ for which $\mathrm{Gal}(K(E[p])/K) \supseteq \mathrm{SL}_2(\mathbb{F}_p)$, and a non-negative integer $r$. For any non-negative integer $j \geq 0$, we denote by $PR(j) := \prod_{k=0}^\infty \frac{1}{1 + p^{-k}} \prod_{k=1}^j \frac{p}{p^k-1}$.
    \begin{itemize}
        \item Suppose $p = 2$. If $r = 0$, then there exists an explicitly computable constant $\delta(E/K) \in [-1/2,1/2]$ depending only on $E$ and $K$ such that 
        \begin{equation}\label{eq:intromain1-2}
            \liminf_{X \to \infty} \frac{\#\{F \in \mathcal{F}_2(K;X) \mid \mathrm{rk}(E/F) - \mathrm{rk}(E/K) =0\}}{N_2(K, X)} \geq 2 \cdot \left(\frac{1}{2} + \delta(E/K) \right) \cdot PR(0).
        \end{equation}
        If $r > 0$, then
        \begin{align}\label{eq:intromain2-2}
        \begin{split}
            & \; \; \; \; \limsup_{X \to \infty} \frac{\#\{K \in \mathcal{F}_p(K;X) \mid \mathrm{rk}(E/F) - \mathrm{rk}(E/K) > (r-1) \}}{N_p(K, X)} \\
            &\leq \sum_{\substack{j=r, \\ j \equiv 0 \mod 2}}^\infty 2 \cdot \left(\frac{1}{2} + \delta(E/K) \right) \cdot PR(j) + \sum_{\substack{j=r, \\ j \equiv 1 \mod 2}}^\infty 2 \cdot \left(\frac{1}{2} - \delta(E/K) \right) \cdot PR(j).
        \end{split}
        \end{align}
        \item Suppose $p$ is odd. If $r=0$, then
        \begin{equation}\label{eq:intromain1}
            \liminf_{X \to \infty} \frac{\#\{F \in \mathcal{F}_p(K;X) \mid \mathrm{rk}(E/F) - \mathrm{rk}(E/K) =0\}}{N_p(K, X)} \geq PR(0).
        \end{equation}
        If $r > 0$, then
        \begin{equation}\label{eq:intromain2}
            \limsup_{X \to \infty} \frac{\#\{K \in \mathcal{F}_p(K;X) \mid \mathrm{rk}(E/F) - \mathrm{rk}(E/K) > (p-1) \cdot (r-1) \}}{N_p(K, X)} \leq \sum_{j=r}^\infty PR(j).
        \end{equation}
    \end{itemize}
\end{theorem}

Note that the right-hand side of \eqref{eq:intromain2} decays quadratic exponentially in $r$ as $r$ grows. We also note that $\delta(E/K) = 0$ if $K$ has a real embedding, see for example \cite[Corollary 7.10]{KMR13}. We estimate a few values of $PR(j)$ and $\sum_{j=r}^\infty PR(j)$ for various $r$ and $p$ in the Tables \ref{table:pr0} and \ref{table:sum_prj}, respectively.

\begin{table}[ht]
    \centering
    \begin{tabular}{c||c|c|c|c|c|c}
         $p$ & 2 & 3 & 5 & 7 & 11 & 13   \\ \hline 
         $PR(0)$ & 0.2097 & 0.3195 & 0.3967  & 0.4274 & 0.4542 & 0.4613 
    \end{tabular}
\caption{Estimates for $PR(0)$}
\label{table:pr0}
\vspace{-25pt}
\end{table}
\begin{table}[ht]
\centering
    \begin{tabular}{c||c|c|c|c|c}
        \backslashbox{$p$}{$r$} & 1 & 2 & 3 & 4 & 5   \\ \hline \hline
        2 &  0.7903 & 0.3709  & 0.0913 & 0.0117 & 0.0007 \\
        3 & 0.6805 & 0.2012 & 0.0215 & 0.0008 & $9.7\times10^{-6}$ \\
        5 & 0.6033 & 0.1075 & 0.0042 & $3.3 \times 10^{-5}$ & $5.3 \times 10^{-8}$ \\
        7 & 0.5727 &  0.0742 & 0.0015 & $4.3 \times 10^{-6}$ & $1.8 \times 10^{-9}$\\
    \end{tabular}  \\
    \caption{Estimates for $\sum_{j=r}^\infty PR(j)$.}
    \label{table:sum_prj}
    \vspace{-10pt}
\end{table}

As a second application of the main theorem,  we obtain an effective Mordell-Lang conjecture for twist families of some superelliptic curves over some number fields. Fix a prime $p \geq 5$ and let $K/\mathbb{Q}$ be a number field containing $\mathbb{Q}(\zeta_p)$. Let $C: y^p = F(x) := x^3 + a_2x^2 + a_1x + a_0$ be a superelliptic curve over $K$ such that the splitting field of $F(x)$ is a Galois $S_3$-extension over $K$. We denote by $\mathcal{S}_p(C,X)$ the set of isomorphism classes of twist families $C_d: dy^p = F(x)$ of $C$ over $K$ such that $\mathrm{Gal}(K(\sqrt[p]{d})/K) \cong \mathbb{Z}/p\mathbb{Z}$, and $\prod_{\mathfrak{p} \mid (d) \subset \mathcal{O}_K} \mathrm{Nm}_\mathbb{Q}^K \mathfrak{p} < X$.
\begin{theorem}\label{thm:intro_superelliptic}
    Assume the Extended Riemann Hypothesis. Then there exists an explicit constant $B(p,K) > 0$, depending only on $p$ and $K$, such that
\begin{equation}
    \limsup_{X \to \infty}  \frac{\sum_{C_d \in \mathcal{S}_p(C,X)}\# C_d(K)}{\# \mathcal{S}_p(C,X)} \leq B(p,K).
\end{equation}
\end{theorem}

A slight modification of the proof of the main theorem allows us to obtain lower bounds on the probability that the Jacobian of a hyperelliptic curve satisfying some mild conditions gains Mordell-Weil rank upon base change over quadratic extensions. We note that the theorem below is consistent with the groundbreaking result by Alex Smith \cite[Theorem 2.14, Example 3.7]{Smith2} for quadratic twist families of some odd hyperelliptic curves, and generalizes his results to quadratic twist families of some other hyperelliptic curves which are not odd.

\begin{theorem}\label{thm:intro_hyperelliptic}
Assume the Extended Riemann Hypothesis. Let $C: y^2 = f(x)$ be a hyperelliptic curve over $K$ for which the splitting field of $f$ over $K$ is a $S_{\deg f}$-Galois extension over $K$. For any non-negative integer $j \geq 0$, we denote by $PR(j) := \prod_{k=0}^\infty \frac{1}{1 + 2^{-k}} \prod_{k=1}^j \frac{2}{2^k - 1}$.

Then there exists an explicitly computable constant $\delta(C/K) \in [-1/2, 1/2]$ depending only on $C$ and $K$ such that
\begin{equation}\label{eq:intromain3-2}
            \liminf_{X \to \infty} \frac{\#\{F \in \mathcal{F}_2(K;X) \mid \mathrm{rk}(\mathrm{Jac}(C)/F) - \mathrm{rk}(\mathrm{Jac}(C)/K) =0\}}{N_2(K, X)} \geq 2 \cdot \left(\frac{1}{2} + \delta(C/K) \right) \cdot PR(0).
        \end{equation}
        If $r > 0$, then
        \begin{align}\label{eq:intromain3-3}
        \begin{split}
            & \; \; \; \; \limsup_{X \to \infty} \frac{\#\{K \in \mathcal{F}_2(K;X) \mid \mathrm{rk}(\mathrm{Jac}(C)/F) - \mathrm{rk}(\mathrm{Jac}(C)/K) > (r-1) \}}{N_2(K, X)} \\
            &\leq \sum_{\substack{j=r, \\ j \equiv 0 \mod 2}}^\infty 2 \cdot \left(\frac{1}{2} + \delta(C/K) \right) \cdot PR(j) + \sum_{\substack{j=r, \\ j \equiv 1 \mod 2}}^\infty 2 \cdot \left(\frac{1}{2} - \delta(C/K) \right) \cdot PR(j).
        \end{split}
        \end{align}
\end{theorem}

\subsection*{Outline} In Section \ref{sec:Selmer} we define Selmer structures, explain their connections to Selmer groups, and elaborate their applications to Jacobians of superelliptic curves and Selmer rank distributions in certain twist families. In Section \ref{sec:Fan} we recall the setup of \cite{KMR13, KMR14} and define the fan structure over which their distributions are defined. In Section \ref{sec:analytic-intermission} we show that the newly constructed fan structure forms a density 1 subset among the set of square-free ideals of bounded norm. In Section \ref{sec:generalthm} we state and prove a distribution (Theorem \ref{thm:fanB}) analogous to those of Klagsbrun-Mazur-Rubin, but where the ordering on characters is by the product of ramified primes rather than by the fan structure. This allows us to prove Theorem \ref{thm:intro}. In Section \ref{sec:superelliptic} we prove Theorem \ref{thm:intro_superelliptic} for cyclic twist families of some superelliptic curves. In Section \ref{sec:hyperelliptic} we prove Theorem \ref{thm:intro_hyperelliptic} by using the setup of \cite{Yu19} for quadratic twist families of some hyperelliptic curves. We finish the manuscript with a summary of the overarching philosophy of the techniques employed in this manuscript in Section \ref{sec:summary}.

\subsection*{Acknowledgements}
We would like to thank Peter Koymans for suggesting to use the radical of an ideal, which is used in stating the main theorem for number fields. We would like to also thank Valentin Blomer and Peter Koymans for very helpful comments and pinpointing errors in the previous version of the manuscript. We would like to thank Jef Laga, Jungin Lee, Carlo Pagano, Alexander Smith, and Melanie Matchett-Wood for helpful discussions. The second author would like to thank the Max Planck Institute for Mathematics for its generous support, Mathematisches Forschungsinstitut Oberwolfach for organizing a workshop on arithmetic statistics for algebraic objects, the Lodha Mathematical Sciences Institute for organizing a thematic program on arithmetic statistics, and the Centre de Recherches Math\'ematiques for organizing a semester program on arithmetic statistics, from which the author had helpful discussions regarding this work with other participants.

\section{Selmer structures associated to twists of Galois modules}\label{sec:Selmer}

We recall the notion of Selmer structures associated to twist families of a fixed $\mathrm{Gal}(\overline{K}/K)$-module satisfying a number of mild conditions. We closely follow the notations used in \cite{KMR13, KMR14, Park25-1}. Throughout, we let $K$ denote a number field and $p$ denote a fixed prime number. 
\begin{condition} \label{condition:T}
Following the notations from \cite{KMR13, KMR14, Park25-1}, let $T$ be a Galois module over $K$ satisfying the following conditions:
\begin{itemize}
    \item $T$ is a $2$-dimensional vector space over $\mathbb{F}_p$,
    \item $T$ admits a continuous action of $\mathrm{Gal}(\overline{K}/K)$,
    \item $T$ admits a non-degenerate, $\mathrm{Gal}(\overline{K}/K)$-equivariant pairing $T \times T \to \mu_p$ which induces a symmetric pairing $H^1(K_v, T) \times H^1(K_v, T) \to \mathbb{F}_p$ for all places $v$ of $K$,
    \item The action of $\mathrm{Gal}(\overline{K}/K)$ on $T$ is irreducible,
    \item We have $\mathrm{Hom}_{\mathrm{Gal}(\overline{K}/K(\mu_p))}(T,T) = \mathbb{F}_p$, and
    \item We have $H^1(\mathrm{Gal}(K(T)/K), T) = 0$.
\end{itemize}
\end{condition}

\begin{example}
    Let $E$ be an elliptic curve over a number field $K$ and let $p$ be any prime. If $\mathrm{Gal}(K(E[p])/K)) \supset \mathrm{SL}_2(\mathbb{F}_p)$, then $T := E[p]$ endowed with the Weil pairing satisfies Condition \ref{condition:T}.
\end{example}

\begin{definition} Let $K$ be a number field and let $T$ be as above.  For each $i = 0, 1, 2$, we denote by $\delta_i$ the following asymptotic densities:
\begin{equation} \label{eq:delta_i}
\delta_i = \lim_{X \to \infty}\frac{\#\{\mathfrak{p} \in \mathcal{O}_K \mid \mathrm{Nm}_\mathbb{Q}^K \mathfrak{p} \leq X, \dim H^1(K_\mathfrak{p}, T) =2i \}}{\#\{\mathfrak{p} \in \mathcal{O}_K \mid \mathrm{Nm}_\mathbb{Q}^K \mathfrak{p} \leq X \}}.
\end{equation}
Note the limit defining each $\delta_i$  exists and can be obtained by applying the Chebotarev density theorem to the Galois extension $K(T)/K$; see the following definition.
\end{definition}

\begin{definition}[cf. Definition 3.3 and Definition 5.11 of \cite{KMR14}]
    We make following notations.
    \begin{itemize}
        \item $\Sigma_T$: a finite set of places of $K$ containing places where $T$ is ramified, places above $p$, all Archimedean places, and all places whose associated prime ideals generate $\mathrm{Cl}(K)/\ell \mathrm{Cl}(K)$.
        \item $\mathfrak{d}$: a square-free product of places of $K$.
        \item $\Sigma_T(\mathfrak{d})$: a set of places $\Sigma_T \cup \left\{ \mathfrak{q} : \mathfrak{q} \mid \mathfrak{d} \right\}$.
        \item $\mathcal{P}$: the set of primes of $K$.
        \item $\mathcal{P}_i$: the set of primes $\mathfrak{q}$ of $K$ such that $\mathfrak{q} \not\in \Sigma$ and $\dim_{\mathbb{F}_p} H^1(K_\mathfrak{q}, T) = 2i$. We note that $\mathfrak{q} \in \mathcal{P}_i$ is equivalent to the following Chebotarev condition with respect to the Galois extension $K(T)/K$:
        \begin{itemize}
            \item $\mathfrak{q} \in \mathcal{P}_0$ if and only if $\mathrm{Frob}_\mathfrak{q} \in \mathrm{Gal}(K(T)/K)$ has order  equal to neither $1$ nor $p$.
            \item $\mathfrak{q} \in \mathcal{P}_1$ if and only if $\mathrm{Frob}_\mathfrak{q} \in \mathrm{Gal}(K(T)/K)$ has order equal to $p$.
            \item $\mathfrak{q} \in \mathcal{P}_2$ if and only if $\mathrm{Frob}_\mathfrak{q} \in \mathrm{Gal}(K(T)/K)$ is trivial.
        \end{itemize}
        \item $\Omega_1$: the Cartesian product of local characters
        \begin{equation*}
            \Omega_1 := \prod_{v \in \Sigma_T} \mathrm{Hom}(\mathrm{Gal}(\overline{K}_v/K_v), \mu_p).
        \end{equation*}
        \item $\Omega_\mathfrak{d}$: the Cartesian product of local characters
        \begin{equation*}
            \Omega_\mathfrak{d} := \prod_{v \in \Sigma_T} \mathrm{Hom}(\mathrm{Gal}(\overline{K}_v/K_v), \mu_p) \times \prod_{\substack{q \mid \mathfrak{d} \\ q \not\in \Sigma_T}} \mathrm{Hom}_{ram}(\mathrm{Gal}(\overline{K}_\mathfrak{q}/K_\mathfrak{q}),\mu_p),
        \end{equation*}
        where $\mathrm{Hom}_{ram}(\mathrm{Gal}(\overline{K}/K),\mu_p)$ is the set of ramified local characters of $\Gal(\overline{K}/K)$ of order $p$.
    \end{itemize}
\end{definition}

\begin{definition}[cf. Definitions 5.1, 5.2, 5.3, and 5.4 of \cite{KMR14}, and Definitions 2.3 and 2.4 of \cite{Park25-1}]
    We recall the following notations.
    \begin{itemize}
        \item $q_v$: a Tate quadratic form on $H^1(K_v, T)$ induced from the pairing $\mu_T: T \times T \to \mu_p$ satisfying the third axiom of Condition \ref{condition:T}.
        \item $\mathcal{H}(q_v)$: the set of maximal isotropic subspaces of $H^1(K_v, T)$ with respect to the quadratic form $q_v$.
        \item $\mathcal{H}_{ram}(q_v)$: the set of maximal isotropic subspaces $\mathcal{H}$ of $H^1(K_v, T)$ such that $\mathcal{H} \cap H^1(K_v, T) = 0$ for any place $v \not\in \Sigma_T$. Note that $\# \mathcal{H}_{ram}(q_\mathfrak{q}) = p^{i-1}$ for all $\mathfrak{q} \in \mathcal{P}_1 \cup \mathcal{P}_2$.
    \end{itemize}
\end{definition}

\begin{definition}[cf. Definitions 5.10 and 5.12 of \cite{KMR14}, and Definition 2.6 of \cite{Park25-1}]
   Twisting data $\alpha$ is a function
    \begin{equation}
        \alpha: \prod_{v \in \Sigma_T} \frac{\mathrm{Hom}(\mathrm{Gal}(\overline{K}_v/K_v), \mu_p)}{\mathrm{Aut}(\mu_p)} \times \prod_{\mathfrak{q} \not\in \Sigma_T} \frac{\mathrm{Hom}_{ram}(\mathrm{Gal}(\overline{K}_\mathfrak{q}/K_\mathfrak{q}), \mu_p)}{\mathrm{Aut}(\mu_p)} \to \prod_{v \in \Sigma_T} \mathcal{H}(q_v) \times \prod_{\mathfrak{q} \not\in \Sigma_T} \mathcal{H}_{ram}(q_\mathfrak{q})
    \end{equation}
    defined as follows.
    \begin{itemize}
        \item At $v \in \Sigma_T$, we have a set-theoretic map
        \begin{equation*}
            \alpha: \frac{\mathrm{Hom}(\mathrm{Gal}(\overline{K}_v/K_v), \mu_p)}{\mathrm{Aut}(\mu_p)} \to \mathcal{H}(q_v).
        \end{equation*}
        \item At $\mathfrak{q} \in \mathcal{P}_0$, we have the trivial map
        \begin{equation*}
            \alpha: \frac{\mathrm{Hom}_{ram}(\mathrm{Gal}(\overline{K}_\mathfrak{q}/K_\mathfrak{q}), \mu_p)}{\mathrm{Aut}(\mu_p)} \to \left\{ 0 \right\}.
        \end{equation*}
        \item At $\mathfrak{q} \in \mathcal{P}_1$, we have the constant map
        \begin{equation*}
            \alpha: \frac{\mathrm{Hom}_{ram}(\mathrm{Gal}(\overline{K}_\mathfrak{q}/K_\mathfrak{q}), \mu_p)}{\mathrm{Aut}(\mu_p)} \to \mathcal{H}_{ram}(q_\mathfrak{q}).
        \end{equation*}
        \item At $\mathfrak{q} \in \mathcal{P}_2$, we have the bijection
        \begin{equation*}
            \alpha: \frac{\mathrm{Hom}_{ram}(\mathrm{Gal}(\overline{K}_\mathfrak{q}/K_\mathfrak{q}), \mu_p)}{\mathrm{Aut}(\mu_p)} \to \mathcal{H}_{ram}(q_\mathfrak{q}).
        \end{equation*}
    \end{itemize}
    Given a square-free product of places $\mathfrak{d}$, we denote by $\alpha_\mathfrak{d}$ the restriction of the domain of $\alpha$ over the set of places lying in $\Sigma_T(\mathfrak{d})$.
\end{definition}

With these definitions in place, we are ready to define Selmer groups and Selmer structures, which will be the main subject of our inquiry.  

\begin{definition}[cf. Section 5 and Section 6 of \cite{KMR14}]
    Given a Cartesian product of characters $\omega \in \Omega_\mathfrak{d}$, we define the Selmer structure as
    \begin{equation}
        \mathrm{Sel}(T, \omega : \alpha) := \mathrm{Ker}\left( H^1(K,T) \to \bigoplus_{v \in \Sigma_T(\mathfrak{d})} \frac{H^1(K_v,T)}{\alpha_\mathfrak{d}(\omega)} \oplus \bigoplus_{v \not\in \Sigma_T(\mathfrak{d})} \frac{H^1(K_v,T)}{H^1_{ur}(K_v,T)} \right).
    \end{equation}
    Similarly, given a global character $\chi \in \mathrm{Hom}(\mathrm{Gal}(\overline{K}/K), \mu_p)$, we define the Selmer structure
    \begin{equation}
        \mathrm{Sel}(T, \chi : \alpha) := \mathrm{Sel}(T, (\chi_v)_{v \in \Sigma_T(\mathfrak{d})} : \alpha),
    \end{equation}
    where $\mathfrak{d}$ is the squarefree product of ramified places of $\chi$, and $\chi_v$ is the restriction of $\chi$ at the place $v$ of $K$.
\end{definition}

We conclude this section with two examples that highlight how the discussion above will be relevant to each of the main theorems. The case for quadratic twist families of hyperelliptic curves will be specified in Section \ref{sec:hyperelliptic}.

\begin{example}[cf. Section 6 of \cite{KMR14}]\label{ex:Twists}
    Choose a prime number $p$. Let $E$ be an elliptic curve over a number field $K$ such that $\mathrm{Gal}(K(E[p])/K) \supset \mathrm{SL}_2(\mathbb{F}_p)$. Given a global character $\chi \in \mathrm{Hom}(\mathrm{Gal}(\overline{K}/K), \mu_p)$, let $L/K$ be the fixed field of $\mathrm{Ker}(\chi)$. Given a place $v$ of $K$, we denote by $\chi_v$ the restriction of $\chi$ to $\mathrm{Hom}(\mathrm{Gal}(\overline{K_v}/K), \mu_p)$. Denote by $\sigma$ the generator of the cyclic Galois group $\mathrm{Gal}(L/K)$. Denote by $\Sigma(\chi)$ the set of places of $K$ consisting of places of bad reduction of $E$ and ramified places of $L/K$.

    Define the following $p-1$ dimensional abelian variety over $K$ associated to $L$:
    \begin{equation}
        A^{\chi} := \mathrm{Ker} \left( \mathrm{Nm}_K^L: \mathrm{Res}_{K}^{L} E \to E \right)
    \end{equation}
      where $\mathrm{Res}$ denotes the Weil restriction from $L$ to $K$. Likewise, given a place $v$ of $K$, set
    \begin{equation*}
        A^{\chi_v} := \mathrm{Ker} \left( \mathrm{Nm}_{K_v}^{L_v}: \mathrm{Res}_{K_v}^{L_v} E \to E \right),
    \end{equation*}
    which is a $p-1$ dimensional abelian variety over $K_v$. We define the twisting data $\alpha$ as the map
    \begin{equation}
        \alpha  (\overline{\chi_v})_{v \in \Sigma(\chi)} = \prod_{v \in \Sigma(\chi)} \frac{A^{\chi_v}(K_v)}{(1-\sigma)A^{\chi_v}(K_v)}.
    \end{equation}
    Then one has
    \begin{equation} \label{eq:SameSelmer}
        \mathrm{Sel}_{1-\sigma}(A^{\chi}/K) = \mathrm{Sel}(E[p], \chi : \alpha),
    \end{equation}
    whose validity follows from the $\mathrm{Gal}(\overline{K}/K)$-equivariant isomorphism
    \begin{equation}
        A^\chi[1-\sigma] \cong E[p],
    \end{equation}
    see for example \cite[Proposition 4.1]{MR07} for more details.
\end{example}

\begin{example}[cf. Section 2 of \cite{Yu16}, Section 3 of \cite{Park25-1}, and Section 3 of \cite{PR12-theta}] \label{example:superelliptic}
    Choose a prime number $p \geq 5$. Let $C: y^p = x^3 + a_2x^2 + a_1x + a_0$ be a superelliptic curve over a number field $K$ such that $K \supset \mathbb{Q}(\zeta_p)$. Note $C$ has genus $p-1$, and the projective curve $C$ (obtained after normalization of projectivization of the affine model of $C$) has a unique point at infinity \cite[Section 3]{Sch98}.
    
    Given a $p$-th power free element $d \in K^\times / (K^\times)^p$, we denote by $C_d: dy^p = x^3 + a_2x^2 + a_1x + a_0$ the order-$p$ twist of $C$. Let $\chi^d \in \mathrm{Hom}(\mathrm{Gal}(\overline{K}/K), \mu_p)$ be a character such that the fixed field of $\mathrm{Ker}(\chi^d)$ is $K(\sqrt[p]{d})$. Given a place $v$ of $K$, denote by $\chi^d_v$ the restriction of $\chi^d$ to $\mathrm{Hom}(\Gal(\overline{K}_v/K_v),\mu_p)$. Denote by $\Sigma(d)$ the set of places of $K$ consisting of places of bad reduction of $\mathrm{Jac}(C)$ and places dividing $d$.

    Let $\zeta_p: C_d \to C_d$ be a morphism which is defined as $(x,y) \mapsto (x, \zeta_p y)$. Consider the induced morphism $\zeta_p: \mathrm{Jac}(C_d) \to \mathrm{Jac}(C_d)$. The torsion submodule $\mathrm{Jac}(C_d)[1-\zeta_p]$ is a 2-dimensional $\mathbb{F}_p$ vector space with a continuous $\mathrm{Gal}(\overline{K}/K)$-action. The principal polarization on $\mathrm{Jac}(C_d)$ induces a non-degenerate bilinear pairing
    \begin{equation}
        \mathrm{Jac}(C_d)[1-\zeta_p] \times \mathrm{Jac}(C_d)[1-\zeta_p] \to \mu_p.
    \end{equation}
    Because $1-\zeta_p$ is self-dual (\cite[Proposition 3.1]{Sch98}), we can use the form $q_v$ constructed in \cite[Lemma 3.4]{KMR13} or \cite[p.260, Case I]{PR12} to induce a non-degenerate quadratic form $q_v: H^1(K_v, \mathrm{Jac}(C_d)[1-\zeta_p]) \to \mathbb{F}_p$ for every place $v$ of $K$. (One can then construct a symmetric bilinear form from $q_v$ via $B(x,y) := q_v(x+y) - q_v(x)-q_v(y)$). Using this fact, we shall endow a global metabolic structure on $\mathrm{Jac}(C)[1-\zeta_p]$ (as in \cite[Definition 3.3]{KMR13}), from which all the consequential results stated in \cite{KMR13, KMR14} can be utilized. In particular, one does not need to show that the pairing on $\mathrm{Jac}(C_d)[1-\zeta_p]$ is alternating, which is one of the conditions assumed throughout \cite{KMR13, KMR14}. We will discuss more on which components of \cite{KMR13, KMR14} we will use in Section \ref{sec:Fan}.

    We have a $\mathrm{Gal}(\overline{K}/K)$-equivariant isomorphism
    \begin{equation}
        \mathrm{Jac}(C_d)[1-\zeta_p] \cong \mathrm{Jac}(C)[1-\zeta_p].
    \end{equation}
    This follows from combining the following two facts. One, that $\mathrm{Jac}(C)[1-\zeta_p]$ is spanned by divisors of form $((\alpha_i, 0) - (\alpha_j, 0))$ where $(\alpha_i, 0)$ is an affine Weierstrass point of $C$ (\cite[Proposition 3.2]{Sch98}): The other, that the $x$-coordinates of Weierstrass points of $C_d$ are multiples of those of Weierstrass points of $C$ by powers of $d$. We refer to \cite[Lemma 3.5]{Park25-1} for further details.
    
    Define the twisting data $\alpha$ as the map
    \begin{equation}
        \alpha((\chi^d_v)_{v \in \Sigma(d)}) = \prod_{v \in \Sigma(d)} \frac{(\mathrm{Jac}(C_d))(K_v)}{(1-\zeta_p)(\mathrm{Jac}(C_d))(K_v)}.
    \end{equation}
    Note that $\alpha$ satisfies all the conditions of being a twisting data as shown in \cite[Lemma 2.16]{Yu16}.
    Then one has
    \begin{equation}
        \mathrm{Sel}_{1-\zeta_p}(\mathrm{Jac}(C_d)/K) = \mathrm{Sel}(\mathrm{Jac}(C)[1-\zeta_p], \chi^d : \alpha).
    \end{equation}
     Importantly, note $\mathrm{Jac}(C)[1-\zeta_p]$ satisfies Condition \ref{condition:T} if the splitting field of the cubic polynomial $x^3 + a_2 x^2 + a_1x + a_0$ is an $S_3$ Galois extension of $K$. See \cite[Lemma 3.6]{Park25-1}, the proof of which applies when $K$ is a global field, for all of the details. We give only a brief summary here. The action of elements of $\mathrm{Gal}(K(\mathrm{Jac}(C)[1-\zeta_p])/K) \cong S_3$ on the 2-dimensional $\mathbb{F}_p$ vector space can be represented by the following six elements of $\mathrm{GL}_2(\mathbb{F}_p)$:
     \begin{equation*}
         \left\{ \begin{pmatrix} 1 & 0 \\ 0 & 1 \end{pmatrix}, \begin{pmatrix} -1 & -1 \\ 0 & 1 \end{pmatrix}, \begin{pmatrix} 1 & 0 \\ -1 & -1 \end{pmatrix}, \begin{pmatrix} 0 & 1 \\ 1 & 0 \end{pmatrix}, \begin{pmatrix} -1 & -1 \\ 1 & 0 \end{pmatrix}, \begin{pmatrix} 0 & 1 \\ -1 & -1 \end{pmatrix} \right\}.
     \end{equation*}
     The condition that $\mathrm{Gal}(\overline{K}/K)$ acts irreducibly on $\mathrm{Jac}(C)[1-\zeta_p]$ follows from the fact that there is no proper subspace of $\mathbb{F}_p^{\oplus 2}$ which is invariant under multiplication by the six matrices. The condition that $\mathrm{Hom}_{\mathrm{Gal}(\overline{K}/K)}(\mathrm{Jac}(C)[1-\zeta_p], Jac(C)[1-\zeta_p]) = \mathbb{F}_p$ follows from the condition that only matrices of form $\begin{pmatrix} a & 0 \\ 0 & a \end{pmatrix}$ for any $a \in \mathbb{F}_p$ commute with all six matrices. Lastly, the last condition follows from the fact that the order of the group $S_3$ does not divide $p \geq 5$.
\end{example}

\section{Fan Structures}\label{sec:Fan}

The goal of this section is to construct a suitable subset of square-free elements of Galois extensions $K / \mathbb{Q}$, over which the techniques of \cite{KMR14} can be utilized to understand the distribution of Selmer structures of twist families of a fixed $\mathrm{Gal}(\overline{K}/K)$-module satisfying Condition \ref{condition:T}.

\subsection{Previous results}
\label{sec:KMRfan}

We start with a review of certain subsets of global $p$-cyclic characters of $K$, called ``fan structures,'' as first  constructed in \cite{KMR14}. Before we proceed, note that some of the notations used in this section (for instance the functions $L_i(X)$, the set $\mathcal{D}_{m,k,X}$, and the fan $\mathcal{B}_{m,k}(X)$) will be replaced by respective new constructions outlined in Section \ref{sec:newfan}.


For each positive integer $i \geq 1$, Klagsbrun, Mazur, and Rubin define a sequence of real valued functions $L_i(X)$ which satisfies
\begin{equation*}
L_{n+1}(X) := \max\left\{L_1\left(\prod_{i \leq n} L_i(X)\right), XL_n(X) \right\}.
\end{equation*}
Fix any number field $K$, a rational prime $p$, and a Galois module $T$ which is an $\mathbb{F}_p$ vector space satisfying Condition \ref{condition:T}. For any place $\mathfrak{p}$ of $K$, define the \textit{width} of $\mathfrak{p}$ by 
$$w(\mathfrak{p}) = 
\begin{cases}
    \dim_{\mathbb{F}_p} T(K_{\mathfrak{p}}) \quad &\text{if } \mu_p \in K_\mathfrak{p}^\times, \\
    0 \quad &\text{otherwise}.
\end{cases}$$
Given a square-free ideal $\mathfrak{d} \subset \mathcal{O}_K$, we define its width by $w(\mathfrak{d}):=\sum_{\mathfrak{p} \mid \mathfrak{d}} w(\mathfrak{p})$.

Next,  define the following set of square-free products of primes of $K$: 
\begin{equation}\label{eq:Dmkx}
    \mathcal{D}_{m,k,X} := \{\mathfrak{d} = \mathfrak{p}_1\mathfrak{p}_2...\mathfrak{p}_m \mid w(\mathfrak{d}) = k, \mathrm{Nm}_\mathbb{Q}^K\mathfrak{p}_i \leq L_i(X) \text{ for all } 1 \leq i \leq m \}.
\end{equation}

Let $C_p(K) = \mathrm{Hom}(G_K, \mu_p)$ be the group of global $p$-cyclic characters of $K$. For any $\mathfrak{d} \in \mathcal{D}_{m,k,X}$, define
\begin{equation}\label{eq:C(d)}
C_p(\mathfrak{d}) := \{ \chi \in C_p(K) \mid \chi \text{ ramified at all }\mathfrak{q} \mid \mathfrak{d} \text{ and unramified outside }\Sigma(\mathfrak{d}) \cup \mathcal{P}_0\},
\end{equation}
and 
\begin{equation}\label{eq:C(d,X)}
    C_p(\mathfrak{d},X) := \left\{ \chi \in C_p(\mathfrak{d}) \; | \; \mathrm{Nm}_\mathbb{Q}^K(\mathfrak{p}) < X \text{ if } \mathfrak{p} \in \Sigma(\mathfrak{d}) \cup \mathcal{P}_0 \right\}.
\end{equation}

Finally, we let 
$$\mathcal{B}_{m,k}(X) = \bigsqcup_{\mathfrak{d} \in \mathcal{D}_{m,k,X}} C_p(\mathfrak{d}, \mathcal{L}(L_{m+1}(X))) \subset C_p(K),$$
where $\mathcal{L}: [1,\infty) \to [1,\infty)$ is some non-decreasing function quantifying the error terms of a choice of an effective version of Chebotarev density theorem \cite[Theorem 8.1]{KMR14}. This set of characters $\mathcal{B}_{m,k}(X)$ is called a \textit{fan structure} on $C(K)$.

In \cite{KMR14}, the distribution of dimensions of Selmer structures of twist families of a fixed $\mathrm{Gal}(\overline{K}/K)$-module $T$ satisfying Condition \ref{condition:T} can be obtained over the set $\mathcal{B}_{m,k}(X)$. In order to state their results, we must first introduce the Markov operators necessary to understand the statistics of the Selmer structures.

\begin{definition}[cf. Definition 6.1 of \cite{park2022prime}]
Let $\mathbf{M}_L = [p_{r,s}]$ be the Markov operator over the state space of non-negative integers $\mathbb{Z}_{\geq 0}$ given by
\begin{equation*}
    p_{r,s} = \begin{cases}
    1 - p^{-r} &\text{ if } s = r-1 \geq 0, \\
    p^{-r} &\text{ if } s = r+1, \\
    0 &\text{ else}.
    \end{cases}
\end{equation*}
\end{definition}

\begin{definition}[cf. Definition 6.3 of \cite{park2022prime}]
Let $\mu: \mathbb{Z}_{\geq 0} \to [0,1]$ be a probability distribution over the state space of non-negative integers $\mathbb{Z}_{\geq 0}$. The parity of $\mu$ is the sum of probabilities at odd state spaces, i.e.
\begin{equation*}
    \rho(\mu) := \sum_{n \text{ odd}} \mu(n).
\end{equation*}
\end{definition}

\begin{definition}[cf. Proposition 2.4 of \cite{KMR14}] \label{def:Markov}
    Let $E^+, E^-: \mathbb{Z}_{\geq 0} \to [0,1]$ be probability distributions such that
\begin{equation*}
    E^+(n) = \begin{cases}
    \prod_{j=1}^\infty (1 + p^{-j})^{-1} \prod_{j=1}^n \frac{p}{p^j - 1} &\text{ if n even}, \\
    0 &\text{ if n odd}.
    \end{cases}
\end{equation*}
\begin{equation*}
    E^-(n) = \begin{cases}
    0 &\text{ if n even}, \\
    \prod_{j=1}^\infty (1 + p^{-j})^{-1} \prod_{j=1}^n \frac{p}{p^j - 1} &\text{ if n odd}.
    \end{cases}
\end{equation*}
Let $\mu: \mathbb{Z}_{\geq 0} \to [0,1]$ be a probability distribution. Then 
\begin{align*}
    \lim_{k \to \infty} \mathbf{M}_L^{2k}(\mu) &= (1 - \rho(\mu)) E^+ + \rho(\mu) E^-, \\
    \lim_{k \to \infty} \mathbf{M}_L^{2k+1}(\mu) &= \rho(\mu)E^+ + (1-\rho(\mu)) E^-.
\end{align*}
\end{definition}

The main result of \cite{KMR14} can be summarized as follows. We note that the disparity of Selmer groups $\Sel(T, \chi : \alpha)$ for $p=2$ originates from the fact that the natural restriction map $C(\mathfrak{d},X) \to \Omega_\mathfrak{d}$ is not surjective, see in particular \cite[Proposition 10.7]{KMR14}.

\begin{theorem}[Main Result of \cite{KMR14}] \label{KMR14:main}
    Let $K$ be any number field. Let $T$ be a Galois module satisfying Condition \ref{condition:T}.
    \begin{enumerate}
        \item Suppose $p = 2$. Let $\Delta$ be an element which generates $\mathrm{Hom}(\mathrm{Gal}(K(T)/K), \mu_2) \cong \mathbb{Z}/2\mathbb{Z}$. Given a fixed finite set of places $\Sigma$ containing $\Sigma_T$, denote by $\mathrm{E}^+_1(n)$ and $\mathrm{E}^-_1(n)$ the probability distributions
        \begin{align}
            \begin{split}
                \mathrm{E}^+_1(n) &:= \frac{\#\{\chi \in \mathcal{C}_2(1,X) : \dim_{\mathbb{F}_2} \mathrm{Sel}(T, \chi : \alpha) = n, \prod_{v \in \Sigma}\chi_v(\Delta) = 1\}}{\# \{\chi \in \mathcal{C}_2(1,X) : \prod_{v \in \chi_v} \chi_v(\Delta) = 1\}}, \\
                \mathrm{E}^-_1(n) &:= \frac{\#\{\chi \in \mathcal{C}_2(1,X) : \dim_{\mathbb{F}_2} \mathrm{Sel}(T, \chi : \alpha) = n, \prod_{v \in \Sigma}\chi_v(\Delta) = -1\}}{\# \{\chi \in \mathcal{C}_2(1,X) : \prod_{v \in \chi_v} \chi_v(\Delta) = -1\}}.
            \end{split}
        \end{align}
        If $\bigcup_X \mathcal{D}_{m,k,X}$ is non-empty, then
        \begin{equation*}
         \lim_{X\to \infty} \frac{\# \left\{\chi \in \mathcal{B}_{m,k}(X) : \dim_{\mathbb{F}_2} \mathrm{Sel}(T, \chi : \alpha) = n \right\}}{\# \mathcal{B}_{m,k}(X)} = \begin{cases}
             \mathbf{M}_L^k(\mathrm{E}^+_1)(n) &\text{ if k even}, \\
             \mathbf{M}_L^k(\mathrm{E}^-_1)(n) &\text{ if k odd}.
         \end{cases}
        \end{equation*}
        \item Suppose $p \neq 2$. Denote by $\mathrm{E}_1(n)$ the probability distribution
    \begin{equation}
        \mathrm{E}_1(n) := \frac{\#\{\chi \in \mathcal{C}_p(1,X) : \dim_{\mathbb{F}_p} \mathrm{Sel}(T, \chi : \alpha) = n\}}{\# \mathcal{C}_p(1,X)}.
    \end{equation}
    If $\bigcup_X \mathcal{D}_{m,k,X}$ is non-empty, then
    \begin{equation*}
     \lim_{X\to \infty} \frac{\# \left\{\chi \in \mathcal{B}_{m,k}(X) : \dim_{\mathbb{F}_p} \mathrm{Sel}(T, \chi : \alpha) = n \right\}}{\# \mathcal{B}_{m,k}(X)} = \mathbf{M}_L^k(\mathrm{E}_1)(n).
    \end{equation*}
    \end{enumerate}
\end{theorem}

We collect all fan structures with all possible widths together: $\mathcal{B}_m(X) = \bigcup_k \mathcal{B}_{m,k}(X)$. It is over such collections of characters (as $X,m \to \infty$) that the distributions on $p$-Selmer groups of twists $E$ by $\chi$ are given in \cite{KMR14}. They determine the distribution 
\begin{equation}\label{eq:KMRthm} 
    \lim_{m \to \infty}\lim_{X \to \infty} \frac{\#\{\chi \in \mathcal{B}_m(X) \mid \dim_{\mathbb{F}_p}\mathrm{Sel}(T, \chi : \alpha)=n\}}{\# \mathcal{B}_m(X)}.
\end{equation}

\subsection{A new fan structure} \label{sec:newfan}

Previous construction of the fan structure $\mathcal{B}_{m,k}(X)$ utilizes a quantification of error terms for an effective form of the Chebotarev Density Theorem. Under the assumption of the Extended Riemann Hypothesis (ERH), we have the following effective form of the Chebotarev Density Theorem. For further details, we refer to \cite[Theorem 4]{Serre81}. 
\begin{theorem}[Conditional Effective Chebotarev] \label{thm:effective_chebotarev}
    Let $F/K$ be a Galois extension of number fields with $G = \mathrm{Gal}(F/K)$, and let $C$ be a union of conjugacy classes of $G$. Let $\pi_C(X)$ be the number of prime ideals of norm at most $X$ whose Frobenius element lies in $C$. Assuming ERH, there exists an absolute constant $c > 0$ such that
    \begin{equation}\label{eq:effectiveChebotarevGRH}
        \left| \pi_C(X) - \frac{|C|}{|G|}\mathrm{Li}(X) \right| \leq c \frac{|C|}{|G|}X^{1/2}\left( \log |\mathrm{Disc}(F)|  + [F:\mathbb{Q}] \log X
        \right).
    \end{equation}
\end{theorem}
The key objective of this subsection is to replace the functions $L_i(X)$, the set $\mathcal{D}_{m,k,X}$, and the fan structure $\mathcal{B}_{m,k,X}$ with simpler new definitions. We will demonstrate in the upcoming sections that under ERH, the statistics of dimensions of Selmer structures remain identical and allow us to obtain the desired statistics over the collection of cyclic prime order extensions ordered by the norms of the radicals of their relative discriminants.

We order the set of ideals of $\mathcal{O}_K$ using the height function defined as below.
\begin{definition}
    Let $K$ be a number field. Given an ideal $\mathfrak{m} \subset \mathcal{O}_K$, we define the radical of $\mathfrak{m}$, denoted $\mathrm{Rad}(\mathfrak{m})$, to be the positive integer given by
    \begin{equation*}
        \mathrm{Rad}(\mathfrak{m}) := \prod_{\substack{\mathfrak{p} \mid \mathfrak{m} \\ \mathfrak{p} \subset \Oh_K \mathrm{ prime}}} \mathrm{Nm}_{\mathbb{Q}}^K \mathfrak{p}.
    \end{equation*}
\end{definition}
It is clear that if two ideals $\mathfrak{m}, \mathfrak{n} \subset \mathcal{O}_K$ are coprime, then $\mathrm{Rad}(\mathfrak{m} \mathfrak{n}) = \mathrm{Rad}(\mathfrak{m}) \mathrm{Rad}(\mathfrak{n})$.
We will use $\mathrm{Rad}(\mathfrak{m})$ to help organize a fan structure over any number field. 

\begin{remark}
There is a version of Malle's Conjecture for counting number fields when ordered by the product-of-ramified-primes (an example of a ``fair counting function''). The goal being to count abelian extensions $K/\mathbb{Q}$ with a fixed Galois group for which $\mathrm{Rad}(|D_K|) \leq X$. For further details see \cite{MakiDensity} for the case of abelian extensions. A correction to Malle's conjecture for fair counting functions can be found in \cite{KoymansPaganoFairCounting}. The densities we obtain can be compared to their results.
\end{remark}

\begin{definition}
    Let $K$ be a number field. Let $\mathcal{I}_K^{\square-free}$ be the subset of ideals $\mathfrak{m}$ of $\Oh_K$ such that $\mathfrak{m}$ is a square-free product of prime ideals in $\mathcal{O}_K$.
\end{definition}
If $K = \mathbb{Q}$, then $\mathcal{I}_K^{\square-free}$ is the set of square-free integers. 

We require a generalization of large deviation principles for Erd\"os-Kac theorem for number fields. That the Erd\"os-Kac theorem for number fields can be found in \cite[Theorem 1]{Li04}.  The desired large deviation principle can still be obtained from \cite{MZ16}, but at the cost of considering ideals lying in a proper subset of $\mathcal{O}_K$. Hence, we instead utilize a generalization of Turan's theorem for number fields as stated in \cite{Li04-2}.


\begin{theorem}[Example 2 and Corollary 1 of \cite{Li04-2}] \label{thm:Erdos-Kac}
Let $K$ be a number field. Given an ideal $\mathfrak{m} \subset \mathcal{O}_K$, we denote by $\omega(\mathfrak{m})$ the number of distinct prime ideals divisors of $\mathfrak{m}$. Then for every $\epsilon > 0$ and every $X > 0$,
\begin{equation}\label{eq:idealcount}
    \# \left\{ \mathfrak{m} \in \mathcal{O}_K : \mathrm{Nm}_{\mathbb{Q}}^K \mathfrak{m} \leq X, |\omega(\mathfrak{m}) - \log \log \mathrm{Nm}_{\mathbb{Q}}^K \mathfrak{m}| > \epsilon \log \log \mathrm{Nm}_{\mathbb{Q}}^K \mathfrak{m} \right\} = o(X).
\end{equation}
\end{theorem}
\begin{proof}
    This is essentially Example 2 of \cite{Li04-2} and follows almost immediately from Theorem 1 of \cite{Li04-2}. First, 
    \begin{itemize}
        \item let $P$ denote the set of primes of $K$;
        \item let $M$ denote the the free abelian monoid generated by elements of $P$, i.e. ideals of $\mathcal{O}_K$;
        \item let $N= \mathrm{Nm}_\mathbb{Q}^K: M \to \mathbb{N}\setminus\{1\}$ be the field norm.
    \end{itemize}
    Given these choices we have 
    \begin{equation}\label{eq:normgrowth}
    \sum_{\substack{N\mathfrak{m} < X \\ \mathfrak{m}\in M}}1 = \kappa X + O(X^{1-\varepsilon}) \quad \text{and} \quad 
    \sum_{\substack{N\mathfrak{p} < X \\ \mathfrak{p} \in P}}1 = O\left(\frac{x}{\log x}\right),
    \end{equation}
    where, in our setting, $\kappa = \mathrm{Res}_{s=1}\zeta_K(s) $. The first estimate above can be found in \cite[Chapter XIII]{LangANT94}; the second estimate is the Prime Ideal Theorem. Given \eqref{eq:normgrowth}, it follows from Theorem 1 of \cite{Li04-2} that 
    \begin{equation}\label{eq:Liu2Thm1}
        \sum_{N \mathfrak{m} < X}(\omega(\mathfrak{m}) - \log \log X)^2 = \kappa X \log \log X + CX + O\left(\frac{X\log \log X}{\log X}\right).
    \end{equation}
    The count in  \eqref{eq:idealcount} follows as an immediate consequence. 

    Indeed, Liu's result, \eqref{eq:Liu2Thm1}, gives the same result so long as $P$ is a countable set endowed with a norm function $N:P \rightarrow \mathbb{N}$, $M$ is the free abelian monoid generated by $P$, and the $N$ obeys growth conditions analogous to \eqref{eq:normgrowth}. 
\end{proof}

\begin{remark}
A better estimate on the error term $o(X)$ can be obtained for any number field $K$, as shown in \cite{MZ16, KP25-LD}. In particular, we obtain that the probability for any Borel set $A \subset \mathbb{R}$, the probability
    \begin{equation*}
        \mathbf{Prob}\left( \frac{\omega(\mathfrak{m})}{\log \log \mathrm{Nm}_{\mathbb{Q}}^K \mathfrak{m}} \in A \right)
    \end{equation*}
    for any uniformly randomly chosen ideal $\mathfrak{m} \subset \mathcal{O}_K$ of norm at most $X$ satisfies the large deviation principle with speed $\log \log \mathrm{Nm}_\mathbb{Q}^K \mathfrak{m}$ and rate function
    \begin{equation*}
        J(x) := \begin{cases}
            x \log x - x + 1 &\text{ if } x \geq 0, \\
            \infty &\text{ otherwise}.
        \end{cases}
    \end{equation*}
\end{remark}

We define the set of square-free ideals of bounded norm:$$I_K(X) = \left\{ \mathfrak{m} \subset \mathcal{I}_K^{\square-free} \; | \; 1 \leq \mathrm{Nm}_\mathbb{Q}^K \mathfrak{m} \leq X \right\}.$$


Our next goal is to define the set $I_K^{Fan}(X) \subset I_K(X)$. 
We make an explicit choice for $L_i$, different from the original one proposed in \cite{KMR14}, with which we define our new fan structure.
\begin{equation} \label{eq:Li-norms}
    L_i(X) := \begin{cases}
        (\log X)^A \; (= T) \; &\text{ for } 1 \leq i \leq m_s \\
        X^{\frac{1}{(2^{m_l-j})}} \; (= M_{i-m_s}) \; &\text{ for } m_s+1 \leq i \leq m-1
    \end{cases}
\end{equation}
Here, $A > 1$ is any positive real number and $m_s$ and $m_l$ will be determined in \eqref{eq:ineqK}. We also introduce notational abbreviations $T$ and $M_j$ which correspond to $\{L_i\}_{i=1}^{m-1}$'s.
We will use the functions $\{L_i\}_{i=1}^{m-1}$ to give an upper bound on the norm of all but the largest prime ideal factors of a square-free ideal, whose norm is at most $X$. We give a new construction of the collection of square-free product of prime ideals $\mathcal{D}_{m,k,X}$ as follows. If a product of $m$ distinct prime ideals, 
\begin{equation*}
    \mathfrak{d} = \prod_{i=1}^{m_s} \mathfrak{q}_{i} \prod_{j=1}^{m_l} \mathfrak{p}_j,
\end{equation*}
where $m_s + m_l = m$, is in $\mathcal{D}_{m,k,X}$, then the prime divisors $\mathfrak{q}_i$ and $\mathfrak{p}_j$ dividing $\mathfrak{d}$ satisfy:
\begin{small}
\begin{align} 
    2 \leq \; & \mathrm{Nm}_{\mathbb{Q}}^K \mathfrak{q}_i \; \leq T = L_i(X) = (\log X)^{A} \text{ for all } 1 \leq i \leq m_s, \; A > 1 \notag \\
    m_s &\leq \frac{1}{\log \log \log X} \log \log X \notag\\
    & \; \mathrm{Nm}_{\mathbb{Q}}^K\mathfrak{p}_j > T \text{ for all } 1 \leq j \leq m_l \notag \\
    T < \; &\mathrm{Nm}_\mathbb{Q}^K \mathfrak{p}_1 \leq X^{\frac{1}{2^{m_l}}} \eqcolon M_1 = L_{m_s+1}(X) \notag, \\
    \mathrm{Nm}_\mathbb{Q}^K \mathfrak{p}_1 < \; & \mathrm{Nm}_\mathbb{Q}^K \mathfrak{p}_2 \leq X^{\frac{1}{2^{m_l-1}}} \eqcolon M_2 = L_{m_s+2}(X) \notag ,\\
    \mathrm{Nm}_\mathbb{Q}^K\mathfrak{p}_2 < \; & \mathrm{Nm}_\mathbb{Q}^K\mathfrak{p}_3 \leq X^{\frac{1}{2^{m_l-2}}} \eqcolon M_3 = L_{m_s+3}(X) \notag, \\
    &\vdots   \label{eq:ineqK}\\
    \mathrm{Nm}_\mathbb{Q}^K\mathfrak{p}_{m_l-2} < \; & \mathrm{Nm}_\mathbb{Q}^K\mathfrak{p}_{m_l-1} \leq X^{\frac{1}{4}} \eqcolon M_{m_l-1} = L_{m-1}(X) \notag, \\
    \mathrm{Nm}_\mathbb{Q}^K\mathfrak{p}_{m_l-1} < \; & \mathrm{Nm}_\mathbb{Q}^K\mathfrak{p}_{m_l} \leq \frac{X}{\prod_{i=1}^{m_s} \mathrm{Nm}_\mathbb{Q}^K \mathfrak{q}_i \cdot \prod_{j=1}^{m_l-1} \mathrm{Nm}_\mathbb{Q}^K \mathfrak{p}_j}, \notag \\
    c_1 \log \log X \leq \; &m_l \leq c_2 \log \log X. \notag
\end{align}
\end{small}
where $\sum_{q_i \mid \mathfrak{d}} w(\mathfrak{q}_i) + \sum_{\mathfrak{p}_j \mid \mathfrak{d}}w(\mathfrak{p}_j)=k$, $c_1 \in (0,1),$ and $c_2 \in (1, \frac{1}{\log 2})$. 

\begin{definition}\label{def:NumberFieldIFan} Fix a number field $K$, $L_i$ as in \eqref{eq:Li-norms}, and $\mathcal{D}_{m,k,X}$ as in \eqref{eq:ineqK}. Define 
$$ I_{K}^{Fan}(X) = \bigsqcup_{m= \lceil c_1 \log \log X \rceil}^{ \lfloor \left(c_2 + \frac{1}{\log \log \log X} \right) \log \log X 
\rfloor} \left(\bigsqcup_{k=0}^{2m} \mathcal{D}_{m,k,X}\right) \subset I_K(X).$$
\end{definition}


\begin{definition}\label{def:B_K} Define the following subsets of $C_p(K)= \mathrm{Hom}(G_K, \mu_p)$:
$$\mathcal{B}_K(X) := \bigsqcup_{\mathfrak{d} \in I_{K}^{Fan}(X)} C(\mathfrak{d}),$$
$$\mathcal{C}_{K}(X) := \bigsqcup_{\mathfrak{d} \in I_K^{Fan}(X)} \left\{ \chi \in C(\mathfrak{d}) :  \mathrm{Rad}(\Delta_{L^\chi/K}) \leq X \right\}, \text{ where }$$
$$L^\chi := \text{ Fixed Field of } \mathrm{Ker}(\chi: \mathrm{Gal}(\overline{\mathbb{Q}}/K) \to \mathbb{Z}/p \mathbb{Z}).$$
\end{definition}

\section{Analytic Intermission} \label{sec:analytic-intermission}

The goal of this section is to demonstrate that the subset $I^{Fan}_K(X)$ is a density $1$ subset of $I_K(X)$ as $X$ grows arbitrarily large.
\begin{proposition} \label{prop:IfanvsI} Let $K$ be any number field. Asymptotically $100\%$ of $I_{K}(X)$ is $I_{K}^{Fan}(X)$. That is, ${\displaystyle \lim_{X \to \infty}\frac{\#I_{K}^{Fan}(X)}{\#I_{K}(X)}=1.}$
\end{proposition}
To prove this proposition, we state some results in analytic number theory. First is a generalization of Mertens' theorems to number fields; see \cite[Theorem 2, Lemma 2.4]{Rosen}.

\begin{lemma}\label{lemma:Merten}
    Let $K$ be a number field. There is a constant $B_K$ such that 
    \begin{equation}
        \sum_{N\mathfrak{p} < x} \frac{1}{N\mathfrak p} = \log \log x + B_K + O_K\left(\frac{1}{\log x}\right)
    \end{equation}
    and
    \begin{equation*}
        \prod_{N \mathfrak p < x} \left(1 - \frac{1}{N\mathfrak{p}}\right)^{-1} = e^{-\gamma}\kappa_K \log x + O_K(1),
    \end{equation*}
    where $\gamma \approx 0.577216$ is the Euler-Mascheroni constant, $\kappa_K$ denotes the residue of the Dedekind zeta function of $K$ as $s=1$, and the sum and product run over prime ideals of $\mathcal{O}_K$ with norm at most $x$.
\end{lemma}
Totally explicit estimates for $B_K$ and the implied constants are given in work of Garcia and Lee, though we will not use them here \cite{GarciaLee22}.

Next, we need a version of a count of integers with exactly $k$ prime factors (e.g. \cite{Sathe53, Selberg54}) but for ideals in the ring of integers in a number field. We state a version due to Wu \cite[Theorem 3]{Wu96}.

\begin{lemma}\label{lemma:SatheSelberg}
    Let $K$ be a number field and let $\pi_k(X,K)$ be the number of ideals of $\mathcal{O}_K$ with norm at most $X$ whose number of prime factors is exactly $k$. Then, 
    \begin{equation*}
        \pi_k(X,K) = \frac{X}{\log X} \frac{(\log \log X)^{k-1}}{(k-1)!}\left(G(X,K)+O_K\left(\frac{k}{(\log \log X)^2}\right)\right)
    \end{equation*}
    where 
    $$G(X,K) = \frac{(\kappa_K)^{\frac{k}{\log \log x}}}{\Gamma\left({\frac{k}{\log \log x}}+1\right)}\prod_{\mathfrak{p}}\left(1+\frac{k}{(N\mathfrak p -1)\log\log x}\right)^{-1}\left(1-\frac{1}{N\mathfrak{p}}\right)^{\frac{k}{\log \log x}},$$
    $\kappa_K$ is the residue of the Dedekind zeta function of $K$ at $s=1$ and the product is over all primes of $\mathcal{O}_K$.
\end{lemma}


We also need a version of a count of ideals of bounded norm all of whose prime ideal factors are of norm at least some parameter $T$. We state a corollary of this result due to Debaene \cite{Debaene19}.
\begin{theorem} \label{thm:Korneel}
    Let $K$ be a number field and let $\mathcal{S}(X,T)$ be the number of ideals of $\mathcal{O}_K$ with norm at most $X$ all of whose prime ideal factors are of norm strictly greater than $T$. Then there exists an explicitly computable constant $C_K$ depending only on $K$ such that for every $T \geq C_K$,
    \begin{equation*}
        \mathcal{S}(X,T) \ll_K \frac{X}{\log T} + X^{1 - \frac{1}{[K:\mathbb{Q}]}} \cdot T^{\frac{2}{[K : \mathbb{Q}]}} \cdot (2 \log T)^{3[K:\mathbb{Q}]}.
    \end{equation*}
\end{theorem}
\begin{proof}
    The proof follows from applying \cite[Theorem 14, Lemma 15, Lemma 16]{Debaene19}.
\end{proof}

Finally, we need a lemma due to Norton for the tail of the exponential function; see \cite[Lemma 4.7]{Norton76}.
\begin{lemma} 
\label{lemma:Norton}
    For $x > 0$ and $\beta > 1$, one has
    \begin{equation*}
        \sum_{n \geq \beta x} \frac{x^n}{n!} \leq \frac{1}{\beta-1} \cdot \left( \frac{\beta}{2 \pi x} \right)^{\frac{1}{2}} \cdot e^{\beta x (1 - \log \beta)}.
    \end{equation*}
\end{lemma}

We first impose some conditions on prime ideal factorization of an ideal $\mathfrak{m} \subset \mathcal{O}_K$.
\begin{definition}
    Let $K$ be a number field. Choose some integers $b_1, b_2$ and a real number $T > 0$. We say that a square-free ideal $\mathfrak{m} \subset \mathcal{O}_K$ satisfies condition $C(b_1, b_2, T)$ if it satisfies the following conditions:
    \begin{itemize}
        \item There exists a factorization of $\mathfrak{m}$ for which $\mathfrak{m} = \mathfrak{n} \cdot \mathfrak{d}$, where all prime ideal factors of $\mathfrak{n}$ have norms strictly greater than $T$, and all prime ideal factors of $\mathfrak{d}$ have norms at most $T$. 
        \item The number of distinct prime ideal factors of $\mathfrak{d}$ is between $b_1$ and $b_2$.
    \end{itemize}
\end{definition}
We first prove the following statement. The proof of this lemma is inspired from the proof of \cite[Lemma 2.10]{PS25}.
\begin{lemma} \label{lemma:smalldivisors}
    Let $K$ be a number field. Given a real number $X > 0$, suppose $b_1, b_2, T$ satisfy the following conditions.
    \begin{itemize}
        \item There exists some $A > 1$ such that $T = (\log X)^A$.
        \item We have $b_1 = 0$.
        \item There exists some $\epsilon \in (0, 1)$ such that $b_2 = \lceil \epsilon \log \log X \rceil$.
    \end{itemize}
    Fix a number $c_2 \in (1, 1/\log2)$. Then for sufficiently large $X$ and any $b > 1$, 
    \begin{align}
    \begin{split}
        & \hspace{15pt} \frac{\# \left\{ \mathfrak{m} \subset \mathcal{O}_K : \mathrm{Nm}_{\mathbb{Q}}^K \mathfrak{m} \leq X,  \omega(\mathfrak{m}) \leq c_2 \log \log X, \; \mathfrak{m} \text{ satisfies } C(0, b_2, T) \right\}}{\#\{\mathfrak{m} \subset \mathcal{O}_K : \mathrm{Nm}_{\mathbb{Q}}^K \mathfrak{m} \leq X,  \omega(\mathfrak{m}) \leq c_2 \log \log X\}} \\
        &= 1 - O \left( \frac{1}{(\log \log X)^b} \right),
    \end{split}
    \end{align}
    where the implied constant depends on $\alpha$ and $b_2$. 
    
\end{lemma}
\begin{proof}
We first observe that
    \begin{align*}
        & \# \left\{ \mathfrak{m} \subset \mathcal{O}_K : \mathrm{Nm}_{\mathbb{Q}}^K(\mathfrak{m}) \leq X,  \omega(\mathfrak{m}) \leq c_2 \log \log X, \; \mathfrak{m} \text{ does not satisfy } C(0, b_2, T) \right\} \\
        &\ll \sum_{t = b_2+1}^{c_2 \log \log X} \sum_{\substack{ \mathrm{Nm}_{\mathbb{Q}}^K \mathfrak{m} \leq X \\ \mathfrak{m} = \mathfrak{n} \cdot \mathfrak{d} \\ \mathfrak{p} \mid \mathfrak{d} \implies \mathrm{Nm}_{\mathbb{Q}}^K \mathfrak{p} \leq T \\ \mathfrak{p} \mid \mathfrak{n} \implies \mathrm{Nm}_{\mathbb{Q}}^K \mathfrak{p} > T \\ \omega(\mathfrak{d}) = t}} 1 \\
        &\ll \sum_{t = b_2+1}^{c_2 \log \log X} \sum_{\substack{ \mathfrak{d} \subset \mathcal{O}_K \\ \mathfrak{p} \mid \mathfrak{d} \implies \mathrm{Nm}_{\mathbb{Q}}^K \mathfrak{p} \leq T \\ \omega(\mathfrak{d}) = t}} \# \left\{ \mathfrak{m} \subset \mathcal{O}_K: \mathrm{Nm}_\mathbb{Q}^K \mathfrak{m} \leq \frac{X}{\mathrm{Nm}_\mathbb{Q}^K \mathfrak{d}}, \mathfrak{p} \mid \mathfrak{m} \implies \mathrm{Nm}_\mathbb{Q}^K \mathfrak{p} > T \right\}
    \end{align*}
    Note that because $T = (\log X)^A$ and every $\mathfrak{d}$ appearing in the inner summand satisfies (for arbitrarily large $X$)
    \begin{equation*}
        \mathrm{Nm}_\mathbb{Q}^K \mathfrak{d} \leq T^{c_2 \log \log X} = (\log X)^{A c_2 \log \log X} \ll X^{\frac{1}{4[K: \mathbb{Q}]}},
    \end{equation*}
    we have $T \ll \left( \frac{X}{\mathrm{Nm}_\mathbb{Q}^K \mathfrak{d}} \right)^{\frac{1}{[K:\mathbb{Q}]}}$ for every $\mathfrak{d}$ appearing in the inner summand.
    By using Theorem \ref{thm:Korneel} and $T = (\log X)^A \ll X^{\frac{1}{3[K:\mathbb{Q}]}}$, we obtain:
    \begin{align*}
        &\ll_K \sum_{t=b_2+1}^{c_2 \log \log X} \sum_{\substack{\mathfrak{d} \subset \mathcal{O}_K \\ \mathfrak{p} \mid \mathfrak{d} \implies \mathrm{Nm}_{\mathbb{Q}}^K \mathfrak{p} \leq T \\ \omega(\mathfrak{d}) = t}} \frac{X}{\mathrm{Nm}_{\mathbb{Q}}^K \mathfrak{d}} \frac{1}{\log T} \\
        &\ll_K X \cdot \sum_{t = b_2+1}^{c_2 \log \log X} \frac{1}{\log T} \sum_{\substack{\mathrm{Nm}_\mathbb{Q}^K \mathfrak{d} \leq T^t \\ \omega(\mathfrak{d}) = t \\ \mathfrak{p} \mid \mathfrak{d} \implies \mathrm{Nm}_{\mathbb{Q}}^K \mathfrak{p} \leq T}} \frac{1}{\mathrm{Nm}_{\mathbb{Q}}^K \mathfrak{d}} \\
        &\ll_K X \cdot \sum_{t = b_2+1}^{c_2 \log \log X} \frac{1}{\log T} \sum_{\substack{\mathrm{Nm}_{\mathbb{Q}}^K \mathfrak{d} \leq T^t \\ \omega(\mathfrak{d}) = t}} \frac{1}{\mathrm{Nm}_{\mathbb{Q}}^K \mathfrak{d}}.
    \end{align*}
    By proof of \cite[Lemma 2.10, page 7, lines 8--10]{PS25} (in particular Lemma \ref{lemma:SatheSelberg} and partial summation), we have
    \begin{align*}
        &\ll_K X \cdot \sum_{t = b_2+1}^{c_2 \log \log X} \frac{1}{\log T} \cdot \frac{(\log \log T^t)^t}{(t-1)!} \\
        &=_K \frac{X}{\log T} \sum_{t=b_2+1}^{c_2 \log \log X} \frac{(\log \log T + \log t)^t}{(t-1)!},
    \end{align*}
    where the implied constant depends solely on the choice of the base number field $K$.
    Because we set $T$ to be some positive power of $\log X$, and $t \leq 2 \log \log X$ by condition on $\omega(\mathfrak{m})$, we can apply Lemma \ref{lemma:Norton} by picking any $\beta > 3$ to obtain
    \begin{align*}
        &\ll_K \frac{X}{\log T} \sum_{t=\epsilon \log \log X +1}^{c_2 \log \log X} \frac{(\log \log T + \log t)^t}{(t-1)!} \\
        &\ll_K \frac{X}{\log T} \sum_{t= \beta (2 \log \log \log X + O(1))}^{c_2 \log \log X} \frac{(2 \log \log \log X + O(1))^t}{(t-1)!} \\
        &\ll_K \frac{X}{\log T} \cdot \frac{1}{\sqrt{\log \log \log X}} \cdot e^{2\beta(1-\log \beta) \log \log \log X} \\
        &\ll_K \frac{X}{\log T} \cdot \frac{(\log \log X)^{2\beta(1-\log \beta)}}{\sqrt{\log \log \log X}}.
    \end{align*}
    Note that $2\beta (1 - \log \beta) < 0$ because we chose $\beta > 3$. We again use the fact that $T$ is any positive power of $\log X$ to conclude the claim.
\end{proof}

\begin{remark}
    One can in fact further improve the error term by setting $T = X^{\frac{1}{(\log X)^{\delta}}}$ for some $\delta \in (0,1)$ which is arbitrarily close to $1$. Then notice that $\log T = (\log X)^{1 - \delta}$, and $\log \log T = (1-\delta) \log \log X$. Furthermore, we also have $T^{c_2 \log \log X} = X^{\frac{c_2 \log \log X}{(\log X)^\delta}} \ll X^{\frac{1}{4[K:\mathbb{Q}]}}$, so we can still apply Theorem \ref{thm:Korneel} to prove an analogue of Lemma \ref{lemma:smalldivisors} for our choice of $T$. We may set $\epsilon = 3(1-\delta) \ll 1$, from which one can deduce that the cardinality of the integers not satisfying $C(0,b_2,T)$ is at most
    \begin{align*}
        &\ll_K \frac{X}{(\log X)^{1-\delta}} \cdot \frac{1}{\sqrt{(1-\delta) \log \log X}} \cdot (\log X)^{3(1-\log 3)(1-\delta)} \\
        &\ll_K \frac{X}{(\log X)^{(3 \log 3 - 2)(1-\delta)}}.
    \end{align*}
    This implies that the density of the set can be improved by an error term of order $O(1/(\log X)^{(3 \log 3 - 2)(1 - \delta)})$.
\end{remark}

We then show the following lemma on the density of set of integers, all prime factors of which are greater than some positive fractional power of the upper bound of the integers.
\begin{lemma} \label{lemma:largeprime}
    Let $K$ be any number field. Suppose $1 \leq k \leq C \log \log X$ for any fixed $C > 0$. Fix a positive real number $a \in (0,\frac{1}{k})$. Then as $X$ grows arbitrarily large,
    \begin{align}
    \begin{split}
        & \hspace{15pt} \frac{\# \left\{ \mathfrak{m} \subset \mathcal{O}_K : \mathrm{Nm}_\mathbb{Q}^K \mathfrak{m} \leq X, \omega(\mathfrak{m}) = k, \exists \text{ prime ideal } \mathfrak{p} \mid \mathfrak{m} \text{ s.t. } \mathrm{Nm}_{\mathbb{Q}}^K \mathfrak{p} \leq X^a \right\}}{\# \left\{ \mathfrak{m} \subset \mathcal{O}_K : \mathrm{Nm}_\mathbb{Q}^K \mathfrak{m} \leq X, \omega(\mathfrak{m}) = k \right\}} \\
        &= 1 - O_K \left( \frac{1}{1-ak} \cdot \left(\frac{-\log a}{\log \log X} \right)^{k-1} \right),
    \end{split}
    \end{align}
    where the implied constant depends only on the choice of the base field $K$.
\end{lemma}
\begin{proof}
We have:
    \begin{align*}
        & \# \left\{ \mathfrak{m} \subset \mathcal{O}_K : \mathrm{Nm}_\mathbb{Q}^K \mathfrak{m} \leq X, \omega(\mathfrak{m}) = k, \mathfrak{p} \mid \mathfrak{m} \implies \mathrm{Nm}_{\mathbb{Q}}^K \mathfrak{p} \geq X^a \right\} \\
        &\ll_K \frac{1}{(k-1)!} \sum_{\substack{\mathfrak{p}_1, \mathfrak{p}_2, \cdots, \mathfrak{p}_{k-1} \\ X^a \leq \mathrm{Nm}_{\mathbb{Q}}^K \mathfrak{p}_1, \cdots, \mathrm{Nm}_{\mathbb{Q}}^K \mathfrak{p}_{k-1} \leq X}} \# \left\{ \mathrm{Nm}_{\mathbb{Q}}^K \mathfrak{p}_k \leq \frac{X}{\mathrm{Nm}_{\mathbb{Q}}^K \mathfrak{p}_1 \cdots \mathrm{Nm}_\mathbb{Q}^K \mathfrak{p}_{k-1}}: \mathfrak{p}_k \text{ is a prime ideal} \right\} \\
        &\ll_K \frac{1}{(k-1)!} \sum_{\substack{\mathfrak{p}_1, \mathfrak{p}_2, \cdots, \mathfrak{p}_{k-1} \\ X^a \leq \mathrm{Nm}_{\mathbb{Q}}^K \mathfrak{p}_1, \cdots, \mathrm{Nm}_{\mathbb{Q}}^K \mathfrak{p}_{k-1} \leq X}} \frac{X}{\prod_{j=1}^{k-1} \mathrm{Nm}_{\mathbb{Q}}^K \mathfrak{p}_j} \cdot \frac{1}{\log X - \log \mathrm{Nm}_\mathbb{Q}^K \mathfrak{p}_1 - \cdots - \log \mathrm{Nm}_\mathbb{Q}^K \mathfrak{p}_{k-1}} \\
        &\ll_K \frac{1}{(k-1)!} \sum_{\substack{\mathfrak{p}_1, \mathfrak{p}_2, \cdots, \mathfrak{p}_{k-1} \\ X^a \leq \mathrm{Nm}_{\mathbb{Q}}^K \mathfrak{p}_1, \cdots, \mathrm{Nm}_{\mathbb{Q}}^K \mathfrak{p}_{k-1} \leq X}} \frac{X}{\prod_{j=1}^{k-1} \mathrm{Nm}_{\mathbb{Q}}^K \mathfrak{p}_j} \cdot \frac{1}{(1-ak)\log X} \\
        &\ll_K \frac{X}{(1-ak)\log X} \frac{1}{(k-1)!} \cdot \prod_{j=1}^{k-1} \left( \sum_{X^a \leq \mathrm{Nm}_\mathbb{Q}^K \mathfrak{p}_j \leq X} \frac{1}{\mathrm{Nm}_\mathbb{Q}^K \mathfrak{p}_j} \right).
    \end{align*}
    By Lemma \ref{lemma:Merten}, we have
    \begin{align*}
        &\ll_K \frac{X}{(1-ak)\log X} \frac{1}{(k-1)!} \cdot (- \log a + O(1/\log X))^{k-1}.
    \end{align*}
    We compare the above asymptotic counting function with Lemma \ref{lemma:SatheSelberg} to obtain the desired statement.
\end{proof}

We now prove Proposition \ref{prop:IfanvsI}. 
\begin{proof}
We recall the definition of $I_K^{Fan}(X)$. For any ideal $\mathfrak{m} \in I_K^{Fan}(X)$ we have the prime ideal factorization
\begin{equation*}
    \mathfrak{d} = \prod_{i=1}^{m_s} \mathfrak{q}_{i} \prod_{j=1}^{m_l} \mathfrak{p}_j,
\end{equation*}
where $m_s + m_l = m$, is in $\mathcal{D}_{m,k,X}$, and the prime divisors $\mathfrak{q}_i$ and $\mathfrak{p}_j$ dividing $\mathfrak{d}$ satisfy:
\begin{small}
\begin{align} 
    2 \leq \; & \mathrm{Nm}_{\mathbb{Q}}^K \mathfrak{q}_i \; \leq T := (\log X)^{A} \text{ for all } 1 \leq i \leq m_s, \; A > 1 \notag \\
    m_s &\leq \frac{1}{\log \log \log X} \log \log X \notag\\
    & \; \mathrm{Nm}_{\mathbb{Q}}^K\mathfrak{p}_j > T \text{ for all } 1 \leq j \leq m_l \notag \\
    T < \; &\mathrm{Nm}_\mathbb{Q}^K \mathfrak{p}_1 \leq X^{\frac{1}{2^{m_l}}} \eqcolon M_1 \notag, \\
    \mathrm{Nm}_\mathbb{Q}^K \mathfrak{p}_1 < \; & \mathrm{Nm}_\mathbb{Q}^K \mathfrak{p}_2 \leq X^{\frac{1}{2^{m_l-1}}} \eqcolon M_2 \notag ,\\
    \mathrm{Nm}_\mathbb{Q}^K\mathfrak{p}_2 < \; & \mathrm{Nm}_\mathbb{Q}^K\mathfrak{p}_3 \leq X^{\frac{1}{2^{m_l-2}}} \eqcolon M_3 \notag, \\
    &\vdots\\
    \mathrm{Nm}_\mathbb{Q}^K\mathfrak{p}_{m_l-2} < \; & \mathrm{Nm}_\mathbb{Q}^K\mathfrak{p}_{m_l-1} \leq X^{\frac{1}{4}} \eqcolon M_{m_l-1} \notag, \\
    \mathrm{Nm}_\mathbb{Q}^K\mathfrak{p}_{m_l-1} < \; & \mathrm{Nm}_\mathbb{Q}^K\mathfrak{p}_{m_l} \leq \frac{X}{\prod_{i=1}^{m_s} \mathrm{Nm}_\mathbb{Q}^K \mathfrak{q}_i \cdot \prod_{j=1}^{m_l-1} \mathrm{Nm}_\mathbb{Q}^K \mathfrak{p}_j}, \notag \\
    c_1 \log \log X \leq \; &m_l \leq c_2 \log \log X. \notag
\end{align}
\end{small}
where $\sum_{q_i \mid \mathfrak{d}} w(\mathfrak{q}_i) + \sum_{\mathfrak{p}_j \mid \mathfrak{d}}w(\mathfrak{p}_j)=k$, $c_1 \in (0,1),$ and $c_2 \in (1, \frac{1}{\log 2})$.

We first use Lemma \ref{lemma:smalldivisors} to show that for any $b > 0$, all but $O(X/(\log \log X)^b)$ of integers at most $X$ satisfy the following two conditions:
\begin{align*}
    2 \leq & \mathrm{Nm}_\mathbb{Q}^K \mathfrak{q}_i \leq T := (\log X)^{A} \notag \\
    m_s &\leq \frac{1}{\log \log \log X} \log \log X \notag
\end{align*}
We use the abbreviation $\epsilon := \frac{1}{\log \log \log X}$. For each choice of $\mathfrak{S} \subset \mathcal{O}_K$ such that $\mathrm{Nm}_\mathbb{Q}^K \mathfrak{S} \in [1, (\log X)^{A \epsilon \log \log X}]$, consider the decomposition of the set of integers as follows.
\begin{align*}
    & \left\{ \mathfrak{m} \subset \mathcal{O}_K : \mathrm{Nm}_\mathbb{Q}^K \mathfrak{m} \leq X, \; c_1 \log \log X \leq \omega(\mathfrak{m}) \leq \left( c_2 + \epsilon \right) \log \log X \right\} \\
    &= \bigsqcup_{m_l=c_1 \log \log X}^{c_2 \log \log X} \bigsqcup_{m_s = 0}^{\epsilon \log \log X} \left\{ \mathfrak{m} \subset \mathcal{O}_K : \mathrm{Nm}_\mathbb{Q}^K \mathfrak{m} \leq X, \omega (\mathfrak{m}/\mathfrak{S}) = m_l, \omega(\mathfrak{S}) = m_s \right\}
\end{align*}
For each $\mathfrak{S}$, we use Lemma \ref{lemma:largeprime} and the fact that $m_l < c_2 \log \log X$ for $c_2 \in (1, 1/\log 2)$ to the set $$\left\{ \mathfrak{m} \subset \mathcal{O}_K : \mathrm{Nm}_\mathbb{Q}^K \mathfrak{m} \leq X, \omega (\mathfrak{m}/\mathfrak{S}) = m_l, \omega(\mathfrak{S}) = m_s \right\}$$ to show that all but a density $O((\log X)^{\log(c_2 \log 2)})$ subset satisfies
\begin{align*}
    T < \mathrm{Nm}_\mathbb{Q}^K \mathfrak{p}_1 < X^{\frac{1}{2^{m_l}}},
\end{align*}
where $\mathfrak{p}_1$ is the prime ideal factor of smallest norm of $\mathfrak{m}/\mathfrak{S}$. We then use Lemma \ref{lemma:largeprime} inductively on $\mathfrak{p}_j$'s for $1 \leq j \leq k-1$ to show that given a fixed choice of $S$ and fixed choices of $\mathfrak{p}_1, \mathfrak{p}_2, \cdots, \mathfrak{p}_{j-1}$, all but a density at most $O(1/\log \log X)$ subset satisfies
\begin{align*}
    \mathrm{Nm}_\mathbb{Q}^K \mathfrak{p}_{j-1} < \mathrm{Nm}_\mathbb{Q}^K \mathfrak{p}_j < X^{\frac{1}{2^{m_l-j+1}}}
\end{align*}
where $\mathfrak{p}_j$ is the smallest prime factor of $\mathfrak{m}/(\mathfrak{S} \cdot \mathfrak{p}_1 \cdot \mathfrak{p}_2 \cdots \mathfrak{p}_{j-1})$. We note that $\mathrm{Nm}_\mathbb{Q}^K \mathfrak{p}_{j-1} < \mathrm{Nm}_\mathbb{Q}^K \mathfrak{p}_j$ by the induction hypothesis and the fact that all the prime ideal factors of $\mathfrak{m}/\mathfrak{S}$ have norms of at least $(\log X)^A$: namely, that in the previous induction step we chose $\mathfrak{p}_{j-1}$ to correspond to the prime ideal factor of smallest norm of $m/(\mathfrak{S} \cdot \mathfrak{p}_1 \cdots \mathfrak{p}_{j-2})$. Combining all the induction steps together, we obtain that all but $O(X/\log \log X)$ of ideals of norm at most $X$ satisfy the desired conditions on the prime ideal divisors.
\end{proof}

Provided below are two additional fan structures constructed in a similar manner to $\mathcal{D}_{m,k,X}$, whose utility will be elucidated in later sections.

\begin{remark} \label{remark:better_fan}
    One can choose other values of $T$ to construct a relaxed version of the new fan structure $\mathcal{D}_{m,k,X}$ which encompasses all but a density $O \left((\log X)^{\frac{c_2 \log 2 + 1}{2} - 1} + (\log X)^{\epsilon' \log \log 2 + \delta} \right)$ subset of $I_K^{\square-free}(X)$ for some small enough $\delta > 0$. 
    
    We let $\epsilon' \in (0, c_1/2)$ be an auxiliary parameter which will be determined later. Suppose only the norms of the prime ideal factors $\mathfrak{q}_i$ for $1 \leq i \leq m_s$ and $\mathfrak{p}_j$ for $1 \leq k \leq m_l - \epsilon' \log \log X$ of $\mathfrak{d} \in \mathcal{D}_{m,k,X}$ satisfy the analogous bounds as stated in \eqref{eq:ineqK}. In particular, we relax the definition of $\mathcal{D}_{m,k,X}$ such that any $\mathfrak{d} \in \mathcal{D}_{m,k,X}$ satisfies
    \begin{small}
    \begin{align*} 
    2 \leq \; & \mathrm{Nm}_{\mathbb{Q}}^K \mathfrak{q}_i \; \leq T := X^{\frac{1}{(\log X)^{\frac{c_2 \log 2 + 1}{2}}}} \text{ for all } 1 \leq i \leq m_s \notag \\
    m_s &\leq \frac{1}{\log \log \log X} \log \log X \notag\\
    & \; \mathrm{Nm}_{\mathbb{Q}}^K\mathfrak{p}_j > T \text{ for all } 1 \leq j \leq m_l \notag \\
    T < \; &\mathrm{Nm}_\mathbb{Q}^K \mathfrak{p}_1 \leq X^{\frac{1}{2^{m_l}}} \eqcolon M_1 \notag, \\
    \mathrm{Nm}_\mathbb{Q}^K \mathfrak{p}_1 < \; & \mathrm{Nm}_\mathbb{Q}^K \mathfrak{p}_2 \leq X^{\frac{1}{2^{m_l-1}}} \eqcolon M_2 \notag ,\\
    &\vdots\\
    \mathrm{Nm}_\mathbb{Q}^K\mathfrak{p}_{m_l- \lfloor \epsilon' \log \log X \rfloor - 2} < \; & \mathrm{Nm}_\mathbb{Q}^K\mathfrak{p}_{m_l- \lfloor \epsilon' \log \log X \rfloor - 1} \leq X^{\frac{1}{2^{\lfloor \epsilon' \log \log X \rfloor}}} \eqcolon M_{m_l- \lfloor \epsilon' \log \log X \rfloor - 1} \notag, \\
    \mathrm{Nm}_\mathbb{Q}^K\mathfrak{p}_{m_l- \lfloor \epsilon' \log \log X \rfloor - 1} < \; & \mathrm{Nm}_\mathbb{Q}^K\mathfrak{p}_{m_l- \lfloor \epsilon' \log \log X \rfloor } \notag \\
    &\vdots \\
     \mathrm{Nm}_\mathbb{Q}^K\mathfrak{p}_{m_l-2} < \; & \mathrm{Nm}_\mathbb{Q}^K\mathfrak{p}_{m_l-1} \notag \\
    \mathrm{Nm}_\mathbb{Q}^K\mathfrak{p}_{m_l-1} < \; & \mathrm{Nm}_\mathbb{Q}^K\mathfrak{p}_{m_l} \leq \frac{X}{\prod_{i=1}^{m_s} \mathrm{Nm}_\mathbb{Q}^K \mathfrak{q}_i \cdot \prod_{j=1}^{m_l-1} \mathrm{Nm}_\mathbb{Q}^K \mathfrak{p}_j}, \notag \\
    c_1 \log \log X \leq \; &m_l \leq c_2 \log \log X. \notag
    \end{align*}
    \end{small}
    Then all but a density $O \left((\log X)^{\frac{c_2 \log 2 + 1}{2} - 1} + (\log X)^{\epsilon' \log \log 2 + \delta} \right)$ subset of $I_K^{\square-free}(X)$ for some small enough $\delta > 0$ lies in $\mathcal{D}_{m,k,X}$ for some $m$ and $k$. The exponent $\frac{c_2 \log 2 + 1}{2} - 1$ originates from applying Lemma \ref{lemma:smalldivisors} with the new choice of $T$, and the exponent $\epsilon' \log \log 2$ originates from using Lemma \ref{lemma:largeprime} up to $m_l - \lfloor \epsilon' \log \log X \rfloor$ many times. 
\end{remark}

\begin{remark} \label{remark:parametrize_fan}
    One can choose other values of $M_1, M_2, \cdots, M_{m_l-1}$ to construct a collection of new fan structures $\{\mathcal{D}_{m,k,X,a}\}_{a}$ indexed by a choice of a positive real number $a \in (\frac{1}{(\log \log X)^A}, 1/2]$ given any positive integer $A \geq 1$.

    As before, consider the prime ideal factorization
\begin{equation*}
    \mathfrak{d} = \prod_{i=1}^{m_s} \mathfrak{q}_{i} \prod_{j=1}^{m_l} \mathfrak{p}_j,
\end{equation*}
where $m_s + m_l = m$. We say that $\mathfrak{d} \in \mathcal{D}_{m,k,X,a}$ if the prime divisors $\mathfrak{q}_i$ and $\mathfrak{p}_j$ dividing $\mathfrak{d}$ satisfy:
\begin{small}
\begin{align} 
    2 \leq \; & \mathrm{Nm}_{\mathbb{Q}}^K \mathfrak{q}_i \; \leq T := (\log X)^{A} \text{ for all } 1 \leq i \leq m_s, \; A > 1 \notag \\
    m_s &\leq \frac{1}{\log \log \log X} \log \log X \notag\\
    & \; \mathrm{Nm}_{\mathbb{Q}}^K\mathfrak{p}_j > T \text{ for all } 1 \leq j \leq m_l \notag \\
    T < \; &\mathrm{Nm}_\mathbb{Q}^K \mathfrak{p}_1 \leq X^{a \cdot \frac{1}{2^{m_l-1}}} \eqcolon M_1(a) \notag, \\
    \mathrm{Nm}_\mathbb{Q}^K \mathfrak{p}_1 < \; & \mathrm{Nm}_\mathbb{Q}^K \mathfrak{p}_2 \leq X^{a \cdot \frac{1}{2^{m_l-2}}} \eqcolon M_2(a) \notag ,\\
    \mathrm{Nm}_\mathbb{Q}^K\mathfrak{p}_2 < \; & \mathrm{Nm}_\mathbb{Q}^K\mathfrak{p}_3 \leq X^{a \cdot \frac{1}{2^{m_l-3}}} \eqcolon M_3(a) \notag, \\
    &\vdots\\
    \mathrm{Nm}_\mathbb{Q}^K\mathfrak{p}_{m_l-2} < \; & \mathrm{Nm}_\mathbb{Q}^K\mathfrak{p}_{m_l-1} \leq X^{a \cdot \frac{1}{2}} \eqcolon M_{m_l-1}(a) \notag, \\
    \mathrm{Nm}_\mathbb{Q}^K\mathfrak{p}_{m_l-1} < \; & \mathrm{Nm}_\mathbb{Q}^K\mathfrak{p}_{m_l} \leq \frac{X}{\prod_{i=1}^{m_s} \mathrm{Nm}_\mathbb{Q}^K \mathfrak{q}_i \cdot \prod_{j=1}^{m_l-1} \mathrm{Nm}_\mathbb{Q}^K \mathfrak{p}_j}, \notag \\
    c_1 \log \log X \leq \; &m_l \leq c_2 \log \log X. \notag
\end{align}
\end{small}
where $\sum_{q_i \mid \mathfrak{d}} w(\mathfrak{q}_i) + \sum_{\mathfrak{p}_j \mid \mathfrak{d}}w(\mathfrak{p}_j)=k$, $c_1 \in (0,1),$ and $c_2 \in (1, \frac{1}{\log 2})$.

By adapting the proof of Proposition \ref{prop:IfanvsI} line by line, it follows that given any real positive number $a \in (1/(\log\log X)^A, 1/2]$ for some positive $A \geq 1$, all but a density $O(-\log a / \log \log X)$ subset of $I_K^{\square-free}(X)$ lies in $\mathcal{D}_{m,k,X,a}$. We note that the requirement that $a \geq 1/(\log \log X)^A$ ensures that the error term $-\log a / \log \log X$ does not become trivial. In addition, the case where $a = 1/2$ recovers our original construction of the fan structure $\mathcal{D}_{m,k,X}$.

By considering disjoint collections of subset of fan structures $\mathcal{D}_{m,k,X,a}$, we can impose additional conditions on the prime ideal whose norm is the largest among all prime ideal factors. Let $\overline{\mathcal{D}}_{m,k,X,a}$ be a subset of the fan structure $\mathcal{D}_{m,k,X,a}$ given by requiring the following condition on $\mathrm{Nm}_\mathbb{Q}^K\mathfrak{p}_{m_l}$:
\begin{equation}
    \frac{X^{1 - 2a(1 - 2^{-m_l})}}{\prod_{i=1}^{m_s} \mathrm{Nm}_\mathbb{Q}^K \mathfrak{q}_i}< \mathrm{Nm}_\mathbb{Q}^K\mathfrak{p}_{m_l} \leq \frac{X^{1 - a(1-2^{-m_l})}}{\prod_{i=1}^{m_s} \mathrm{Nm}_\mathbb{Q}^K \mathfrak{q}_i}.
\end{equation}
We now choose $a \in \left\{ 2^{-i} \right\}_{i=1}^{\lfloor \log \log \log X \rfloor}$. Then for any positive integers $i \neq j$, we have $$\overline{\mathcal{D}}_{m,k,X,2^{-i}} \cap \overline{\mathcal{D}}_{m,k,X,2^{-j}} = \emptyset.$$ By applying Lemma \ref{lemma:Merten}, which allows us to compute density of ideals of $\mathcal{O}_K$ of ideal norm at most $X$ whose prime ideal factor of largest norm is at most $X^a$, we can conclude that the disjoint union 
\begin{equation}
    I_K^{Fan}(X) := \bigsqcup_{m = \lceil c_1 \log \log X \rceil}^{\lfloor \left( c_2 + \frac{1}{\log \log \log X} \right) \log \log X \rfloor} \left( \bigsqcup_{k=0}^{2m} \bigsqcup_{i=1}^{\lfloor \log \log \log X \rfloor} \overline{\mathcal{D}}_{m,k,X,2^{-i}} \right).
\end{equation}
contains all but a density $O(\log \log \log X / \log \log X)$ subset of ideals in $I_K^{\square-free}(X)$.

Although the motivation for constructing such a fan seems unmotivated at the moment, we will use the disjoint union of fans $\bigsqcup_{k=1}^{\lfloor \log \log \log X \rfloor} \overline{\mathcal{D}}_{m,k,X,2^{-k}}$ to compute the distribution of 2-Selmer groups of quadratic twist families of Jacobians of hyperelliptic curves, which may not be feasibly obtained from using the original fan structure appearing in Proposition \ref{prop:IfanvsI}.
\end{remark}

One of the other tools from probability theory we will use is the geometric ergodicity theorem for irreducible and aperiodic Markov operators over countable state spaces. Provided below is an adaptation of \cite[Theorem 15.0.1]{MT93} and \cite[Theorem 6.6]{park2022prime} to weighted sums of Markov operators $\mathbf{M}_L$.
\begin{theorem}[Theorem 15.0.1 of \cite{MT93}, Theorem 6.6 of \cite{park2022prime}] \label{thm:geometric-ergodicity}
    Given a probability distribution $\mu: \mathbb{Z}_{\geq 0} \to [0,1]$ such that there exists some $N > 0$ such that $\mu(n) = 0$ for all $n \geq N$, we denote by $\mathbb{E}[p^\mu] := \sum_{n=0}^\infty \mu(n) \cdot p^n$. We also denote by $\mathbf{Id}$ the identity Markov operator.
    \begin{enumerate}
        \item Suppose $k$ is an odd integer. Then there exist explicitly computable constant $B > 0$ and $\gamma \in (0,1)$ such that
        \begin{equation*}
            \sup_{n \in \mathbb{Z}_{\geq 0}} \left| \mathbf{M}_L^{k}\mu(n) - \left(\rho(\mu)E^+ + (1-\rho(\mu))E^-\right)(n) \right| < B \cdot \mathbb{E}[p^\mu] \cdot \gamma^k.
        \end{equation*}
        \item Suppose $k$ is an even integer. Then there exists explicitly computable constant $B > 0$ and $\gamma \in (0,1)$ such that
        \begin{equation*}
            \sup_{n \in \mathbb{Z}_{\geq 0}} \left| \mathbf{M}_L^{k}\mu(n) - \left((1-\rho(\mu))E^+ + \rho(\mu)E^-\right)(n) \right| < B \cdot \mathbb{E}[p^\mu] \cdot \gamma^k.
        \end{equation*}
        \item Suppose $\alpha_0, \alpha_1, \alpha_2$ are real numbers such that $\alpha_1 \neq 0$, one of $\alpha_0$ or $\alpha_2 \neq 0$, and $\alpha_0 + \alpha_1 + \alpha_2 = 1$. Then there exists explicitly computable constant $B > 0$ and $\gamma \in (0,1)$ such that
        \begin{equation*}
            \sup_{n \in \mathbb{Z}_{\geq 0}} \left| (\alpha_0 \cdot \mathbf{Id} + \alpha_1 \cdot \mathbf{M}_L + \alpha_2 \cdot \mathbf{M}_L^2)^k\mu(n) - \left(\frac{1}{2}E^+ + \frac{1}{2}E^-\right)(n) \right| < B \cdot \mathbb{E}[p^\mu] \cdot \gamma^k.
        \end{equation*}
        \item Suppose $\alpha_0, \alpha_2$ are real numbers such that $\alpha_2 \neq 0$ and $\alpha_0 + \alpha_2 = 1$. Then there exists explicitly computable constant $B > 0$ and $\gamma \in (0,1)$ such that
        \begin{equation*}
            \sup_{n \in \mathbb{Z}_{\geq 0}} \left| (\alpha_0 \cdot \mathbf{Id} + \alpha_2 \cdot \mathbf{M}_L^2)^k\mu(n) - \left(\rho(\mu)E^+ + (1-\rho(\mu))E^-\right)(n) \right| < B \cdot \mathbb{E}[p^\mu] \cdot \gamma^k.
        \end{equation*}
    \end{enumerate}
\end{theorem}

\section{Distribution of Selmer structures: proof of Theorem \ref{thm:intro}}\label{sec:generalthm}

Combining all the observations from previous sections, we prove the following result on the distribution of dimensions of twisted Selmer structures.
\begin{theorem} \label{thm:fanB} 
Assume ERH. Fix a prime $p$, a number field $K$, and a fan structure $\mathcal{B}_K(X)$ as above.
\begin{itemize}
    \item Suppose $p = 2$. Recall the probability distributions $\mathrm{E}^+_1(n)$ and $\mathrm{E}^-_1(n)$ from Theorem \ref{KMR14:main}, where we fix $\Sigma = \Sigma_T$. Denote by $\delta(T/K) := \frac{1}{2}(\rho(\mathrm{E}^-_1)-\rho(\mathrm{E}^+_1))$. Then
    \begin{small}
    \begin{align}
        \lim_{X \to \infty} \frac{\#\left\{ \chi \in \mathcal{B}_K(X) : \dim_{\mathbb{F}_p} \Sel(T, \chi : \alpha) = n \right\}}{\# \mathcal{B}_K(X)} &= \left(\frac{1}{2} + \delta(T/K) \right) E^+(n) + \left(\frac{1}{2} - \delta(T/K) \right) E^-(n)   \label{thm:main-numberfld-B-2} \\
        \text{and} \notag \\
        \lim_{X \to \infty} \frac{\#\left\{ \chi \in \mathcal{C}_K(X) : \dim_{\mathbb{F}_p} \Sel(T, \chi : \alpha) = n \right\}}{\# \mathcal{C}_K(X)} &= \left(\frac{1}{2} + \delta(T/K) \right) E^+(n) + \left(\frac{1}{2} - \delta(T/K) \right) E^-(n).  \label{thm:main-numberfld-C-2}
    \end{align}
    \end{small}
    \item Suppose $\delta_1 \neq 0$ and $p \neq 2$. Then 
    \begin{align}
        \lim_{X \to \infty} \frac{\#\left\{ \chi \in \mathcal{B}_K(X) : \dim_{\mathbb{F}_p} \Sel(T, \chi : \alpha) = n \right\}}{\# \mathcal{B}_K(X)} &= \frac{1}{2} E^+(n) + \frac{1}{2} E^-(n)   \label{thm:main-numberfld-B-1/2} \\
        \text{and} \notag \\
        \lim_{X \to \infty} \frac{\#\left\{ \chi \in \mathcal{C}_K(X) : \dim_{\mathbb{F}_p} \Sel(T, \chi : \alpha) = n \right\}}{\# \mathcal{C}_K(X)} &= \frac{1}{2} E^+(n) + \frac{1}{2} E^-(n).  \label{thm:main-numberfld-C-1/2}
    \end{align}
    \item Suppose $\delta_1 = 0$ and $p \neq 2$. Recall the probability distribution $\mathrm{E}_1(n)$ from Theorem \ref{KMR14:main}, where we fix $\Sigma = \Sigma_T$.
    Then 
    \begin{align}
        \lim_{X \to \infty} \frac{\#\left\{ \chi \in \mathcal{B}_K(X) : \dim_{\mathbb{F}_p} \Sel(T, \chi : \alpha) = n \right\}}{\# \mathcal{B}_K(X)} &= (1-\rho(\mathrm{E}_1)) E^+(n) + \rho(\mathrm{E}_1) E^-(n)   \label{thm:main-numberfld-B} \\
        \text{and} \notag \\
        \lim_{X \to \infty} \frac{\#\left\{ \chi \in \mathcal{C}_K(X) : \dim_{\mathbb{F}_p} \Sel(T, \chi : \alpha) = n \right\}}{\# \mathcal{C}_K(X)} &= (1-\rho(\mathrm{E}_1)) E^+(n) + \rho(\mathrm{E}_1) E^-(n).  \label{thm:main-numberfld-C}
    \end{align}
\end{itemize}
\end{theorem}
\begin{proof}

We first prove the desired statistical statement over the subset $\mathcal{B}_K(X)$.
Given a square-free product of places $\mathfrak{d}$ of $K$, let $\omega \in \Omega_\mathfrak{d}$.
For any Selmer structure $\mathrm{Sel}(T, \omega : \alpha)$ let $F_{\omega}/K(T)$ be the fixed field of  $\cap_c \ker \tilde c \subseteq G_K$  where $\tilde c$ is the restriction of $c \in \mathrm{Sel}(T, \omega: \alpha)$ to $\mathrm{Hom}(K(T), T)$. 

For the restriction $\tilde c$ of any $c \in \mathrm{Sel}(T, \omega : \alpha)$, say $F_c$ is the fixed field of $\ker \tilde c$. Note that $\mathrm{Gal}(F_c/K(T)) \leq (\mathbb{Z}/p\mathbb{Z})^2$. Consequently,  $\mathrm{Gal}(F_\omega/K) \hookrightarrow \mathrm{SL}_2(\mathbb{F}_p) \wr (\mathbb{Z}/p\mathbb{Z})^r$ for some non-negative integer $r$. Moreover, $[F_\omega : K] \leq cp^{c'm}$ for some constants $c,c'$ depending only on $T$ and $|\mathrm{Disc}(F)| \ll_{K, \Sigma} \mathrm{Nm}_\mathbb{Q}^K\mathfrak{d}^{[K :\mathbb{Q}]}$. 
 
Starting from here, we assume ERH. Suppose $\mathfrak{d} \in I_K^{Fan}(X)$. By Theorem \ref{thm:effective_chebotarev}, for any two conjugate-invariant non-trivial subsets $S, S' \subset \mathrm{Gal}(F_\omega/K)$ and any small enough $\epsilon > 0$, there exists an absolute constant $c > 0$ such that for sufficiently large $X$ we have
 \begin{equation}
     \left| \frac{\pi_S(X)}{\pi_{S'}(X)} - \frac{|S|}{|S'|}\right| \leq c \cdot  \frac{|S|}{|S'|} M_1^{-\frac{1}{2} + \epsilon},
 \end{equation}
where we recall that $\pi_S(X)$ is the number of prime ideals of norm at most $X$ whose Frobenius lies in $S$, and $M_1 = X^{\frac{1}{2^{m_l-1}}}$. For each $i = 0, 1, 2$, we denote by $\delta_i$ the following probabilities:
\begin{equation}
\delta_i = \frac{\#\{\mathfrak{p} \in \mathcal{O}_K\; : \mathrm{Nm}_\mathbb{Q}^K \mathfrak{p} \leq X, \omega(\mathfrak{p}) = i\}}{\#\{\mathfrak{p} \in \mathcal{O}_K\; : \mathrm{Nm}_\mathbb{Q}^K \mathfrak{p} \leq X \}}.
\end{equation}
We note that the densities $\delta_i$ can be obtained from applying Chebotarev density theorem to the Galois extension $K(T)/K$, see for example \cite[Lemma 5.9]{KMR14}. 

Before we proceed, we consider the following decomposition of the set $I_K^{Fan}(X)$. We can rewrite the prime ideal factorization of $\mathfrak{d} \in I_K^{Fan}(X)$ as
\begin{equation*}
    \mathfrak{d} = \prod_{i=1}^{m_s} \mathfrak{q}_i \cdot \prod_{j=1}^{m_l} \mathfrak{p}_{j}.
\end{equation*}
Given a fixed $m$, consider the following subset of ideals of $\mathcal{O}_K$:
\begin{equation}
    \mathcal{S}(X) := \left\{ \mathfrak{S} \subset \mathcal{O}_K : \omega(\mathfrak{S}) \leq \frac{\log \log X}{\log \log \log X}, \; \mathfrak{q} \mid \mathfrak{S} \implies \mathrm{Nm}_\mathbb{Q}^K \mathfrak{q} \leq (\log X)^A \right\},
\end{equation}
Given an ideal $\mathfrak{S} \in \mathcal{S}(X)$ and a prime ideal $\mathfrak{p} \in \mathcal{P}_{max}(X)$, we consider the following subset of $I_K^{Fan}(X)$:
\begin{equation}
    I_K^{Fan}(X,\mathfrak{S}) := \left\{ \mathfrak{d} := \prod_{i=1}^{m_s} \mathfrak{q}_i \prod_{j=1}^{m_l} \mathfrak{p}_j \in I_K^{Fan}(X) : \prod_{i=1}^{m_s} \mathfrak{q}_i = \mathfrak{S} \right\}.
\end{equation}
By definition, we have
\begin{equation}
    I_K^{Fan}(X) = \bigsqcup_{\mathfrak{S} \in \mathcal{S}(X)} I_K^{Fan}(X, \mathfrak{S}).
\end{equation}
Then we have the decomposition of the fan structure $\mathcal{B}_K(X)$ as 
\begin{equation}
    \mathcal{B}_K(X) := \bigsqcup_{\mathfrak{S} \in \mathcal{S}(X)} \bigsqcup_{\mathfrak{d} \in I_K^{Fan}(X,\mathfrak{S})} C_p(\mathfrak{d}).
\end{equation}

We treat the case of $p=2$ and $p \geq 3$ separately. 

\medskip 

\noindent \textbf{Case A, $p \neq 2$:}

Suppose first that $p \neq 2$.

We consider subsets of local characters $\bigsqcup_{\mathfrak{d} \in I_K^{Fan}(X, \mathfrak{S})} C_p(\mathfrak{d})$ for each choice of $\mathfrak{S} \in \mathcal{S}(X)$. Define for each non-negative integer $n \geq 0$ the probability
\begin{equation}
    E_\mathfrak{S}(n) := \frac{\#\{\chi \in \mathcal{C}_p(\mathfrak{S},X) : \dim_{\mathbb{F}_p} \mathrm{Sel}(T, \chi : \alpha) = n\}}{\# \mathcal{C}_p(\mathfrak{S},X)}.
\end{equation}
Then \cite[Proposition 7.2, Proposition 10.7]{KMR14} imply that there exists some Markov operator $\mathbf{S} := [s_{i,j}]_{i,j \geq 0}$ such that
\begin{equation*}
    E_\mathfrak{S}(n) = \mathbf{S}(E_1)(n),
\end{equation*}
and $s_{i,j} = 0$ whenever $|i - j| > 2 \omega(\mathfrak{S})$. We note that the construction of $\mathbf{S}$ depends on the choice of $\mathfrak{S}$, but we will demonstrate using Theorem \ref{thm:geometric-ergodicity} that such a dependence does not affect the statement of the main theorem.

The order of prime ideal factors in which we twist the Selmer structure $\mathrm{Sel}(T, \chi : \alpha)$ with respect to collections of $\mathfrak{d} = \mathfrak{S} \cdot \prod_{i=1}^{m_l} \mathfrak{p}_{m_l} \in I_K^{Fan}(X, \mathfrak{S})$ is given by
\begin{equation*}
    \mathfrak{p}_1, \mathfrak{p}_2, \mathfrak{p}_3, \cdots, \mathfrak{p}_{m_l-1}, \mathfrak{p}_{m_l}.
\end{equation*}
We consecutively use \cite[Proposition 9.4, Proposition 9.5, Theorem 11.6]{KMR14} for all such $m_l$ many prime ideal factors to construct the desired Markov operator governing variations of Selmer structures with respect to consecutive twists by $\mathfrak{p}_{m_i}$. Even though the norm conditions on $\mathfrak{p}_{m_l}$ depend on the norms of $\mathfrak{p}_i$'s and $\mathfrak{q}_j$'s, because the product of all such norms are at most $X^{\frac{1}{2} - \frac{1}{(\log X)^{c_2 \log 2}}} (\log X)^{\log \log X}$, the set of possible choices of ideals $\mathfrak{p}_{m_l}$ contain the set of ideals $\mathfrak{p}_{m_l}$ of norm between $X^{\frac{1}{4}}$ and $X^{\frac{1}{2} + \frac{1}{(\log X)^{c_2 \log 2}}} (\log X)^{-\log \log X}$. Hence, one can still apply Chebotarev density theorem as in \cite[Proposition 9.4, Proposition 9.5]{KMR14} to construct the desired Markov operator governing the variation of local Selmer structures with respect to twisting by $\mathfrak{p}_{m_l}$. Doing so allows us to define a set of non-negative real numbers $\{P_m\}_{m \in \mathbb{Z}}$ depending on the choice of $X$ that satisfies
\begin{equation}
\begin{cases}
    \sum_{m=0}^\infty P_m &= 1, \\
    P_m = 0 &\text{ if } m < c_1 \log \log X \text{ or } m > c_2 \log \log X,
\end{cases}
\end{equation}
such that the desired probability distribution
\begin{equation} \label{eq:keyexpression-p}
\frac{\#\left\{ \chi \in \bigsqcup_{\mathfrak{d} \in I_K^{Fan}(X,\mathfrak{S)}}C_p(\mathfrak{d}) : \dim_{\mathbb{F}_p} \Sel(T, \chi : \alpha) = n \right\}}{\# \bigsqcup_{\mathfrak{d} \in I_K^{Fan}(X,\mathfrak{S)}}C_p(\mathfrak{d})}
\end{equation}
can be rewritten as
\begin{align*}
    &= \sum_{m_l = \lceil c_1 \log \log X \rceil}^{\lfloor c_2 \log \log X \rfloor} P_{m_l} \cdot \left(\left(\left( \delta_0 \mathbf{Id} + \delta_1 \mathbf{M}_L + \delta_2 \mathbf{M}_L^2 \right)^{m_l}\right) (E_\mathfrak{S}) (n)\right) + O \left( X^{\left(-\frac{1}{2} + \epsilon \right) \cdot \frac{1}{(\log X)^{c_2 \log 2}}} \cdot \log \log X \right) \\
    &= \sum_{m_l = \lceil c_1 \log \log X \rceil}^{\lfloor c_2 \log \log X \rfloor} P_{m_l} \cdot \left( \left( \left( \delta_0 \mathbf{Id} + \delta_1 \mathbf{M}_L + \delta_2 \mathbf{M}_L^2 \right)^{m_l} \cdot \mathbf{S} \right) (E_1) (n)\right) + O \left( X^{\left(-\frac{1}{2} + \epsilon \right) \cdot \frac{1}{(\log X)^{c_2 \log 2}}} \cdot \log \log X \right),
\end{align*}
where the error term is obtained from using $M_1 := X^{\frac{1}{2^{m_l}}}$ and substituting $m_l = \lfloor c_2 \log \log X \rfloor$.
We further split our analysis into the cases $\delta_1 =0$ and $\delta_1 \neq 0$. Note that by assumption on $T$ we can guarantee that $\delta_0 \neq 0$ and $\delta_2 \neq 0$.


\begin{itemize}
    \item  \textbf{Case A.1: $\delta_1 \neq 0$}. We note that if $\delta_1 \neq 0$, then the Markov operator $\delta_0 \mathbf{Id} + \delta_1 \mathbf{M}_L + \delta_2 \mathbf{M}_L^2$ is irreducible, aperiodic, and positive recurrent over the state space of non-negative integers $\mathbb{Z}_{\geq 0}$. Hence, the Markov operator has a unique stationary distribution $\pi$ over the state space $\mathbb{Z}_{\geq 0}$. It is straightforward to check that the unique stationary distribution $\pi$ is given by the right hand side of \eqref{thm:main-numberfld-B-1/2}.

\item \textbf{Case A.2: $\delta_1 = 0$}.  If $\delta_1 = 0$, then the Markov operator $\delta_0 \mathbf{Id} + \delta_2 \mathbf{M}_L^2$ restricts to an irreducible, aperiodic, and positive recurrent Markov operator over two state spaces - one consisting of non-negative even integers $\mathbb{Z}_{\geq 0, even}$, and the other consisting of non-negative odd integers $\mathbb{Z}_{\geq 0, odd}$. Furthermore, the parity parameters $\rho(\mathbf{S}(E_1))$ and $\rho(E_1)$ are identical. We then let $X$ to grow arbitrarily large and use Definition \ref{def:Markov} to obtain the desired statistical statement over the subset $\mathcal{B}_K(X)$, proving \eqref{thm:main-numberfld-B}. 
\end{itemize}
We conclude \textbf{Case A} by taking disjoint unions of $I_K^{Fan}(X, \mathfrak{S})$ as $\mathfrak{S}$ varies over $\mathcal{S}(X)$.

\medskip 

\noindent \textbf{Case B, $p = 2$:}

We now suppose that $p = 2$. We recall the following probability distributions stated in Theorem \ref{KMR14:main}:
\begin{align*}
            \begin{split}
                \mathrm{E}^+_1(n) &:= \frac{\#\{\chi \in \mathcal{C}(1,X) : \dim_{\mathbb{F}_2} \mathrm{Sel}(T, \chi : \alpha) = n, \prod_{v \in \Sigma}\chi_v(\Delta) = 1\}}{\# \{\chi \in \mathcal{C}(1,X) : \prod_{v \in \chi_v} \chi_v(\Delta) = 1\}}, \\
                \mathrm{E}^-_1(n) &:= \frac{\#\{\chi \in \mathcal{C}(1,X) : \dim_{\mathbb{F}_2} \mathrm{Sel}(T, \chi : \alpha) = n, \prod_{v \in \Sigma}\chi_v(\Delta) = -1\}}{\# \{\chi \in \mathcal{C}(1,X) : \prod_{v \in \chi_v} \chi_v(\Delta) = -1\}}.
            \end{split}
\end{align*}
Denote by $\mathbb{R}[\mathbf{M}_L]$ a formal polynomial ring in one variable (given by the Markov operator $\mathbf{M}_L$). Let $\Phi_{even}, \Phi_{odd}: \mathbb{R}[\mathbf{M}_L] \to \mathbb{R}[\mathbf{M}_L]$ be $\mathbb{R}$-linear maps defined as
\begin{equation}
    \Phi_{even}(\mathbf{M}_L^k) := \begin{cases}
        \mathbf{M}_L^k &\text{ if } k \equiv 0 \mod 2, \\
        0 &\text{ if } k \equiv 1 \mod 2.
    \end{cases}
\end{equation}
\begin{equation}
    \Phi_{odd}(\mathbf{M}_L^k) := \begin{cases}
        0 &\text{ if } k \equiv 0 \mod 2, \\
        \mathbf{M}_L^k &\text{ if } k \equiv 1 \mod 2.
    \end{cases}
\end{equation}
We proceed as in \textbf{Case A}. We consider for each non-negative integer $n \geq 0$ the probability
\begin{align*}
    E_\mathfrak{S}(n) &:= \frac{\#\{\chi \in \mathcal{C}_p(\mathfrak{S},X) : \dim_{\mathbb{F}_p} \mathrm{Sel}(T, \chi : \alpha) = n\}}{\# \mathcal{C}_p(\mathfrak{S},X)}
\end{align*}
Then \cite[Proposition 7.2, Proposition 10.7]{KMR14} imply that there exists some Markov operator $\mathbf{S} := [s_{i,j}]_{i,j \geq 0}$ satisfying $s_{i,j} = 0$ whenever $|i - j| > 2 \omega(\mathfrak{S})$ and $|i - j| \not\equiv \omega(\mathfrak{S}) \mod 2$ such that
\begin{align*}
    E_\mathfrak{S}(n) = \begin{cases}
        \mathbf{S}(E_1^+)(n) &\text{ if } \omega(\mathfrak{S}) \equiv 0 \mod 2, \\
        \mathbf{S}(E_1^-)(n) &\text{ otherwise }.
    \end{cases}
\end{align*}
Furthermore, we also have $\rho(E_{\mathfrak{S}}(n)) = \rho(E_1^+(n))$ if $\omega(\mathfrak{S}) \equiv 0 \mod 2$, and $\rho(E_{\mathfrak{S}}(n)) = 1 - \rho(E_1^-(n))$ otherwise.

As before, the order of prime ideal factors in which we twist the Selmer structure $\mathrm{Sel}(T, \chi : \alpha)$ with respect to collections of $\mathfrak{d} = \mathfrak{S} \cdot \prod_{i=1}^{m_l} \mathfrak{p}_{m_l} \in I_K^{Fan}(X, \mathfrak{S})$ is given by
\begin{equation*}
    \mathfrak{p}_1, \mathfrak{p}_2, \mathfrak{p}_3, \cdots, \mathfrak{p}_{m_l-1}, \mathfrak{p}_{m_l}.
\end{equation*}
We consecutively use \cite[Proposition 9.4, Proposition 9.5, Theorem 11.6]{KMR14} for all such $m_l$ many prime ideal factors. Again, one can still apply Chebotarev density theorem as in \cite[Proposition 9.4, Proposition 9.5]{KMR14} to construct the desired Markov operator governing the variation of local Selmer structures with respect to twisting by $\mathfrak{p}_{m_l}$. 

One additional condition we need to consider is the parity of $\omega(\prod_{j=1}^{m_l} \mathfrak{p}_j)$ which affects the initial distribution to which we apply the desired Markov operators, as stated in Theorem \ref{KMR14:main}. Taking this into consideration, we can thus construct a set of non-negative real numbers $\{Q_m^\pm\}_{m \in \mathbb{Z}}$ depending on the choice of $X$ that satisfies
\begin{equation}
\begin{cases}
    \sum_{m=0}^\infty Q_m^+ + \sum_{m=0}^\infty Q_m^- &= 1, \\
    Q_m^\pm = 0 &\text{ if } m < c_1 \log \log X \text{ or } m > c_2 \log \log X,
\end{cases}
\end{equation}
such that the desired probability distribution
\begin{equation} \label{eqn:p=2}
\frac{\#\left\{ \chi \in \bigsqcup_{\mathfrak{d} \in I_K^{Fan}(X, \mathfrak{S})}C_p(\mathfrak{d}) : \dim_{\mathbb{F}_2} \Sel(T, \chi : \alpha) = n \right\}}{\# \bigsqcup_{\mathfrak{d} \in I_K^{Fan}(X, \mathfrak{S})}C_p(\mathfrak{d})}
\end{equation}
can be rewritten as follows. If $\omega(\mathfrak{S})$ is even, then we have
\begin{small}
\begin{align*}
    \eqref{eqn:p=2} &= \sum_{m_l = \lceil c_1 \log \log X \rceil}^{\lfloor c_2 \log \log X \rfloor} \biggl\{ Q_{m_l}^+ \cdot \biggl[ \Phi_{even} \left(\left( \frac{1}{3} \mathbf{Id} + \frac{1}{2} \mathbf{M}_L + \frac{1}{6} \mathbf{M}_L^2 \right)^{m_l} \right) \cdot \mathbf{S} \\
    & \hspace{100pt} + \Phi_{odd} \left(\left( \frac{1}{3} \mathbf{Id} + \frac{1}{2} \mathbf{M}_L + \frac{1}{6} \mathbf{M}_L^2 \right)^{m_l} \right) \cdot \mathbf{S}\biggr] (E^+_1) (n) \\
    & \hspace{60pt} + Q_{m_l}^- \cdot \biggl[ \Phi_{odd} \left(\left( \frac{1}{3} \mathbf{Id} + \frac{1}{2} \mathbf{M}_L + \frac{1}{6} \mathbf{M}_L^2 \right)^{m_l} \right) \cdot \mathbf{S} \\
    & \hspace{100pt} + \Phi_{even} \left(\left( \frac{1}{3} \mathbf{Id} + \frac{1}{2} \mathbf{M}_L + \frac{1}{6} \mathbf{M}_L^2 \right)^{m_l} \right) \cdot \mathbf{S} \biggr] (E^-_1)(n) \biggr\} \\
    & \hspace{15pt} + O \left( X^{\left(-\frac{1}{2} + \epsilon \right) \cdot \frac{1}{(\log X)^{c_2 \log 2}}} \cdot \log \log X \right).
\end{align*}
\end{small}
Otherwise, if $\omega(\mathfrak{S})$ is odd, then we have
\begin{small}
\begin{align*}
    \eqref{eqn:p=2} &= \sum_{m_l = \lceil c_1 \log \log X \rceil}^{\lfloor c_2 \log \log X \rfloor} \biggl\{ Q_{m_l}^+ \cdot \biggl[ \Phi_{even} \left(\left( \frac{1}{3} \mathbf{Id} + \frac{1}{2} \mathbf{M}_L + \frac{1}{6} \mathbf{M}_L^2 \right)^{m_l} \right) \cdot \mathbf{S} \\
    & \hspace{100pt} + \Phi_{odd} \left(\left( \frac{1}{3} \mathbf{Id} + \frac{1}{2} \mathbf{M}_L + \frac{1}{6} \mathbf{M}_L^2 \right)^{m_l} \right) \cdot \mathbf{S}\biggr] (E^-_1) (n) \\
    & \hspace{60pt} + Q_{m_l}^- \cdot \biggl[ \Phi_{odd} \left(\left( \frac{1}{3} \mathbf{Id} + \frac{1}{2} \mathbf{M}_L + \frac{1}{6} \mathbf{M}_L^2 \right)^{m_l} \right) \cdot \mathbf{S} \\
    & \hspace{100pt} + \Phi_{even} \left(\left( \frac{1}{3} \mathbf{Id} + \frac{1}{2} \mathbf{M}_L + \frac{1}{6} \mathbf{M}_L^2 \right)^{m_l} \right) \cdot \mathbf{S} \biggr] (E^+_1)(n) \biggr\} \\
    & \hspace{15pt} + O \left( X^{\left(-\frac{1}{2} + \epsilon \right) \cdot \frac{1}{(\log X)^{c_2 \log 2}}} \cdot \log \log X \right).
\end{align*}
\end{small}
Here the identifications $\delta_0 = \frac{1}{3}$, $\delta_1 = \frac{1}{2}$, and $\delta_2 = \frac{1}{6}$ come from using effective Chebotarev density theorem with respect to the Galois extension $K(T)/K$. 

For each $m_l$, we can rewrite the Markov operators appearing in the summation as
\begin{align*}
    \Phi_{even} \left(\left( \frac{1}{3} \mathbf{Id} + \frac{1}{2} \mathbf{M}_L + \frac{1}{6} \mathbf{M}_L^2 \right)^{m_l} \right) &= \sum_{\substack{k = 0 \\ k \equiv 0 \mod 2}}^{2m_l-2} \delta_{m_l,k} \mathbf{M}_L^k, \\
    \Phi_{odd} \left(\left( \frac{1}{3} \mathbf{Id} + \frac{1}{2} \mathbf{M}_L + \frac{1}{6} \mathbf{M}_L^2 \right)^{m_l} \right) &= \sum_{\substack{k = 0 \\ k \equiv 1 \mod 2}}^{2m_l-2} \delta_{m_l,k} \mathbf{M}_L^k.
\end{align*}
where $\{\delta_{m_l,k}\}_{k=0}^{2m_l}$ are non-negative constants obtained from multinomial expansion of the expression $\left( \frac{1}{3} \mathbf{Id} + \frac{1}{2} \mathbf{M}_L + \frac{1}{6} \mathbf{M}_L^2 \right)^{m_l}$. We first observe that 
\begin{equation} \label{eqn:p=2-parity}
    \sum_{\substack{k=0 \\ k \equiv 0 \mod 2}}^{2m_l} \delta_{m_l,k} = \sum_{\substack{k=0 \\ k \equiv 1 \mod 2}}^{2m_l} \delta_{m_l,k} = \frac{1}{2}.
\end{equation}
Observe that by comparing the coefficients of the multinomial expansion, we obtain
\begin{equation} \label{eqn:p=2-parity2}
    \sum_{k=0}^{\frac{m_l}{2}} \delta_{m_l,k} \leq \sum_{k=\frac{3}{4}m_l}^{m_l} \frac{1}{3^k} < \frac{3}{2} \cdot 3^{-\frac{3}{4}m_l}.
\end{equation}
Equation \eqref{eqn:p=2-parity2} allows us to rewrite the expression \eqref{eqn:p=2} as follows. In case $\omega(\mathfrak{S})$ is even, we have
\begin{align*}
    \eqref{eqn:p=2} &= \sum_{m_l = \lceil c_1 \log \log X \rceil}^{\lfloor c_2 \log \log X \rfloor} Q_{m_l}^+ \cdot \left( \sum_{\substack{k=\frac{m_l}{2} \\ k \equiv 0 \mod 2}}^{2m_l} \delta_{m,k} (\mathbf{M}_L^k \cdot \mathbf{S}) + \sum_{\substack{k=\frac{m_l}{2} \\ k \equiv 1 \mod 2}}^{2m_l} \delta_{m,k} (\mathbf{M}_L^k \cdot \mathbf{S}) \right)(E_1^+)(n) \\
    & \hspace{50pt}+ Q_{m_l}^- \cdot \left( \sum_{\substack{k=\frac{m_l}{2} \\ k \equiv 1 \mod 2}}^{2m_l} \delta_{m,k} (\mathbf{M}_L^k \cdot \mathbf{S}) + \sum_{\substack{k=\frac{m_l}{2} \\ k \equiv 0 \mod 2}}^{2m_l} \delta_{m,k} (\mathbf{M}_L^k \cdot \mathbf{S}) \right) (E_1^-(n))\\
    & \hspace{15pt} + O \left( (\log X)^{-\frac{3}{4} c_1 \log 3} + X^{\left(-\frac{1}{2} + \epsilon \right) \cdot \frac{1}{(\log X)^{c_2 \log 2}}} \right).
\end{align*}
If $\omega(\mathfrak{S})$ is odd, then we have
\begin{align*}
    \eqref{eqn:p=2} &= \sum_{m_l = \lceil c_1 \log \log X \rceil}^{\lfloor c_2 \log \log X \rfloor} Q_{m_l}^+ \cdot \left( \sum_{\substack{k=\frac{m_l}{2} \\ k \equiv 0 \mod 2}}^{2m_l} \delta_{m,k} (\mathbf{M}_L^k \cdot \mathbf{S}) + \sum_{\substack{k=\frac{m_l}{2} \\ k \equiv 1 \mod 2}}^{2m_l} \delta_{m,k} (\mathbf{M}_L^k \cdot \mathbf{S}) \right)(E_1^-)(n) \\
    & \hspace{50pt}+ Q_{m_l}^- \cdot \left( \sum_{\substack{k=\frac{m_l}{2} \\ k \equiv 1 \mod 2}}^{2m_l} \delta_{m,k} (\mathbf{M}_L^k \cdot \mathbf{S}) + \sum_{\substack{k=\frac{m_l}{2} \\ k \equiv 0 \mod 2}}^{2m_l} \delta_{m,k} (\mathbf{M}_L^k \cdot \mathbf{S}) \right) (E_1^+(n))\\
    & \hspace{15pt} + O \left( (\log X)^{-\frac{3}{4} c_1 \log 3} + X^{\left(-\frac{1}{2} + \epsilon \right) \cdot \frac{1}{(\log X)^{c_2 \log 2}}} \right).
\end{align*}
Let $X$ be arbitrarily large. Then we use stationary distributions of the Markov chains $\mathbf{M}_L^{2k}$ and $\mathbf{M}_L^{2k+1}$ appearing in Definition \ref{def:Markov} and equation \eqref{eqn:p=2-parity} to rewrite \eqref{eqn:p=2} as
\begin{align*}
    \lim_{X \to \infty} \eqref{eqn:p=2} &= \left(\frac{1}{2} (1 - \rho(E^+_1)) E^+(n) + \frac{1}{2} \rho(E^+_1) E^-(n)\right) + \left(\frac{1}{2} \rho(E^-_1) E^+(n) + \frac{1}{2} (1-\rho(E^-_1)) E^-(n)\right) \\
    &= \left(\frac{1}{2} + \frac{1}{2} \left(\rho(E^-_1) - \rho(E^+_1) \right) \right) E^+(n) + \left(\frac{1}{2} - \frac{1}{2} \left(\rho(E^-_1) - \rho(E^+_1) \right) \right) E^-(n).
\end{align*}
We then conclude \textbf{Case B} by taking disjoint unions of $I_K^{Fan}(X, \mathfrak{S})$ as $\mathfrak{S}$ varies over $\mathcal{S}(X)$.

\medskip 

\noindent \textbf{Geometric Ergodicity}: One fact we crucially use in both \textbf{Case A} and \textbf{Case B} is that the probability distributions $\mathbf{S}(E_1)(n)$ and $\mathbf{S}(E_1^\pm)(n)$ are equal to $0$ as long as $n > 2 \frac{\log \log X}{\log \log \log X} + 2 + C_T$ for some fixed constant $C_T$ depending only on the choice of $T$. We can apply geometric ergodicity of Markov operators, in particular Theorem \ref{thm:geometric-ergodicity} with respective choices of parts (1),(2),(3),(4), and prime $p$. In the statement of the theorem we set $\mu := \mathbf{S}(E_1)$ or $\mathbf{S}(E_1^\pm)$ and set $N := 2 \frac{\log \log X}{\log \log \log X} + 2 + C_T$. Then one has
\begin{equation*}
    \mathbb{E}[p^\mu] < p^{2 \frac{\log \log X}{\log \log \log X} + 2 + C_T} = p^{2 + C_T} \cdot (\log X)^{\frac{2 \log p}{\log \log \log X}}.
\end{equation*}
Because in both \textbf{Case A} and \textbf{Case B} we apply the respective Markov operators appearing in each part of Theorem \ref{thm:geometric-ergodicity} by at least $k = \frac{1}{2} \cdot \lceil c_1 \log \log X \rceil$ many iterations, the supremum of the rate of convergence of the result of applying the Markov operator to $\mu$ is at most
\begin{align*}
    B \cdot \mathbb{E}[p^\mu] \cdot \gamma^k &< B \cdot p^{2 + C_T} \cdot (\log X)^{\frac{2 \log p}{\log \log \log X}} \cdot \gamma^{\frac{1}{2} \cdot c_1 \log \log X} \\
    &\leq B \cdot p^{2 + C_T} \cdot (\log X)^{\frac{2 \log p}{\log \log \log X}} \cdot (\log X)^{\frac{c_1}{2} \log \gamma} \\
    &\leq B \cdot p^{2 + C_T} \cdot (\log X)^{\frac{c_1}{2} \log \gamma + \frac{2 \log p}{\log \log \log X}}
\end{align*}
for some explicit constant $B > 0$ and $\gamma \in (0,1)$. Hence for both cases, rate of convergence is given by some explicit constant multiple of $(\log X)^{\frac{c_1}{2} \log \gamma + o(1)}$. In particular, for the expression \eqref{eq:keyexpression-p}, the rate of convergence is given by some explicit constant multiple of $(\log X)^{c_1\log \gamma + o(1)}$.

\medskip 

\noindent \textbf{Extending the statistics:}   We now prove \eqref{thm:main-numberfld-C-2}, \eqref{thm:main-numberfld-C-1/2}, and \eqref{thm:main-numberfld-C}, the analogous statements over $\mathcal{C}_K(X)$. By Theorem \ref{thm:Erdos-Kac}, all but $o(X)$ many ideals $\mathfrak{m} \subset \mathcal{O}_K$ of norm bounded above by $X$ satisfy the condition that $\omega(\mathfrak{m})$ is between $c_1 \log \log X$ and $c_2 \log \log X$ for some $c_1 \in (0,1)$ and $c_2 \in (1, 1/\log 2)$. By Proposition \ref{prop:IfanvsI}, the number of square-free ideals $\mathfrak{m} \subset \mathcal{I}_K^{\square-free}$ satisfying $c_1 \log \log X \leq \omega(\mathfrak{m}) \leq c_2 \log \log X$ that does not lie in $\mathcal{I}_K^{Fan}(X)$ is at most $O(X/\log \log X)$, where the implied constant of the error term depends on $K$. As we let $X$ to grow arbitrarily large, the number of such ideals form a density $0$ subset of $\mathcal{I}_K^{\square-free}$.
We combine this fact with Theorem \ref{thm:Erdos-Kac} to obtain
\begin{equation}
    \frac{\# I_K^{Fan}(X)}{\# I_K(X)} = 1 - o_K(X),
\end{equation}
where the implied constant of the error term depends on $K$. We can then take $X$ to grow arbitrarily large that asymptotically $100 \%$ of $I_K(X)$ is $I_K^{Fan}(X)$, from which the desired statistical statement over $\mathcal{C}_K(X)$ is determined from that over $\mathcal{B}_K(X)$.
\end{proof}

\begin{remark} \label{remark:improve_rate_convergence}
    For the case where $K = \mathbb{F}_q(t)$ is a global function field with $q \equiv 1 \text{ mod } p$, the rate of convergence of twisted Selmer groups of a fixed Galois module satisfying an analogue of Condition \ref{condition:T} is of order $O(1/(n \log q)^{c(p)})$ for some constant $c(p) \in (0,1)$ depending on $p$ \cite{park2022prime, Park25-1}. Here, $n$ is the degree of the conductor of the order $p$ cyclic character of $\mathrm{Gal}(\overline{K}/K)$ twisting the fixed Galois module. It is a natural question to ask whether a same order of magnitude of the rate of convergence of distribution of twisted Selmer groups can be obtained as well for $K$ a number field assuming ERH.
    
    We can use the relaxed definition of $\mathcal{D}_{m,k,X}$ (and the fan structure $\mathcal{B}_m(X)$ accordingly) appearing in Remark \ref{remark:better_fan} to improve the rate of convergence of the limits in Theorem \ref{thm:fanB} to be of order $O \left(1 / (\log X)^c \right)$ for some constant $c \in (0,1)$. For such a $\mathfrak{d} \in \mathcal{D}_{m,k,X}$, we consider the order of prime ideal factors (excluding $\mathfrak{S}$) in which we twist the Selmer structure $\mathrm{Sel}(T, \chi : \alpha)$ to be
    \begin{equation*}
        \mathfrak{p}_1, \mathfrak{p}_2, \cdots, \mathfrak{p}_{m_l - \lfloor \epsilon' \log \log X \rfloor - 1}, \mathfrak{p}_{m_l - \lfloor \epsilon' \log \log X \rfloor }, \cdots, \mathfrak{p}_{m_l - 1}, \mathfrak{p}_{m_l}.
    \end{equation*}
    We note that one can still apply Chebotarev density theorem governing the twists by $\mathfrak{p}_{m_l - \lfloor \epsilon' \log \log X \rfloor}$, $\cdots$, $\mathfrak{p}_{m_l}$ because the one is choosing at most $\epsilon' \log \log X$ many prime ideals inside the set of prime ideals that contain the set of those of norm between $X^a$ and $X^{\frac{1}{2}}$ for any $a > 0$.
    Under such an order of twist, the distribution of dimensions of $\mathrm{Sel}(T, \chi : \alpha)$ converges to the respective distribution with rate of convergence at most of order $O \left( p^{2 \epsilon' \log \log X + \delta} \cdot \gamma^{c_1 \log \log X} \right)$, where $\gamma \in (0,1)$ is the geometric rate of convergence of the Markov operator $\mathbf{M}_L$ obtained from \cite[Theorem 15.0.1]{MT93}. We choose $\epsilon' \in (0, \frac{-c_1 \log \gamma}{2 \log p})$ to ensure that the term $p^{2 \epsilon' \log \log X + \delta} \cdot \gamma^{c_1 \log \log X} \ll_\delta (\log X)^{2 \epsilon' \log p + c_1 \log \gamma}$ converges to $0$ as $X$ grows arbitrarily large. Considering all the error terms, the exponent $c$ determining the rate of convergence $O(1/(\log X)^c)$ can be computed from computing the order $$O \left( (\log X)^{\frac{c_2 \log 2 + 1}{2} - 1} + (\log X)^{\epsilon' \log \log 2 + \delta} + (\log X)^{2 \epsilon' \log p + c_1 \log \gamma} \right)$$ for any suitable choice of $\epsilon' \in (0, \frac{-c_1 \log \gamma}{2 \log p})$.
\end{remark}

\begin{remark}
    The stationary distribution for \textit{every} Markov process under consideration is the same given the parity of the initial distribution $\mu$: Namely, the one given by $\delta_0 \mathbf{Id} + \delta_1 \mathbf{M}_L + \delta_2 \mathbf{M}_L^2$, for the mod-$p$ Lagrangian operator $\mathbf{M}_L$. For this reason, we don't need to be explicit about the constants $P_{m_l}$ and $Q_{m_l}^{\pm}$ in the proof of the theorem.
\end{remark}

\begin{remark}
    Using the large deviation results from \cite{MZ16, KP25-LD}, one can explicitly compute that the rate of convergence of the limits appearing in Theorem \ref{thm:fanB} is of order $O \left( 1 / \log \log X \right)$. The main contributor for the rate of convergence originates from the condition on the norm of the second largest prime ideal factor $\mathrm{Nm}_\mathbb{Q}^K \mathfrak{p}_{m_l-1}$ of elements $\mathfrak{d}$ in $\mathcal{D}_{m,k,X}$. This rate of convergence (conditional under ERH) is a slight improvement of the rate of convergence of distribution of $2^k$-Selmer groups of twist families of some fixed Galois modules computed in \cite[Theorem 4.18]{Smith1}, which is of order $O \left( \mathrm{Exp} \left(c (\log \log \log X)^{-\frac{1}{2}} \right) \right)$ for some constant $c > 0$.
\end{remark}

We can now use Theorem \ref{thm:fanB} to prove Theorem \ref{thm:intro}.
\begin{proof}[Proof of Theorem \ref{thm:intro}]
    Let $r_1,r_2$ be the number of real/complex embeddings of $K$. Then,
    \begin{equation}\label{eq:CtoN}
    N_p(K,X) = p^{r_1+r_2-1} \cdot \#\{F := \text{Fixed Field of} \ker \chi \mid \chi \in C_p(X) \}.
    \end{equation}
    Given a generator $\sigma \in \mathrm{Gal}(F/K) \cong \mathbb{Z}/p\mathbb{Z}$, we have
    \begin{equation}\label{eq:rankbound}
    \mathrm{rk}(E/F) - \mathrm{rk}(E/K) \leq (p-1) \cdot \dim_{\mathbb{F}_p} \mathrm{Sel}_{1-\sigma}(A^\chi / K)
    \end{equation}
    from Example \ref{ex:Twists}.
    The result now follows from \eqref{eq:CtoN}, \eqref{eq:rankbound},  Theorem \ref{thm:fanB}, and Definition \ref{def:Markov}.
\end{proof}
\begin{remark}
    In fact, Theorem \ref{thm:intro} affirms a weaker statement of \cite[Example 3.9]{Smith2}, where Smith shows that a given principally polarized abelian variety over a number field $K$ satisfying some technical conditions (see for example \cite[Case 2.12, Case 2.13]{Smith2}) obtains rank growth of at most $\ell - 1$ over $100\%$ of cyclic $\ell$ Galois extensions over $K$. Our result, however, demonstrates that assuming ERH allows one to improve the rate of convergence up to a power of log saving error term.
\end{remark}

\section{Superelliptic Curves: proof of Theorem \ref{thm:intro_superelliptic}}\label{sec:superelliptic}

We now proceed to prove Theorem \ref{thm:intro_superelliptic}. Fix any prime $p \geq 5$. Let $K$ be a number field which contains $\mathbb{Q}(\zeta_p)$. Fix a superelliptic curve $C: y^p = x^3 + a_2x^2 + a_1x + a_0$ such that the splitting field of the cubic polynomial $x^3 + a_2x^2 + a_1x + a_0$ is an $S_3$ Galois extension over $K$.

We denote by $\mathcal{S}_p(X)$ the set of isomorphism classes of twist families of $C$ as follows:
\begin{equation}
    \mathcal{S}_p(C, X) := \left\{ C_d: dy^p = x^3 + a_2x^2 + a_1x + a_0 \; | \; \mathrm{Gal}(K(\sqrt[p]{d})/K) \cong \mathbb{Z}/p\mathbb{Z}, \; K(\sqrt[p]{d}) \in \mathcal{C}_K(X) \right\}.
\end{equation}
We obtain the following theorem.
\begin{theorem}
    Assume ERH and fix a prime $p \geq 5$. Suppose $C: y^p = F(x) := x^3 + a_2x^2 + a_1x + a_0$ is a superelliptic curve over $K$ such that the splitting field of $F(x)$ is an $S_3$ Galois extension over $K$. Then there exists a fixed constant $\delta_C \in [0,1]$ depending on the choice of $C$ such that for any $n \geq 0$,
    \begin{equation}
        \lim_{X \to \infty} \frac{\#\{C_d \in \mathcal{S}_p(C,X) \; | \; \dim_{\mathbb{F}_p} \mathrm{Sel}_{1-\zeta_p}(\mathrm{Jac}(C_d)/K) = n\}}{\# \mathcal{S}_p(C,X)} = (1-\delta_C) E^+(n) + \delta_C E^-(n).
    \end{equation}
\end{theorem}
\begin{proof}
    From Example \ref{example:superelliptic}, we have
    \begin{equation*}
        \mathrm{Sel}_{1-\zeta_p}(\mathrm{Jac}(C_d)/K) = \mathrm{Sel}(\mathrm{Jac}(C)[1-\zeta_p], \chi^d : \alpha),
    \end{equation*}
    where $\chi^d$ and $\alpha$ are defined as in the example.
    We note that $\delta_1 = 0$ for the given family $\mathcal{S}_p(C,X)$ for any such choice of $C$. This follows from combining \cite[Lemma 5.9]{KMR14} with the fact that no Frobenius element in the Galois group of the splitting field of the cubic polynomial $x^3 + a_2 x^2 + a_1 x + a_0$ are of order $p \geq 5$. Therefore, Theorem \ref{thm:fanB} gives the desired statement of the theorem, with $\delta_C := \rho(\mathrm{E}_1)$. We note that $\delta_C \neq \frac{1}{2}$ because the denominator of $\mathrm{E}_1$ is odd when $p \neq 2$, see for example \cite[Theorem 3.3]{Park25-1} for a similar phenomenon observed for analogously defined twist families of superelliptic curves over global function fields.
\end{proof}

The following result on uniform Mordell-Lang conjecture for curves \cite{DGH21} relates the algebraic rank of the Jacobian of a higher genus curve $C$ with its number of $K$-rational points. We state the simplified version of their theorem for superelliptic curves $C_d: dy^p = x^3 + a_2 x^2 + a_1 x + a_0$, whose genus is equal to $p-1$.
\begin{theorem}[Theorem 1.1 of \cite{DGH21}] \label{thm:uniform_mordell_lang}
    There exists a fixed constant $c(p,K) \geq 1$ independent of $X > 0$ and a choice of a superelliptic curve $C$ such that for any superelliptic curve $C_d \in \mathcal{S}_p(C,X)$,
    \begin{equation}
        \# C_d(K) \leq c(p,K)^{1 + \mathrm{rk}(\mathrm{Jac}(C_d)/K)}.
    \end{equation}
\end{theorem}

Using the above theorem, we can obtain the following result on Mordell-Lang conjecture for families of curves in $\mathcal{S}_p(C,X)$ as follows. The theorem below is a refinement of Theorem \ref{thm:intro_superelliptic}.
\begin{theorem}
    Assume ERH and fix a prime $p \geq 5$. Suppose $C: y^p = F(x) := x^3 + a_2x^2 + a_1x + a_0$ is a superelliptic curve over $K$ such that the splitting field of $F(x)$ is an $S_3$ Galois extension over $K$. Recall that $c(p,K)$ is an explicit constant from Theorem \ref{thm:uniform_mordell_lang}.
    \begin{itemize}
        \item For every $n \geq 0$,
        \begin{align}
        \limsup_{X \to \infty} \frac{\#\left\{ C_d \in \mathcal{S}_p(C,X) \; | \; \# C_d(K) \leq c(p,K)^{1+(p-1) \cdot n} \right\}}{\# \mathcal{S}_p(C,X)} \leq \min \left[\sum_{k=0}^n E^{+}(k), \sum_{k=0}^n E^-(k) \right].
    \end{align}
        \item There exists an explicit fixed constant $B < \infty$ such that
        \begin{align}
            \limsup_{X \to \infty} \frac{\sum_{C_d \in \mathcal{S}_p(C,X)} \# C_d(K)}{\# \mathcal{S}_p(C,X)} \leq B.
        \end{align}
        \item For arbitrarily large $X$, at least $99\%$ of curves $C_d \in \mathcal{S}_p(C,X)$ has at most $c(p,K)^{3p-2}$ many $K$-rational points.
    \end{itemize}
\end{theorem}
\begin{proof}
    We use the upper bound from \cite[Proposition 3.5, Corollary 3.7]{Sch98} to get
    \begin{equation}
        \mathrm{rk}(\mathrm{Jac}(C_d)/K) \leq (p-1) \cdot \dim_{\mathbb{F}_p} \mathrm{Sel}_{1-\zeta_p}(\mathrm{Jac}(C_d)/K).
    \end{equation}
    Combined with Theorem \ref{thm:uniform_mordell_lang}, we obtain
    \begin{equation*}
        \# C_d(K) \leq c(p,K)^{1 + (p-1) \cdot \dim_{\mathbb{F}_p}\mathrm{Sel}_{1-\zeta_p}(\mathrm{Jac}(C_d)/K)}.
    \end{equation*}
    We then use the identity
    \begin{equation}
        \min_{\delta \in [0,1]} \left[ (1-\delta) \sum_{k=0}^n E^+(k) + \delta \sum_{k=0}^n E^-(k) \right] = \min \left[\sum_{k=0}^n E^{+}(k), \sum_{k=0}^n E^-(k) \right]
    \end{equation}
    to obtain the first statement of the theorem. To prove the second statement, we use the inequality 
    \begin{equation}
        \limsup_{X \to \infty} \frac{\sum_{C_d \in \mathcal{S}_p(C,X)} \# C_d(K)}{\# \mathcal{S}_p(C,X)} \leq \min \left[ \sum_{k=0}^\infty E^+(k) \cdot c(p,K)^{1+(p-1)\cdot k}, \sum_{k=0}^\infty E^-(k) \cdot c(p,K)^{1+(p-1) \cdot k} \right].
    \end{equation}
    Because both functions $E^+(n)$ and $E^-(n)$ decay quadratic exponentially at a rate of $O(p^{-n(n-1)/2})$ as $n$ grows arbitrarily large, both summations appearing on the right hand side of the inequality converge. We set $B$ to be the right hand side of the inequality to prove the second statement of the theorem. The last statement of the theorem follows from the inequality that holds for every $p \geq 5$:
    \begin{equation*}
        \min \left[\sum_{k=0}^2 E^{+}(k), \sum_{k=0}^2 E^-(k) \right] \geq \frac{99}{100}.
    \end{equation*}
\end{proof}


\section{Hyperelliptic Curves: proof of Theorem \ref{thm:intro_hyperelliptic}} \label{sec:hyperelliptic}

Using the techniques of the proof from the previous sections, in particular the case for $p = 2$ of Theorem \ref{thm:fanB}, we prove Theorem \ref{thm:hyperelliptic}. We borrow all the notations from Section \ref{sec:Selmer} and \ref{sec:newfan}, but with suitable adjustments. 

\subsection{Poonen-Rains Heuristics}

The statistics of dimensions of 2-Selmer groups of quadratic twist families of Jacobians of hyperelliptic curves depend on whether the hyperelliptic curve has a marked rational point at infinity or not. We introduce an invariant of Jacobians of hyperelliptic curves, whose size affects the distribution of the desired 2-Selmer groups.

We recall that $\Sha^1(K, \mathrm{Jac}(C)[2])$ is the following group:
\begin{equation}
    \Sha^1(K, \mathrm{Jac}(C)[2]) := \mathrm{Ker} \left( H^1(K, \mathrm{Jac}(C)[2]) \to \prod_{v \text{ place of } K} H^1(K_v, \mathrm{Jac}(C)[2]) \right).
\end{equation}
In fact, we have $\Sha^1(K, \mathrm{Jac}(C)[2]) = 0$ or $\mathbb{Z}/2\mathbb{Z}$. This follows from a highly non-trivial result by Poonen and Stoll \cite{PS99}, see in particular \cite[Conjecture 1.8, Remark 2.8, 4.15]{PR12} for more details. 

Given a quadratic character $\chi \in \mathcal{C}_K(X)$, we denote by $C^\chi$ the quadratic twist of $C$ by the global character $\chi$. The Galois-equivariant isomorphism $\mathrm{Jac}(C)[2] \cong \mathrm{Jac}(C)^\chi[2]$ implies that one has
\begin{equation*}
    \Sha^1(K, \mathrm{Jac}(C)[2]) \cong \Sha^1(K, \mathrm{Jac}(C^\chi)[2]).
\end{equation*}
We also have
\begin{equation*}
    H^1(K, \mathrm{Jac}(C)[2]) \cong H^1(K, \mathrm{Jac}(C^\chi)[2])
\end{equation*}
and for any local field $K_v$, 
\begin{equation*}
    H^1(K_v, \mathrm{Jac}(C)[2]) \cong H^1(K_v, \mathrm{Jac}(C^\chi)[2]).
\end{equation*}
Using \cite[Theorem 4.14]{PR12}, we obtain that the quotient of the 2-Selmer group
\begin{equation*}
    \frac{\mathrm{Sel}_2(\mathrm{Jac}(C^\chi)/K)}{\Sha^1(K, \mathrm{Jac}(C)[2])}
\end{equation*}
is isomorphic to the intersection of the images of the horizontal and the vertical map as presented in the diagram below (note that the quotienting space is invariant of the choice of $\chi$).
\begin{equation}
    \begin{tikzcd}
    & H^1(K, \mathrm{Jac}(C)[2]) \arrow[d] \\
    \prod_{v \text{ place of } K} \frac{\mathrm{Jac}(C^\chi)(K_v)}{2\mathrm{Jac}(C^\chi)(K_v)} \arrow[r] & \prod_{v \text{ place of } K} H^1(K_v, \mathrm{Jac}(C)[2]).
    \end{tikzcd}
\end{equation}
To understand the distribution of dimensions of $\mathrm{Sel}_2(\mathrm{Jac}(C^\chi)/K)$ as $\chi$ varies over $\mathcal{C}_K(X)$, it suffices to understand the distribution of the image of the horizontal maps within a fixed infinite dimensional vector space $\prod_{v \text{ place of } K} H^1(K_v, \mathrm{Jac}(C)[2])$. One can then scale the resulting distribution of quotients of 2-Selmer groups by the size of the subgroup $\Sha^1(K, \mathrm{Jac}(C)[2])$ to obtain the desired statistics.

\subsection{Yu's results}

From this subsection and onwards, we assume that $C$ is a hyperelliptic curve of genus $g$ over $K$ which satisfies the following conditions. We note that we consider both cases where $C$ has either a $K$-rational point at infinity or not.
\begin{itemize}
    \item $\mathrm{Jac}(C)[2]$ is a $2g$-dimensional vector space over $\mathbb{F}_2$.
    \item The action of $\mathrm{Gal}(\overline{K}/K)$ on $\mathrm{Jac}(C)[2]$ is irreducible.
    \item We have $\mathrm{Hom}_{\mathrm{Gal}(\overline{K}/K(\mu_p))}(\mathrm{Jac}(C)[2], \mathrm{Jac}(C)[2]) = \mathbb{F}_2$.
    \item We have $H^1(\mathrm{Gal}(K(\mathrm{Jac}(C)[2])/K), \mathrm{Jac}(C)[2]) = 0$.
    \item The following density of prime elements of $K$ is positive:
    \begin{equation*}
        \delta_1 := \lim_{X \to \infty}\frac{\#\{\mathfrak{p} \in \mathcal{O}_K \mid \mathrm{Nm}_\mathbb{Q}^K \mathfrak{p} \leq X, \dim H^1(K_\mathfrak{p}, T) = 2 \}}{\#\{\mathfrak{p} \in \mathcal{O}_K \mid \mathrm{Nm}_\mathbb{Q}^K \mathfrak{p} \leq X \}} > 0.
    \end{equation*}
\end{itemize}
For example, if $C$ is an odd hyperelliptic curve of genus $g$ (i.e. the affine equation defining $C$ is of form $C: y^2 = f(x)$ for some degree $2g+1$ polynomial $f$ over $K$, then all conditions above are satisfied if the splitting field of $f$ is a $S_{2g+1}$ Galois extension over $K$.

Given a square-free element $d \in K^\times / (K^\times)^2$, let $\chi^d \in \mathrm{Hom}(\mathrm{Gal}(\overline{K}/K), \mu_2)$ be a character whose fixed field of $\mathrm{Ker}(\chi^d)$ is $K(\sqrt{d})$. We can write a Weierstrass model for quadratic twist of $C$ by $\chi^d$ as $C^{\chi^d}: dy^2 = f(x)$. We denote by $\chi_v^d$ the restriction of $\chi^d$ to $\mathrm{Hom}(\mathrm{Gal}(\overline{K_v}/K_v), \mu_2)$ for any place $v$. As in Example \ref{example:superelliptic}, we denote by $\Sigma(d)$ the set of places of $K$ consisting of places of bad reduction of $\mathrm{Jac}(C)$ and places dividing $d$.

We have the Weil pairing which is a non-degenerate bilinear pairing
    \begin{equation}
        \mathrm{Jac}(C^{\chi^d})[2] \times \mathrm{Jac}(C^{\chi^d})[2] \to \mu_2,
    \end{equation}
    which induces a non-degenerate Tate local pairing for every place $v$ of $K$:
    \begin{equation*}
        H^1(K_v, \mathrm{Jac}(C)[2]) \times H^1(K_v, \mathrm{Jac}(C)[2]) \to \mathbb{F}_2.
    \end{equation*}
    Here we use the $\mathrm{Gal}(\overline{K}/K)$-equivariant isomorphism
    \begin{equation}
        \mathrm{Jac}(C^{\chi^d})[2] \cong \mathrm{Jac}(C)[2].
    \end{equation}
    We note that there is a subtlety with choosing a suitable form $q_v$ that is non-degenerate and symmetric. Following the notation in \cite[Definition 3.15]{PR12-theta}, we denote by $c_{\mathcal{T}} \in H^1(K, \mathrm{Jac}(C)[2])$ the class of the theta characteristic torsor of $\mathrm{Jac}(C)$. We also define
    \begin{equation}
        \Sha^1(K, \mathrm{Jac}(C)[2]) := \mathrm{Ker} \left( H^1(K, \mathrm{Jac}(C)[2]) \to \prod_{v \text{ place of } K} H^1(K_v, \mathrm{Jac}(C)[2]) \right).
    \end{equation}
    We divide up our analysis into four cases.
    \begin{itemize}
        \item $\deg f \equiv 1 \mod 2$
        \item $\deg f \equiv 0 \mod 2$, and $c_{\mathcal{T}} = 0$.
        \item $\deg f \equiv 0 \mod 2$, and $c_{\mathcal{T}} \in \mathrm{Sel}_2(\mathrm{Jac}(C)/K) \setminus \Sha^1(K, \mathrm{Jac}(C)[2])$.
        \item $\deg f \equiv 0 \mod 2$, and $c_{\mathcal{T}} \in \Sha^1(K, \mathrm{Jac}(C)[2])$.
    \end{itemize}
    In first, second, and fourth cases, we may use \cite[Theorem 3.16]{PR12-theta} to show that for all places $v$ of $K$, the induced pairing on $$H^1(K_v,\mathrm{Jac}(C)[2]) \times H^1(K_v,\mathrm{Jac}(C)[2]) \to \mathbb{F}_2$$ obtained from composing the cup product with the induced map on the second cohomology group by the Weil pairing is a non-degenerate symmetric pairing. In the third case, we may use \cite[Proposition 3.12]{PR12-theta} to show that the induced pairing on $H^1(K_v, \mathrm{Jac}(C)[2])$ is a non-degenerate symmetric pairing at all places $v$ of $K$ such that $\mathrm{Jac}(C)$ has good reduction. For other places $v$, one can still define a non-degenerate quadratic form $q_v: H^1(K_v, \mathrm{Jac}(C)[2]) \to \mathbb{F}_2$ by using \cite[Remark 2.8]{PR12}, where the new quadratic form preserves maximal isotropic subspaces. As in the superelliptic curve case, we endow a global metabolic structure on $\mathrm{Jac}(C)[1-\zeta_p]$ (as in \cite[Definition 3.3]{KMR13}), from which all the consequential results stated in \cite{KMR13, KMR14, Yu16, Yu19} can be utilized.

We extend all the definitions appearing in Section \ref{sec:Selmer} to hyperelliptic curves by allowing the indices $i$ appearing in the definitions of $\delta_i$ and $\mathcal{P}_i$ to satisfy $0 \leq i \leq 2g$, and by replacing $T$ by $Jac(C)[2]$. The twisting data $\alpha$ is given by the map
    \begin{equation}
        \alpha((\chi_v^d)_{v \in \Sigma(d)}) = \prod_{v \in \Sigma(d)} \frac{(\mathrm{Jac}(C_d))(K_v)}{2 (\mathrm{Jac}(C_d))(K_v)}.
    \end{equation}
Given a twisting data $\alpha$ and $\mathfrak{q} \in \mathcal{P}_i$ for $i \geq 3$, one has an injection (but not a surjection)
\begin{equation*}
    \alpha: \frac{\mathrm{Hom}_{ram}(\mathrm{Gal}(\overline{K}_q/K_q), \mu_p)}{\mathrm{Aut}(\mu_p)} \to \mathcal{H}_{ram}(q_v).
\end{equation*}
Then one has
\begin{equation}
    \mathrm{Sel}_2(\mathrm{Jac}(C^{\chi^d})/K) = \mathrm{Sel}(\mathrm{Jac}(C)[2], \chi^d : \alpha).
\end{equation}

With relevant generalizations of definitions in place, we can summarize key results and insights from \cite{Yu19} which will be of great importance for this manuscript. We note that Yu only focused on the case where $C$ has a $K$-rational point at infinity. Nevertheless, parts of Yu's results presented in this subsection can be generalized to any hyperelliptic curve. Instead of considering the local Selmer structure $\mathrm{Sel}(\mathrm{Jac}(C)[2], \omega : \alpha)$, one can instead consider the following quotient of the local Selmer structure
\begin{equation*}
    \overline{\mathrm{Sel}}(\mathrm{Jac}(C)[2], \omega : \alpha) := \frac{\mathrm{Sel}(\mathrm{Jac}(C)[2], \omega : \alpha)}{\Sha^1(K, \mathrm{Jac}(C)[2])}
\end{equation*}
regardless of whether $C$ has either a $K$-rational point at infinity or not, as long as the aformentioned conditions on $C$ holds. Note that regardless of the choice of $\omega$, every local Selmer structure contains a fixed invariant subspace $\Sha^1(K, \mathrm{Jac}(C)[2])$. The results presented below are adaptations of some results from \cite{Yu19} which will be of relevance for our results.

\begin{lemma}[Lemma 5.5 of \cite{Yu19}] 
\label{Yu:lemma5.5}
    Let $\omega \in \Omega_\mathfrak{d}$ and $\mathfrak{q} \nmid \mathfrak{d}$. Let $\mathfrak{q} \in \mathcal{P}_n$ and $\omega' \in \Omega_\mathfrak{\mathfrak{dq}}$ such that $\omega'$ has identical local characters at every place but $\mathfrak{q}$. Then
    \begin{itemize}
        \item $\dim_{\mathbb{F}_2} \overline{\mathrm{Sel}}(\mathrm{Jac}(C)[2],\omega' : \alpha) - \dim_{\mathbb{F}_2} \overline{\mathrm{Sel}}(\mathrm{Jac}(C)[2], \omega : \alpha) = \dim_{\mathbb{F}_2}(\alpha_\mathfrak{q}(\omega'_\mathfrak{q}) \cap V_\mathfrak{q}(\omega)) - \dim_{\mathbb{F}_2}(V_\mathfrak{q}(\omega) \cap H^1_{ur}(K_\mathfrak{q},\mathrm{Jac}(C)[2]))$
        \item $\dim_{\mathbb{F}_2} \overline{\mathrm{Sel}}(\mathrm{Jac}(C)[2], \omega' : \alpha) - \dim_{\mathbb{F}_2} \overline{\mathrm{Sel}}(\mathrm{Jac}(C)[2], \omega: \alpha) \equiv n \text{ mod } 2$.
    \end{itemize}
\end{lemma}

Recall that given $\omega \in \Omega_\mathfrak{d}$ and $\mathfrak{q} \not\in \Sigma(\mathfrak{d})$, we have
\begin{equation}
    t(\omega, \mathfrak{q}) := \dim_{\mathbb{F}_2} \left( \mathrm{Image} \left( \overline{\mathrm{Sel}}(\mathrm{Jac}(C)[2], \omega: \alpha) \to H^1_{ur}(K_\mathfrak{q}, \mathrm{Jac}(C)[2]) \right) \right).
\end{equation}
Yu also shows that there is an explicit Chebotarev condition that determines $t(\omega, \mathfrak{q})$, whose argument still generalizes well to principally polarized abelian varieties as well.
\begin{theorem}[Proposition 6.4 of \cite{Yu19}] \label{Yu:prop6.4}
    Fix $\omega \in \Omega_\mathfrak{d}$. We abbreviate $\dim_{\mathbb{F}_2} \overline{\mathrm{Sel}}(Jac(C)[2], \omega : \alpha) = r$, and we let
    \begin{equation}
        E_{n,i,r} := (2^{-r})^{n-i} \cdot \prod_{h=0}^{i-1} \left( 1 - 2^{-r + h} \right) \cdot \prod_{m=1}^{n-i} \frac{2^{i+m}-1}{2^m-1}.
    \end{equation}
    Then for every $Y > \mathrm{Nm}_{\mathbb{Q}}^K(\mathfrak{d})$ and every $X > \mathcal{L}(Y)$ satisfying some Chebotarev growth conditions, we have
    \begin{equation}
        \left| \frac{\# \{\mathfrak{q} \in \mathcal{P}_n(X) : \mathfrak{q} \nmid \mathfrak{d}, \; t(\omega, \mathfrak{q}) = i\}}{\#\{\mathfrak{q} \in \mathcal{P}_n(X) : \mathfrak{q} \nmid \mathfrak{d}\}} - E_{i,n,r} \right| < \frac{1}{9gY}.
    \end{equation}
\end{theorem}
\begin{proof}
    The quantities $E_{n,i,r}$ are obtained from Chebotarev conditions over the field
\begin{equation}
    F_{\mathfrak{d},\omega} := \text{ Fixed Field of } \bigcap_{c \in \overline{\mathrm{Sel}}(\mathrm{Jac}(C)[2], \omega : \alpha)} \mathrm{Ker} \left(res(c)\right),
\end{equation}
and $res: H^1(K,\mathrm{Jac}(C)[2]) \to \mathrm{Hom}(\mathrm{Gal}(\overline{K}/K(\mathrm{Jac}(C)[2]), \mathrm{Jac}(C)[2])^{\mathrm{Gal}(K(\mathrm{Jac}(C)[2])/K)}$. Provided below are the subsets of $\mathrm{Gal}(F_{\mathfrak{d},\omega}/K)$ which determine the probabilities $E_{n,i,r}$:
\begin{align}
    \begin{split}
        S_{n,r} &:= \{ \sigma \in \mathrm{Gal}(F_{\mathfrak{d},\omega}/K) : \dim_{\mathbb{F}_2}(\mathrm{Jac}(C)[2]/(\sigma-1)\mathrm{Jac}(C)[2]) = n, \dim_{\mathbb{F}_2} \overline{\mathrm{Sel}}(Jac(C)[2], \omega : \alpha) = r\}, \\
        S'_{n,i,r} &:= \{ \sigma \in S_{n,r} : \dim_{\mathbb{F}_2} \mathrm{Im}\left( \mathrm{loc}_\sigma: \overline{\mathrm{Sel}}(\mathrm{Jac}(C)[2], \omega : \alpha ) \to \mathrm{Jac}(C)[2]/(\sigma-1)\mathrm{Jac}(C)[2] \right) = i\},\\
        E_{n,i,r} &:= \frac{\# S'_{n,i,r}}{\# S_{n,r}}.
    \end{split}
\end{align}

We note that the field $F_{\mathfrak{d}, \omega}$ satisfies the following properties.
\begin{itemize}
    \item The field $F_{\mathfrak{d},\omega}/K$ is unramified away from $\Sigma(\mathfrak{d})$, where $\Sigma$ consists of places lying above $2$ and places of bad reduction of $C$.
    \item $\mathrm{Gal}(F_{\mathfrak{d},\omega}/K(\mathrm{Jac}(C)[2])) \cong \mathrm{Jac}(C)[2]^{\dim_{\mathbb{F}_2} \overline{\mathrm{Sel}}(\mathrm{Jac}(C)[2], \omega: \alpha)}$.
\end{itemize}
\end{proof}
When $\mathfrak{q} \in \mathcal{P}_1$, the above theorem shows that the operators $\mathbf{M}_L = [2_{r,s}]$ (obtained by setting $p = 2$) governs the variation of dimensions $ \dim_{\mathbb{F}_2} \overline{\mathrm{Sel}}(\mathrm{Jac}(C)[2], \omega' : \alpha) - \dim_{\mathbb{F}_2} \overline{\mathrm{Sel}}(\mathrm{Jac}(C)[2], \omega: \alpha)$.
\begin{theorem}[Theorem 6.9 and Remark 6.10 of \cite{Yu19}.]
    We use notations as in Theorem \ref{Yu:prop6.4}. Denote by $\mu_\omega$ the probability distribution over $\mathbb{Z}_{\geq 0}$ defined as
    \begin{equation*}
        \mu_\omega(j) := \begin{cases}
            1 &\text{ if } j = \dim_{\mathbb{F}_2} \overline{\Sel}(\mathrm{Jac}(C)[2], \omega: \alpha), \\
            0 &\text{ otherwise}.
        \end{cases}
    \end{equation*}
    Then for every $Y > \mathrm{Nm}_{\mathbb{Q}}^K(\mathfrak{d})$ and every $X > \mathcal{L}(Y)$ satisfying some Chebotarev growth conditions,
    \begin{equation*}
        \left| \frac{\# \{\mathfrak{q} \in \mathcal{P}_1(X) : \mathfrak{q} \nmid \mathfrak{d}, \; \dim_{\mathbb{F}_2} \overline{\Sel}(\mathrm{Jac}(C)[2], \omega' : \alpha) - \dim_{\mathbb{F}_2} \overline{\Sel}(\mathrm{Jac}(C)[2], \omega: \alpha) = i\}}{\#\{\mathfrak{q} \in \mathcal{P}_1(X) : \mathfrak{q} \nmid \mathfrak{d}\}} - \mathbf{M}_L \mu_\omega \right| < \frac{1}{9gY}.
    \end{equation*}
\end{theorem}

\subsection{Improvements using ERH}

Yu's result \cite{Yu19} makes a crucial assumption on the distribution of local Kummer maps (denoted as ``UDRL'' condition) to obtain the desired Poonen-Rains distribution for the dimensions of $2$-Selmer groups of quadratic twist families of odd hyperelliptic curves. We demonstrate that assuming ERH allows us to still obtain the Poonen-Rains distribution without assuming the ``UDRL'' condition. To do so, we prove a number of intermediary results obtained from combining Yu's results and ERH. Before we proceed, we note that we use the abbreviation $\mathrm{rk}(\omega) = \dim_{\mathbb{F}_2} \overline{\mathrm{Sel}}(Jac(C)[2], \omega : \alpha)$ given $\omega \in \Omega_{\mathfrak{d}}$, in case the notation gets too cluttered.

\begin{proposition} \label{Yu+GRH}
    Assume ERH. Denote by $\omega(\mathfrak{d})$ the number of distinct prime ideals of $\mathfrak{d} \in \mathcal{D}$. Fix $\mathfrak{d} \in \mathcal{D}$. Then for every $\mathfrak{d}' \mid \mathfrak{d}$ and $\omega' \in \Omega_{\mathfrak{d}}$, there exist explicit constants $c > 0$ and $B(C,K)$ depending only on the hyperelliptic curve $C$ and $K$ such that for any $X \geq 2$,
    \begin{align}
    \begin{split}
        & \; \; \; \; \left| \# \{\mathfrak{q} \in \mathcal{P}_n(X) : \mathfrak{q} \nmid \mathfrak{d'}, \; t(\omega', \mathfrak{q}) = i\} - E_{i,n,\mathrm{rk}(\omega')}^{\mathfrak{d'},\omega'} \cdot \#\{\mathfrak{q} \in \mathcal{P}_n(X) : \mathfrak{q} \nmid \mathfrak{d}'\} \right| \\
        &< c \cdot X^{1/2} \cdot B(C,K) \cdot [K:\mathbb{Q}] \cdot 2^{4 g^2 \omega(\mathfrak{d})} \cdot (\log \mathrm{Nm}_\mathbb{Q}^K \mathfrak{d} + \log X),
    \end{split}
    \end{align}
    where $g$ is the genus of $C$ and
    \begin{equation}
        E_{i,n,\mathrm{rk}(\omega')}^{\mathfrak{d'},\omega'} := (2^{-\mathrm{rk}(\omega')})^{n-i} \cdot \prod_{h=0}^{i-1} \left( 1 - 2^{-\mathrm{rk}(\omega') + h} \right) \cdot \prod_{m=1}^{n-i} \frac{2^{i+m}-1}{2^m-1}.
    \end{equation}
\end{proposition}
\begin{proof}
   Any field $F_{\mathfrak{d}', \omega'}$ that one needs to construct to compute the desired probability $E_{i,n,\mathrm{rk}(\omega')}^{\mathfrak{d'},\omega'}$ satisfies the following conditions: 
\begin{itemize}
    \item The degree of $F_{\mathfrak{d}', \omega'}$ over $K$ is uniformly bounded by 
    \begin{equation}
        [F_{\mathfrak{d}',\omega'}:K] \leq B_1(C,K) \cdot (2^{2g})^{2g \cdot \omega(\mathfrak{d}')} \leq B_1(C,K) \cdot 2^{4g^2 \omega(\mathfrak{d})},
    \end{equation}
    where $B_1(C,K) > 0$ is some fixed constant depending only on $C$ and $K$.
    \item The number of ramified places of $F_{\mathfrak{d}',\omega'}$ over $K$ is uniformly bounded by $B_2(C,K) + \omega(\mathfrak{d})$, where $B_2(C,K)$ is some fixed constant depending only on $C$ and $K$.
    \item The discriminant of $F_{\mathfrak{d}', \omega'}$ satisfies
    \begin{equation}
        |\mathrm{Disc}(F_{\mathfrak{d}',\omega'})| \leq (\mathrm{Nm}_\mathbb{Q}^K \mathfrak{d}')^{[F_{\mathfrak{d}', \omega'}:\mathbb{Q}]} \leq (\mathrm{Nm}_\mathbb{Q}^K \mathfrak{d})^{[K:\mathbb{Q}] \cdot B_3(C,K) \cdot 2^{4g^2 \omega(\mathfrak{d})}},
    \end{equation}
    where $B_3(C,K)$ is a constant depending on $B_1(C,K)$ and $B_2(C,K)$.
\end{itemize}
By applying Theorem \ref{thm:effective_chebotarev}, proof of Theorem \ref{Yu:prop6.4}, and the three observations above, we have
\begin{align} \label{eq:effectiveChebotarevGRH2}
\begin{split}
    & \; \; \; \; \left| \# \{\mathfrak{q} \in \mathcal{P}_n(X) : \mathfrak{q} \nmid \mathfrak{d'}, \; t(\omega', \mathfrak{q}) = i\} - E_{i,n,\mathrm{rk}(\omega')}^{\mathfrak{d'},\omega'} \cdot \#\{\mathfrak{q} \in \mathcal{P}_n(X) : \mathfrak{q} \nmid \mathfrak{d}'\} \right| \\
    &\leq c \cdot X^{1/2} \cdot B(C,K) \cdot [K:\mathbb{Q}] \cdot 2^{4 g^2 \omega(\mathfrak{d})} \cdot (\log \mathrm{Nm}_\mathbb{Q}^K + \log X),
\end{split}        
\end{align}
where we take $B(C,K) := \max(B_1(C,K), B_2(C,K), B_3(C,K))$.
\end{proof}

We can now combine Lemma \ref{Yu:lemma5.5}, Proposition \ref{Yu+GRH}, and the fan structure $I_K^{Fan}(X)$ to obtain the following result that the first moments of the dimension of local Selmer structures $\overline{\mathrm{Sel}}(\mathrm{Jac}(C)[2], \omega : \alpha)$, where $\omega$ is ramified over at most $2 \log \log X$ many places in $\mathcal{P}_2 \cup \cdots \cup \mathcal{P}_{2g}$ with large enough norm, is uniformly bounded by $2g + 4$.
\begin{proposition} \label{prop:crucial1}
    Assume ERH. Let $X$ be any real number. Fix $\mathfrak{d} \in \mathcal{D}$ such that $\mathrm{Nm}_\mathbb{Q}^K\mathfrak{d} \leq X$ and $\omega(\mathfrak{d}) \leq 2 \log \log X$. Let $f: \mathbb{R} \to [1,\infty)$ be any function which satisfies
    \begin{equation}
        f(X) \gg (\log X)^{16 g^2 \log 2 + 3}
    \end{equation}
    for every $X \in \mathbb{R}$. Then for any fixed $\omega \in \Omega_\mathfrak{d}$ such that $\mathrm{rk}(\omega) \geq 2g+4$ and $2 \leq n \leq 2g$, we have
    \begin{align}
        \begin{split}
            \frac{\sum_{\mathfrak{q} \in \mathcal{P}_n(f(X)), \mathfrak{q} \nmid \mathfrak{d}} \sum_{\omega' \in \eta_{\omega,\mathfrak{q}}^{-1}(\omega)} 2^{\mathrm{rk}(\omega')}}{\#\{\mathfrak{q} \in \mathcal{P}_n(f(X)), \mathfrak{q} \nmid \mathfrak{d}\}} < 2^{\mathrm{rk}(\omega) - n + 1} + O\left( \frac{(\log X)^{8 g^2 \log 2 + 1}}{f(X)^{\frac{1}{2}}} \right),
        \end{split}
    \end{align}
    where the implicit constant of the error term depends only on $A$ and $K$.
\end{proposition}
\begin{proof}
    By applying Proposition \ref{Yu+GRH} with two conditions on $\mathfrak{d}$, we have
    \begin{align}
    \begin{split}
        & \; \; \; \; \left| \# \{\mathfrak{q} \in \mathcal{P}_n(f(X)) : \mathfrak{q} \nmid \mathfrak{d}, \; t(\omega, \mathfrak{q}) = i\} - E_{i,n,\mathrm{rk}(\omega)}^{\mathfrak{d},\omega} \cdot \#\{\mathfrak{q} \in \mathcal{P}_n(f(X)) : \mathfrak{q} \nmid \mathfrak{d}\} \right| \\
        &< f(X)^{1/2} \cdot C(C,K) \cdot [K:\mathbb{Q}] \cdot (\log X)^{8g^2 \log 2 + 1}
    \end{split}
    \end{align}
    for some explicit constant $C(C,K)$. By \cite[Lemma 5.5]{Yu19}, we have
\begin{equation} \label{eqn:crucial1}
    - t(\omega, \mathfrak{q}) \leq \dim_{\mathbb{F}_2} \overline{\mathrm{Sel}}(\mathrm{Jac}(C)[2],\alpha: \omega') - \dim_{\mathbb{F}_2} \overline{\mathrm{Sel}}(\mathrm{Jac}(C)[2],\alpha: \omega) \leq n - 2t(\omega, \mathfrak{q}).
\end{equation}
    Throughout the proof, we abbreviate $\mathrm{rk}(\omega) := r$. Then we have
\begin{align} \label{eq:twistbound}
\begin{split}
    & \; \; \; \; \frac{\sum_{\mathfrak{q} \in \mathcal{P}_n(f(X)), \mathfrak{q} \nmid \mathfrak{d}} \sum_{\omega' \in \eta_{\omega,\mathfrak{q}}^{-1}(\omega)} 2^{\mathrm{rk}(\omega')}}{\#\{\mathfrak{q} \in \mathcal{P}_n(f(X)) : \mathfrak{q} \nmid \mathfrak{d}\}} \\
    &\leq \sum_{i=0}^n 2^{r + n - 2i} \cdot E_{n,i,r} + O\left( \frac{(\log X)^{8 g^2 \log 2 + 1}}{f(X)^{\frac{1}{2}}} \right) \\
    &= 2^r \cdot \left(\sum_{i=0}^n 2^{-(r-1)(n-i)-i} \cdot \prod_{h=0}^{i-1} \left( 1 - 2^{-r + h} \right) \cdot \prod_{m=1}^{n-i} \frac{2^{i+m}-1}{2^m-1}\right) + O\left( \frac{(\log X)^{8 g^2 \log 2 + 1}}{f(X)^{\frac{1}{2}}} \right) \\
    &= 2^r \cdot \sum_{i=0}^n 2^{-i} \cdot \prod_{h=0}^{i-1} \left( 1 - 2^{-r + h} \right) \cdot \prod_{m=1}^{n-i} \frac{2^{i+m}-1}{2^{r-1} \cdot (2^m-1)} + O\left( \frac{(\log X)^{8 g^2 \log 2 + 1}}{f(X)^{\frac{1}{2}}} \right)
\end{split}
\end{align}
Because one has $r \geq 2g + 4$ and $i \leq n \leq 2g$, we have
\begin{equation*}
    \prod_{h=0}^{i-1} (1 - 2^{-r + h}) \leq 1.
\end{equation*}
Furthermore, using the identity that holds for any $m \geq 1$,
\begin{equation*}
    \frac{2^{i+m} - 1}{2^m - 1} < 2^{i+ 1},
\end{equation*}
we can obtain the following inequality:
\begin{equation*}
    \prod_{h=0}^{i-1} (1 - 2^{-r + h}) \cdot \prod_{m=1}^{n-i} \frac{2^{i+m}-1}{2^{r-1} \cdot (2^m - 1)} < \prod_{m=1}^{n-i} 2^{-r + i + 2}.
\end{equation*}
Plugging in the inequality to the equation gives
\begin{align*}
    (\ref{eq:twistbound})
    &< 2^r \cdot \sum_{i=0}^n 2^{-i} \cdot \prod_{m=1}^{n-i} 2^{-r+i+2} + O\left( \frac{(\log X)^{8 g^2 \log 2 + 1}}{f(X)^{\frac{1}{2}}} \right) \\
    &= 2^r \cdot \sum_{i=0}^n 2^{(n-i)(-r+i+2) - i} + O\left( \frac{(\log X)^{8 g^2 \log 2 + 1}}{f(X)^{\frac{1}{2}}} \right) \\
    &< 2^r \cdot \sum_{i=0}^n 2^{(n-i)(-r+n+2) - i} + O\left( \frac{(\log X)^{8 g^2 \log 2 + 1}}{f(X)^{\frac{1}{2}}} \right) \\
    &= 2^r \cdot 2^{-n(r-n-2)} \cdot \sum_{i=0}^n 2^{i(r-(n+3))} + O\left( \frac{(\log X)^{8 g^2 \log 2 + 1}}{f(X)^{\frac{1}{2}}} \right).
\end{align*}

Because $0 \leq n \leq 2g$, we have as long as $r \geq 2g + 4$
\begin{align} \label{eqn:crucial2}
\begin{split}
   \frac{\sum_{\mathfrak{q} \in \mathcal{P}_n(f(X)), \mathfrak{q} \nmid \mathfrak{d}} \sum_{\omega' \in \eta_{\omega,\mathfrak{q}}^{-1}(\omega)} 2^{\mathrm{rk}(\omega')}}{\#\{\mathfrak{q} \in \mathcal{P}_n(f(X)) : \mathfrak{q} \nmid \mathfrak{d}\}} &< 2^{r-n} \cdot \frac{1}{1 - 2^{-(r-n-3)}} + O\left( \frac{(\log X)^{8 g^2 \log 2 + 1}}{f(X)^{\frac{1}{2}}} \right) \\
   &\leq 2^{r-n + 1} + O\left( \frac{(\log X)^{8 g^2 \log 2 + 1}}{f(X)^{\frac{1}{2}}} \right).
\end{split}
\end{align}
\end{proof}

\subsection{Distribution of 2-Selmer groups}

Using the previous two subsections, we compute the distribution of 2-Selmer groups of Jacobians of quadratic twist families of hyperelliptic curves over the fan structure $\mathcal{B}_K(X)$. We first introduce a number of notations. 
\begin{definition}
Given a fixed finite set of places $\Sigma$ containing $\Sigma_{\mathrm{Jac}(C)[2]}$, denote by $\mathrm{E}_1^+(n)$ and $\mathrm{E}_1^-(n)$ the probability distributions
\begin{align}
    \begin{split}
        \mathrm{E}^+_1(n) &:= \frac{\#\{\chi \in \mathcal{C}_2(1,X) : \dim_{\mathbb{F}_2} \overline{\mathrm{Sel}}(\mathrm{Jac}(C)[2], \chi : \alpha) = n, \prod_{v \in \Sigma}\chi_v(\Delta) = 1\}}{\# \{\chi \in \mathcal{C}_2(1,X) : \prod_{v \in \chi_v} \chi_v(\Delta) = 1\}}, \\
                \mathrm{E}^-_1(n) &:= \frac{\#\{\chi \in \mathcal{C}_2(1,X) : \dim_{\mathbb{F}_2} \overline{\mathrm{Sel}}(\mathrm{Jac}(C)[2], \chi : \alpha) = n, \prod_{v \in \Sigma}\chi_v(\Delta) = -1\}}{\# \{\chi \in \mathcal{C}_2(1,X) : \prod_{v \in \chi_v} \chi_v(\Delta) = -1\}}.
    \end{split}
\end{align}
\end{definition}

We state the statement regarding the distribution of 2-Selmer groups to be proven in this subsection.

\begin{theorem}\label{thm:hyperelliptic}
Assume ERH. Let $C: y^2 = f(x)$ be a hyperelliptic curve over a number field $K$ such that the splitting field of $f$ is a $S_{\deg f}$ extension over $K$.

\begin{itemize}
    \item Suppose $\Sha^1(K, Jac(C)[2]) = 0$. Then there exists an explicitly computable constant $\delta(C/K) \in [-1/2, 1/2]$ depending only on $C$ and $K$ such that
    \begin{align}
    \begin{split}
        & \hspace{15pt} \lim_{X \to \infty} \frac{\#\{\chi \in \mathcal{C}_K(X) : \dim_{\mathbb{F}_2} \mathrm{Sel}_2(Jac(C^\chi)/K) = n\}}{\# \mathcal{C}_K(X)} \\
        &= \left( \frac{1}{2} + \delta(C/K) \right) E^{+}(n) + \left( \frac{1}{2} - \delta(C/K) \right) E^-(n).
    \end{split}
    \end{align}
    \item Suppose $\Sha^1(K, Jac(C)[2]) = \mathbb{Z}/2\mathbb{Z}$. Then there exists an explicitly computable constant $\delta(E/K) \in [-1/2, 1/2]$ depending only on $C$ and $K$ such that
    \begin{align}
    \begin{split}
        & \hspace{15pt} \lim_{X \to \infty} \frac{\#\{\chi \in \mathcal{C}_K(X) : \dim_{\mathbb{F}_2} \mathrm{Sel}_2(Jac(C^\chi)/K) = n + 1\}}{\# \mathcal{C}_K(X)} \\
        &= \left( \frac{1}{2} + \delta(C/K) \right) E^{+}(n) + \left( \frac{1}{2} - \delta(C/K) \right) E^-(n).
    \end{split}
    \end{align}
\end{itemize}
\end{theorem}
\begin{proof}
The idea of the proof is an adaptation of the proof of Theorem \ref{thm:fanB} for the case $p = 2$. The key change in the argument revolves around finding a suitable order of prime ideal factors in which we twist the Selmer structure $\overline{\mathrm{Sel}}(\mathrm{Jac}(C)[2], \chi : \alpha)$. 

Given a fixed $m$, recall the subset of ideals $\mathcal{S}(X)$ of $\mathcal{O}_K$ from Theorem \ref{thm:fanB}. Unlike in Theorem \ref{thm:fanB}, however, we will not use the conventional fan structure $\mathcal{D}_{m,k,X}$ to define the fan structure $I_K^{Fan}(X)$ for the sake of simplifying the argument of the proof. Instead, we will use the relaxed fan structure $\bigsqcup_{a=1}^{\lfloor \log \log \log X \rfloor}\overline{\mathcal{D}}_{m,k,X,2^{-a}}$ appearing in Remark \ref{remark:parametrize_fan}. With this substituted fan structure, we can substitute the definition of $I_K^{Fan}(X)$ to be used in this proof to be
\begin{equation}
    I_K^{Fan}(X) := \bigsqcup_{m = \lceil c_1 \log \log X \rceil}^{\lfloor \left( c_2 + \frac{1}{\log \log \log X} \right) \log \log X \rfloor} \left( \bigsqcup_{k=0}^{2m} \bigsqcup_{a=1}^{\lfloor \log \log \log X \rfloor} \overline{\mathcal{D}}_{m,k,X,2^{-a}} \right).
\end{equation}
We borrow the notations $\mathcal{S}(X), I_K^{Fan}(X, \mathfrak{S}), \mathbb{R}[\mathbf{M}_L], \Phi_{even}, \Phi_{odd}$ from the proof of Theorem \ref{thm:fanB}.

Recall that we can rewrite the prime ideal factorization of $\mathfrak{d} \in I_K^{Fan}(X)$ as
\begin{equation*}
    \mathfrak{d} = \prod_{i=1}^{m_s} \mathfrak{q}_i \cdot \prod_{j=1}^{m_l} \mathfrak{p}_j.
\end{equation*}
In addition, we will specify the prime factors $\mathfrak{p}_j$ as
\begin{equation*}
    \prod_{j=1}^m \mathfrak{p}_j = \prod_{j_0=1}^{m_0} \mathfrak{p}_{0,j_0} \cdot \prod_{j_1 = 1}^{m_1} \mathfrak{p}_{1,j_1} \cdot \prod_{j_2 = 1}^{m_2} \mathfrak{p}_{2,j_2} \cdots \prod_{j_{2g} = 1}^{m_{2g}} \mathfrak{p}_{2g,j_{2g}}
\end{equation*}
where we have $\sum_{k=0}^{2g} m_{k} = m_l$, and $\mathfrak{p}_{k,j_k} \in \mathcal{P}_k$ for each $0 \leq k \leq 2g$ and all $1 \leq j_k \leq m_k$.

We will rearrange the order of prime ideal factors as
\begin{equation*}
    \mathfrak{d} = \prod_{i=1}^{m_s} \mathfrak{q}_i \cdot \left( \prod_{k = 2}^{2g} \prod_{j_k = 1}^{m_k} \mathfrak{p}_{k,j_k} \right) \cdot \left( \prod_{k=0}^1 \prod_{j_k=1}^{m_k} \mathfrak{p}_{k,j_k} \right).
\end{equation*}
We consider the following subset of ideals of $\mathcal{O}_K$:
\begin{equation}
    \mathcal{W}(X) := \left\{ \mathfrak{W} \subset I_K^{Fan}(X) : \mathfrak{q} \mid \omega(\mathfrak{W}) \implies \mathrm{Nm}_\mathbb{Q}^K \mathfrak{q} \geq (\log X)^A \text{ and } \mathfrak{q} \in \cup_{k=2}^{2g} \mathcal{P}_k \right\}.
\end{equation}
Given an ideal $\mathfrak{S} \in \mathcal{S}(X)$, $\mathfrak{W} \in \mathcal{W}(X)$, we consider the following subset of $I_K^{Fan}(X)$:
\begin{equation}
    I_K^{Fan}(X, \mathfrak{S}, \mathfrak{W}) := \left\{ \mathfrak{d} \; : \;  \prod_{i=1}^{m_s} \mathfrak{q}_i = \mathfrak{S}, \; \prod_{k = 2}^{2g} \prod_{j_k = 1}^{m_k} \mathfrak{p}_{k,j_k} = \mathfrak{W} \right\}.
\end{equation}
By definition we have
\begin{equation}
    I_K^{Fan}(X, \mathfrak{S}) = \bigsqcup_{\mathfrak{W} \in \mathcal{W}(X)} I_K^{Fan}(X, \mathfrak{S}, \mathfrak{W}),
\end{equation}
and the decomposition of the fan structure $\mathcal{B}_K(X)$ as
\begin{equation*}
    \mathcal{B}_K(X) := \bigsqcup_{\mathfrak{S} \in \mathcal{S}(X)} \bigsqcup_{\mathfrak{W} \in \mathcal{W}(X)} \bigsqcup_{\mathfrak{d} \in I_K^{Fan}(X, \mathfrak{S}, \mathfrak{W})} C_p(\mathfrak{d}).
\end{equation*}

Now we proceed as in \textbf{Case B} in the proof of Theorem \ref{thm:fanB}. We consider subsets of local characters $\bigsqcup_{\mathfrak{W} \in \mathcal{W}(X)} \bigsqcup_{\mathfrak{d} \in I_K^{Fan}(X, \mathfrak{S}, \mathfrak{W})} C_p(\mathfrak{d})$ (or equivalently $\bigsqcup_{\mathfrak{d} \in I_K^{Fan}(X, \mathfrak{S})} C_p(\mathfrak{d})$) for each choice of $\mathfrak{S} \in \mathcal{S}(X)$. We consider for each non-negative integer $n \geq 0$ the probability 
\begin{equation*}
    E_\mathfrak{S}(n) := \frac{\#\{\chi \in \mathcal{C}_p(\mathfrak{S},X) : \dim_{\mathbb{F}_2} \overline{\mathrm{Sel}}(\mathrm{Jac}(C)[2], \chi : \alpha) = n\}}{\# \mathcal{C}_p(\mathfrak{S},X)}.
\end{equation*}
Similar to Theorem \ref{thm:fanB}, by Lemma \ref{Yu:lemma5.5} (i.e. \cite[Lemma 5.5]{Yu19}), there exists some Markov operator $\mathbf{S} := [s_{i,j}]_{i,j \geq 0}$ satisfying $s_i,j = 0$ whenever $|i - j| > 2 \omega(\mathfrak{S})$ and $|i - j| \not\equiv \omega(\mathfrak{S}) \mod 2$ such that
\begin{equation*}
    E_\mathfrak{S}(n) = \begin{cases}
        \mathbf{S}(E_1^+)(n) &\text{ if } \omega(\mathfrak{S}) \equiv 0 \mod 2, \\
        \mathbf{S}(E_1^-)(n) &\text{ otherwise }.
    \end{cases}
\end{equation*}
The lemma also implies that $\rho(E_\mathfrak{S}(n) = \rho(E_1^+(n))$ if $\omega(\mathfrak{S}) \equiv 0 \mod 2$ and $\rho(E_\mathfrak{S}(n)) = 1 - \rho(E_1^-(n))$.

The order of the prime ideal factors in which we twist the Selmer structure $\overline{\mathrm{Sel}}(\mathrm{Jac}(C)[2], \chi : \alpha)$ is given by
\begin{equation*}
    \mathfrak{W}, \mathfrak{p}_{0,1}, \cdots, \mathfrak{p}_{0,m_0}, \mathfrak{p}_{1,1}, \cdots, \mathfrak{p}_{1,m_1}.
\end{equation*}
In particular, we twist the Jacobian of the hyperelliptic curve by prime ideal factors $\mathfrak{p}_{k,j_k}$ for $2 \leq k \leq 2g$ first, and then twist by $\mathfrak{p}_{0,j_0}$ and $\mathfrak{p}_{1,j_1}$ afterwards. At first glance, the dependence of norms among prime ideal factors of $\mathfrak{d}$ may seem to pose as an obstruction to swap the order of prime ideal factors with which we twist the hyperelliptic curve. Nevertheless, the sets $\overline{\mathcal{D}}_{m,k,X,2^{-a}}$ can be well approximated by disjoint unions of $m_l$ many tuples of prime ideals, whose conditions on their norms are independent from each other. Given an ideal $\mathfrak{S} \in \mathcal{S}(X)$, we define the following subset
\begin{equation}
    \overline{\mathcal{D}}_{m,k,X,2^{-a}}^{\mathfrak{S}} := \left\{ \mathfrak{d} \in \overline{\mathcal{D}}_{m,k,X,2^{-a}} \; : \; \prod_{i=1}^{m_l}  \mathfrak{q}_i = \mathfrak{S} \right\}.
\end{equation}
Given an ideal $\mathfrak{S} \in \mathcal{S}(X)$, consider the following $m_l$ tuples of prime ideals, where the notations $T$ and $M_j(2^{-a})$ are defined as in Remark \ref{remark:parametrize_fan}.
\begin{align}
    \begin{split}
        \overline{\mathcal{D}}_{m,k,X,2^{-a}}^{*, \mathfrak{S}} &:= \biggl\{ (\mathfrak{p}_1, \mathfrak{p}_2, \cdots, \mathfrak{p}_{m_l}) \in \left( \prod_{j=1}^{m_l-1} [T, M_j(2^{-a})] \right) \times \left[ \frac{X^{1-2^{-a+1}(1-2^{-m_l})}}{\mathrm{Nm}_{\mathbb{Q}}^K \mathfrak{S}}, \frac{X^{1-2^{-a}(1-2^{-m_l})}}{\mathrm{Nm}_{\mathbb{Q}}^K \mathfrak{S}} \right] \\
        & \hspace{30pt} \; : \; \mathfrak{p}_j \text{ is prime } \forall 1 \leq j \leq m_l \biggr\}.
    \end{split}
\end{align}
Then one obtains for any given $\mathfrak{S} \in \mathcal{S}(X)$ and any small enough $\epsilon > 0$,
\begin{equation*}
    \# \overline{\mathcal{D}}_{m,k,X,2^{-a}}^{\mathfrak{S}} = \# \overline{\mathcal{D}}_{m,k,X,2^{-a}}^{*, \mathfrak{S}} + O_K \left( \frac{\# \overline{\mathcal{D}}_{m,k,X,2^{-a}}^{\mathfrak{S}}}{(M_1(2^{-a}))^{1 - \epsilon}} \right).
\end{equation*}
It follows that given any collection of $m_l$ many Galois extensions $L_1, L_2, \cdots, L_{m_l}$ over $K$, the resulting asymptotic densities obtained from applying Chebotarev density theorem over $L_j/K$ consecutively to prime ideal factors $\mathfrak{p}_j$ for $1 \leq j \leq m_l$ appearing in $\overline{\mathcal{D}}_{m,k,X,2^{-a}}^{\mathfrak{S}}$ are, up to an error term of order $O_K\left( (M_1(2^{-a}))^{-\frac{1}{2} + \epsilon} \right)$ for some small enough $\epsilon > 0$, equal to those obtained from applying Chebotarev density theorem over $L_j/K$ consecutively to prime ideal factors $\mathfrak{p}_j$ for $1 \leq j \leq m_l$ appearing in $\overline{\mathcal{D}}_{m,k,X,2^{-a}}^{*,\mathfrak{S}}$. We observe that each prime ideal factor $\mathfrak{p}_j$ of $\overline{\mathcal{D}}_{m,k,X,2^{-a}}^{*,\mathfrak{S}}$ varies over an interval which is independent from any other interval. Hence, the asymptotic densities obtained over $\overline{\mathcal{D}}_{m,k,X,2^{-a}}^{*,\mathfrak{S}}$ are independent from the ordinal of prime ideal factors $\mathfrak{p}_j$ to which we apply Chebotarev density theorem consecutively. 

We may further subdivide the collection of tuples in $\overline{\mathcal{D}}_{m,k,X,2^{-a}}^{*,\mathfrak{S}}$ into $3^{m_l}$ many subsets. Such subsets are obtained by specifying whether for each $1 \leq j \leq m_l$, the prime ideal $\mathfrak{p}_j$ lies in $\mathcal{P}_0$, $\mathcal{P}_1$, or $\bigsqcup_{k=2}^{2g} \mathcal{P}_k$. Over each subset, we may apply Chebotarev density theorem in the following order. First, we apply the Chebotarev density theorem over the coordinates whose corresponding prime ideal factors satisfy $\mathfrak{p}_j \in \bigsqcup_{k=2}^{2g} \mathcal{P}_k$. Next, we apply the theorem over the coordinates whose corresponding prime ideal factors satisfy $\mathfrak{p}_j \in \mathcal{P}_0$. Lastly, we apply the theorem over rest of the coordinates, where the prime ideal factors satisfy $\mathfrak{p}_j \in \mathcal{P}_1$. By taking weighted average of asymptotic densities over such $3^{m_l}$ many subsets, we shall deduce the desired asymptotic densities over $\overline{\mathcal{D}}_{m,k,X,2^{-a}}^{\mathfrak{S}}$ up to an error term of order $O_K\left( (M_1(2^{-a}))^{-\frac{1}{2} + \epsilon} \right)$ for some small enough $\epsilon > 0$.

We now use Proposition \ref{Yu:prop6.4} (or \cite[Proposition 6.4]{Yu19}) to construct two Markov operators:
\begin{itemize}
    \item $\mathbf{W}_{\text{width }\geq 2}$: models variations of $\overline{\mathrm{Sel}}(\mathrm{Jac}(C)[2], \chi : \alpha)$ with respect to twisting by $\mathfrak{W}$. We note that the definition of $\mathbf{W}_{\text{width } \geq 2}$ depends on the choice of $\mathfrak{W}$, but such a dependence will not affect the statement of the theorem.
    \item $\mathbf{M}_{C} := \frac{\delta_0}{\delta_0 + \delta_1} \cdot \mathbf{Id} + \frac{\delta_1}{\delta_0 + \delta_1} \cdot \mathbf{M}_L$: models variations of $\overline{\mathrm{Sel}}(\mathrm{Jac}(C)[2], \chi : \alpha)$ with respect to twisting by $m_0 + m_1$ many prime ideal factors $\mathfrak{p}_{0,j_0}$ and $\mathfrak{p}_{1,j_1}$. We note that because $\delta_0 > 0$ and $\delta_1 > 0$, the Markov operator $\mathbf{M}_C$ is an irreducible aperiodic Markov operator over the countable state space $\mathbb{Z}_{\geq 0}$. Furthermore, the quantity $\frac{m_0 + m_1}{m_l}$ is a positive number for sufficiently large $X$.
\end{itemize}

Using Proposition \ref{prop:crucial1}, we have that for sufficiently large $X$,
\begin{equation} \label{eq:thm1.3_crucial2}
    \mathbb{E}[\mathbf{W}_{\text{width } \geq 2} \cdot \mathbf{S} (E_1^\pm(n))] < 2g + 6.
\end{equation}
We now use these two inequalities and the geometric ergodicity of the Markov operator $\mathbf{M}_C$ as stated in Theorem \ref{thm:geometric-ergodicity} to rewrite the desired probability distribution
\begin{equation} \label{eq:desired_prob_2}
    \frac{\#\left\{ \chi \in \bigsqcup_{\mathfrak{d} \in I_K^{Fan}(X, \mathfrak{S})}C_p(\mathfrak{d}) : \dim_{\mathbb{F}_2} \overline{\mathrm{Sel}}(\mathrm{Jac}(C)[2], \chi : \alpha) = n\right\}}{\# \bigsqcup_{\mathfrak{d} \in I_K^{Fan}(X, \mathfrak{S})} C_p(\mathfrak{d})}.
\end{equation}
Let $\{Q_m^\pm\}_{m \in \mathbb{Z}}$ be a set of non-negative real numbers  depending on the choice of $X$ that satisfies
\begin{equation*}
    \begin{cases}
        \sum_{m=0}^\infty Q_m^+ + \sum_{m=0}^\infty Q_m^- = 1, \\
        Q_m^\pm = 0 &\text{ if } m < c_1 \log \log X \text{ or } m > c_2 \log \log X.
    \end{cases}
\end{equation*}
Recall that this is the identical construction shown in the proof of Theorem \ref{thm:fanB}). Let $\epsilon \in (0,1/2)$ be a small enough number. If $\omega(\mathfrak{S})$ is even, then we have
\begin{align*}
    (\ref{eq:desired_prob_2}) &= \sum_{m=\lceil c_1 \log \log X \rceil}^{\lfloor c_2 \log \log X \rfloor} \biggl\{ Q_{m_l}^+ \cdot \biggl[ \Phi_{even} \left( \mathbf{M}_C^{m_0 + m_1}\right) \cdot \mathbf{W}_{\text{width } \geq 2} \cdot \mathbf{S} \\
    & \hspace{100pt} + \Phi_{odd} \left( \mathbf{M}_C^{m_0 + m_1} \right) \cdot \mathbf{W}_{\text{width } \geq 2} \cdot \mathbf{S}  \biggr] (E_1^+)(n) \\
    & \hspace{50pt} + Q_{m_l}^- \cdot \biggl[ \Phi_{odd} \left( \mathbf{M}_C^{m_0 + m_1} \right) \cdot \mathbf{W}_{\text{width } \geq 2} \cdot \mathbf{S} \\
    & \hspace{100pt} + \Phi_{even} \left( \mathbf{M}_C^{m_0 + m_1} \right) \cdot \mathbf{W}_{\text{width } \geq 2} \cdot \mathbf{S} \biggr](E_1^-)(n)\biggr\} \\
    &+ O \left( X^{\left(-\frac{1}{2} + \epsilon \right) \cdot \frac{1}{(\log X)^{c_2 \log 2}} \cdot \frac{1}{(\log \log X)^{\log 2}}} \cdot \log \log X \right).
\end{align*}
If $\omega(\mathfrak{S})$ is odd, then we have
\begin{align*}
    (\ref{eq:desired_prob_2}) &= \sum_{m=\lceil c_1 \log \log X \rceil}^{\lfloor c_2 \log \log X \rfloor} \biggl\{ Q_{m_l}^+ \cdot \biggl[ \Phi_{even} \left( \mathbf{M}_C^{m_0 + m_1}\right) \cdot \mathbf{W}_{\text{width } \geq 2} \cdot \mathbf{S} \\
    & \hspace{100pt} + \Phi_{odd} \left( \mathbf{M}_C^{m_0 + m_1} \right) \cdot \mathbf{W}_{\text{width } \geq 2} \cdot \mathbf{S}  \biggr] (E_1^-)(n) \\
    & \hspace{50pt} + Q_{m_l}^- \cdot \biggl[ \Phi_{odd} \left( \mathbf{M}_C^{m_0 + m_1} \right)\cdot \mathbf{W}_{\text{width } \geq 2} \cdot \mathbf{S} \\
    & \hspace{100pt} + \Phi_{even} \left( \mathbf{M}_C^{m_0 + m_1} \right) \cdot \mathbf{W}_{\text{width } \geq 2} \cdot \mathbf{S} \biggr](E_1^+)(n)\biggr\} \\
    &+ O \left( X^{\left(-\frac{1}{2} + \epsilon \right) \cdot \frac{1}{(\log X)^{c_2 \log 2}} \cdot \frac{1}{(\log \log X)^{\log 2}}} \cdot \log \log X \right).
\end{align*}
For both expressions, the error terms are obtained from computing the error term $M_1((\log \log X)^{\log 2})$ and setting $m_l = \lfloor c_2 \log \log X \rfloor$.

Let $\pi_C$ be the probability distribution over $\mathbb{Z}_{\geq 0}$ given by
\begin{equation}
    \pi_C(n) := \left( \frac{1}{2} + \frac{1}{2} \left( \rho(E_1^-) - \rho(E_1^+) \right) \right) E^+(n) + \left( \frac{1}{2} - \frac{1}{2} \left( \rho(E_1^-) - \rho(E_1^+) \right) \right) E^-(n).
\end{equation}
By Theorem \ref{thm:geometric-ergodicity}, there exists an explicitly computable rate of convergence $0 < \gamma_C < 1$ depending only on $C$ and an explicit constant $\beta(C,K)$ depending only on $C$ and $K$ such that for sufficiently large $X$ we have
\begin{align*}
    \sup_{n \in \mathbb{Z}_{\geq 0}}\left| (\ref{eq:desired_prob_2}) - \pi_C(n) \right| < & \beta(C,K) \cdot \max_{k \in \{+, -\}} \left(\mathbb{E}[V(\mathbf{W}_{\text{width } \geq 2} \cdot \mathbf{S}(E_1^k(n)))] \right) \cdot \gamma^{m_0 + m_1} \\
    &+ O \left( X^{\left(-\frac{1}{2} + \epsilon \right) \cdot \frac{1}{(\log X)^{c_2 \log 2}} \cdot \frac{1}{(\log \log X)^{\log 2}}} \cdot \log \log X \right),
\end{align*}
where $V(x) := 2^x$, and given any probability distribution $\mu: \mathbb{Z}_{\geq 0} \to [0,1]$, we define $V(\mu) := \sum_{x =0}^\infty \mu(x) \cdot 2^x$.
Using equation (\ref{eq:thm1.3_crucial2}), we obtain that for sufficiently large $X$ and $\beta'(C,K) := 2^{2g + 6} \cdot \beta(C,K)$,
\begin{equation*}
    \sup_{n \in \mathbb{Z}_{\geq 0}}\left| (\ref{eq:desired_prob_2}) - \pi_C(n) \right| < \beta'(C,K) \cdot \gamma^{m_0 + m_1} + O \left( X^{\left(-\frac{1}{2} + \epsilon \right) \cdot \frac{1}{(\log X)^{c_2 \log 2}} \cdot \frac{1}{(\log \log X)^{\log 2}}} \cdot \log \log X \right).
\end{equation*}
Because $\delta_0 + \delta_1 > 0$ by assumption, we may use effective Chebotarev density theorem with respect to the splitting field of $f(x)$ which defines the equation for $C: y^2 = f(x)$ to obtain that that $m_0 + m_1 = m \cdot (\delta_0 + \delta_1) + o(\log \log X)$. This implies 
\begin{equation*}
    \sup_{n \in \mathbb{Z}_{\geq 0}}\left| (\ref{eq:desired_prob_2}) - \pi_C(n) \right| < \beta'(C,K) \cdot (\log X)^{c_1 \cdot (\delta_0 + \delta_1) \cdot\log \gamma + o(1)} + O \left( X^{\left(-\frac{1}{2} + \epsilon \right) \cdot \frac{1}{(\log X)^{c_2 \log 2}} \cdot \frac{1}{(\log \log X)^{\log 2}}} \cdot \log \log X \right),
\end{equation*}
the right hand side of which converges to $0$ as $X$ grows arbitrarily large, because $0 < \gamma < 1$. By setting $\delta(C/K) := \frac{1}{2} \cdot \left( \rho(E_1^-) - \rho(E_1^+) \right)$, taking disjoint unions of $I_K^{Fan}(X, \mathfrak{S})$ as $\mathfrak{S}$ varies over $\mathcal{S}(X)$, and using Theorem \ref{thm:Erdos-Kac} and Proposition \ref{prop:IfanvsI}, we obtain
\begin{align*}
    & \hspace{15pt} \lim_{X \to \infty} \frac{\#\left\{ \chi \in \mathcal{C}_K(X) : \dim_{\mathbb{F}_2} \overline{\mathrm{Sel}}_2(\mathrm{Jac}(C)[2], \chi : \alpha) = n \right\}}{\# \mathcal{C}_K(X)} \\
    &= \left( \frac{1}{2} + \delta(C/K) \right) E^+(n) + \left( \frac{1}{2} - \delta(C/K) \right) E^-(n).
\end{align*}
We obtain the statement of the theorem by using \cite[Theorem 4.14]{PR12}, where we identify
\begin{equation*}
    \frac{\mathrm{Sel}_2(\mathrm{Jac}(C^\chi)/K)}{\Sha^1(K, \mathrm{Jac}(C)[2])} = \overline{\mathrm{Sel}}(\mathrm{Jac}(C)[2], \chi : \alpha),
\end{equation*}
and $\alpha$ is defined as
\begin{equation*}
    \alpha(\overline{\chi_v})_{v \in \Sigma_C(\chi)} := \prod_{v \in \Sigma_C(\chi)} \frac{\mathrm{Jac}(C^{\chi_v})(K_v)}{2\mathrm{Jac}(C^{\chi_v})(K_v)}.
\end{equation*}
\end{proof}

\begin{remark}
    We may relax the definition of $\overline{\mathcal{D}}_{m,k,X,a}$ as presented in Remark \ref{remark:improve_rate_convergence} to improve the rate of convergence of the limits in Theorem \ref{thm:hyperelliptic} to be of order $O(1/(\log X)^c)$ for some explicitly computable constant $c \in (0,1)$.
\end{remark}

\begin{remark}
    Theorem \ref{thm:hyperelliptic} proves that depending on the cardinality of $\Sha^1(K, \mathrm{Jac}(C)[2])$, the average size of 2-selmer groups of quadratic twist families of a given hyperelliptic curve $C$ converges to either $3$ or $6$ assuming ERH. Our result hence verifies \cite[Conjecture 1.8]{PR12} conditional on ERH. We also note that Theorem \ref{thm:hyperelliptic} verifies Alex Smith's monumental results on the distribution of $2$-Selmer groups of quadratic twist families of hyperelliptic curves (though assuming ERH but with possible exponential improvements on the rate of convergence), see for example \cite[Theorem 2.14, Example 3.7]{Smith2}. On the other hand, Theorem \ref{thm:hyperelliptic} also generalizes Smith's result for quadratic twist families of hyperelliptic curves satisfying $\#\Sha^1(K, \mathrm{Jac}(C)[2]) = 2$ conditional on ERH. As stated in \cite[Case 2.12, Case 2.13]{Smith2}, the quadratic twist families Smith considers require that there are no common Selmer elements lying in every $\mathrm{Sel}_2(\mathrm{Jac}(C^\chi)/K)$.
\end{remark}

\begin{remark}
    Similar to the proof of Theorem \ref{thm:intro_superelliptic}, we can obtain average uniform Mordell-Lang results for quadratic twist families of $C$ assuming ERH. In fact, our result implies that as the genus of the given hyperelliptic curve $C$ grows arbitrarily large and assuming ERH, we can utilize Chabauty-Coleman-Stoll method to compute explicit upper bounds on the nubmer of $K$-rational solutions for all but $O \left(2^{-(g-1)(g-2)/2} \right)$ of quadratic twist families of $C$.
\end{remark}

We conclude this manuscript by proving Theorem \ref{thm:intro_hyperelliptic} in the introduction.

\begin{proof}[Proof of Theorem \ref{thm:intro_hyperelliptic}]
    Let $r_1, r_2$ be the number of real/complex embeddings of $K$. Then we have
    \begin{equation}
        N_2(K,X) = 2^{r_1 + r_2 - 1} \cdot \# \left\{F := \text{ Fixed Field of ker} \chi \; : \; \chi \in \mathcal{C}_2(X) \right\}.
    \end{equation}
    We have
    \begin{equation}
        \mathrm{rk}(\mathrm{Jac}(C)/F) - \mathrm{rk}(\mathrm{Jac}(C)/K) = \dim_{\mathbb{F}_2} \mathrm{Sel}_2(\mathrm{Jac}(C^\chi)/K) - \dim_{\mathbb{F}_2} \Sha^1(K, \mathrm{Jac}(C)[2]).
    \end{equation}
    The result follows from combining the above two equations with Theorem \ref{thm:hyperelliptic} and Definition \ref{def:Markov}.
\end{proof}

\section{Summary} \label{sec:summary}

The overarching strategy to prove Theorems \ref{thm:intro}, \ref{thm:intro_superelliptic}, and \ref{thm:intro_hyperelliptic} is summarized in the table below. The first column states four keywords - Discrete Time, Data, Model, and Asymptotics - which summarizes and encapsulates the procedural flow we use in this manuscript to compute the distribution of desired Selmer groups. The second column ``Arithmetic Statistics'' summarizes each smaller steps performed in the proof of the main theorems. The third column ``Probability Theory'' represents what concepts from probability theory or data analysis are related to each smaller steps performed in the proof of the main theorems. Lastly, the fourth column ``Tools'' gives a summary of tools from analytic number theory and stochastic processes utilized to prove the desired results.

\begin{table}[ht]
    \centering
    \begin{tabular}{|c||c||c|c|}
        \hline
        \textbf{Concept} & \textbf{Arith. Stat.} & \textbf{Prob. Theory} & \textbf{Tools} \\
        \hline
        \hline
        \multirow{4}{*}{\textbf{Discrete Time}} & $\omega(d) = 1, 2, 3, \cdots$ & Time variable & Prime factorization \\
        & $\omega(d) \sim \log \log X$ & Average time & Erd\"os-Kac Theorem \\
        & \textcolor{blue}{$c_1 \leq \frac{\omega(d)}{\log\log X} \leq c_2$} & \textbf{\textcolor{blue}{100\% of time}} & \textbf{\textcolor{blue}{Large deviation estimates}} \\
        & \textcolor{blue}{$I_K^{Fan}(X)$} & \textbf{\textcolor{blue}{Rearrangement}} & \textbf{\textcolor{blue}{Sathe-Selberg, Sieve}} \\
        \hline
        \hline
        \multirow{3}{*}{\textbf{Data}} & $\mathbf{J}_d$ & Time series data & Jacobians \\
        & $\mathbf{J}_d[2] \cong \mathbf{J}[2]$ & Common domain & Galois-equivariance \\
        & $\mathrm{Sel}_2(\mathbf{J}_d/K)$ & Measurements & Selmer groups \\
        \hline
        \hline
        \multirow{3}{*}{\textbf{Model}} & $\mathbf{M}_L$ & Markov operator & Chebotarev density theorem \\
        & weighted $\mathbf{M}_L$ sums & Irreducibility, Aperiodicity & Large $\mathrm{Gal}(K(\mathbf{J}[2])/K)$ \\
        & \textcolor{blue}{$p_1 p_2 \leftrightarrow p_2 p_1$} & \textbf{\textcolor{blue}{Swap operation order}} & \textbf{\textcolor{blue}{Extended RH}} \\
        \hline
        \hline
        \multirow{3}{*}{\textbf{Asymptotics}} & $PR$ & Stationary Distribution & Irreducibility, Aperiodicity \\
        & \textcolor{blue}{$1/(\log X)^c$} & \textbf{\textcolor{blue}{Rate of Convergence}} & \textbf{\textcolor{blue}{Geometric ergodicity}} \\
        & \textcolor{blue}{$\mathbb{E}[\mathbf{W}_{\text{ width} \geq 2}] < \infty$} & \textbf{\textcolor{blue}{Invariance w.r.t. order}} & \textbf{\textcolor{blue}{Large ${\mathrm{Gal}(K(\mathrm{J}[2])/K)}$}}  \\
        \hline
        \hline
    \end{tabular}
    \caption{A table for quadratic twist families of Jacobians of hyperelliptic curves, which summarizes overarching concepts, relevant comparisons between concepts in probability theory and arithmetic statistics, as well as relevant tools utilized to obtain such relations.}
    \label{tab:summary}
\end{table}

We pay particular focus to computing the distribution of 2-Selmer groups of quadratic twist families of Jacobians of hyperelliptic curves over number fields. Here we use the abbreviation $\mathbf{J}_d := \mathrm{Jac}(C_d)$ and $\mathbf{J} := \mathrm{Jac}(C)$. The novel contributions made in this manuscript is written in \textbf{\textcolor{blue}{blue bolded text}}. To obtain Theorems \ref{thm:intro} and \ref{thm:intro_superelliptic}, we can alter the respective abelian varieties, torsion submodules, and Selmer groups as follows.
\begin{itemize}
    \item \textbf{Theorem \ref{thm:intro}}: We substitute the \textbf{Data} and \textbf{Asymptotics} row as in Table \ref{table:thmintro}.
\begin{table}[ht]
    \centering
    \begin{tabular}{|c||c||c|c|}
        \hline
        \textbf{Concept} & \textbf{Prob. Theory} & \textbf{Arith. Stat.} & \textbf{Tools} \\
        \hline
        \hline
        \multirow{3}{*}{\textbf{Data}} & Time series data & $A^{\chi}$ & Weil restriction \\
        & Common domain & $A^{\chi}[1-\sigma] \cong E[p]$ & Galois-equivariance \\
        & Measurements & $\mathrm{Sel}_{1-\sigma}(A^{\chi}/K)$ & Selmer groups \\
        \hline
        \hline
        \multirow{3}{*}{\textbf{Asymptotics}} & $PR$ & Stationary Distribution & Irreducibility, Aperiodicity \\
        & \textcolor{blue}{$1/(\log X)^c$} & \textbf{\textcolor{blue}{Rate of Convergence}} & \textbf{\textcolor{blue}{Geometric ergodicity}} \\
        & \textcolor{blue}{$\mathbb{E}[\mathbf{P}^\pm_{max} \cdot S] < \infty$} & \textbf{\textcolor{blue}{Invariance w.r.t. order}} & \textbf{\textcolor{blue}{Large ${\mathrm{Gal}(K(\mathrm{E}[p])/K)}$}}  \\
        \hline
        \hline
    \end{tabular}
    \caption{A table for rank growths of elliptic curves over cyclic prime extensions.}
    \label{table:thmintro}
\end{table}
    \item \textbf{Theorem \ref{thm:intro_superelliptic}}: We substitute the \textbf{Data} and \textbf{Asymptotics} row as in Table \ref{table:thmsuper}.
    \begin{table}[ht]
    \centering
    \begin{tabular}{|c||c||c|c|}
        \hline
        \textbf{Concept} & \textbf{Prob. Theory} & \textbf{Arith. Stat.} & \textbf{Tools} \\
        \hline
        \hline
        \multirow{3}{*}{\textbf{Data}} & Time series data & $\mathbf{J}_d$ & Jacobians \\
        & Common domain & $\mathbf{J}_d[1-\zeta_p] \cong \mathbf{J}[1-\zeta_p]$ & Galois-equivariance \\
        & Measurements & $\mathrm{Sel}_{1-\zeta_p}(\mathbf{J}_d/K)$ & Selmer groups \\
        \hline
        \hline
        \multirow{3}{*}{\textbf{Asymptotics}} & $PR$ & Stationary Distribution & Irreducibility, Aperiodicity \\
        & \textcolor{blue}{$1/(\log X)^c$} & \textbf{\textcolor{blue}{Rate of Convergence}} & \textbf{\textcolor{blue}{Geometric ergodicity}} \\
        & \textcolor{blue}{$\mathbb{E}[\mathbf{P}^\pm_{max} \cdot S] < \infty$} & \textbf{\textcolor{blue}{Invariance w.r.t. order}} & \textbf{\textcolor{blue}{Large ${\mathrm{Gal}(K(\mathbf{J}[1-\zeta_p])/K)}$}}  \\
        \hline
        \hline
    \end{tabular}
    \caption{A table for twist families of superelliptic curves. Here we abbreviate as before $\mathbf{J}_d := \mathrm{Jac}(C_d)$ and $\mathbf{J} := \mathrm{Jac}(C)$.}
    \label{table:thmsuper}
\end{table}
\end{itemize}

The previous works by Klagsbrun, Mazur, and Rubin \cite{KMR14} and Yu \cite{Yu19} constitute the parts of the table written in black text. In particular, these works allow us to obtain desired distributions of $\mathrm{Sel}_{1-\sigma}(A^{\chi}/K)$ and $\mathrm{Sel}_{1-\zeta_p}(\mathrm{Jac}(C_d)/K)$, but where $d$ is non-canonically ordered by the fan structure. 

There are three contributions this manuscript proposes.
\begin{itemize}
    \item We can utilize large deviation estimates, Sathe-Selberg theorem, and Selberg's sieve to order $d$ based on the norm of its radical of the ideal generated by $d$.
    \item We can utilize the extended Riemann hypothesis to swap the order of prime ideals used to twist Selmer structures of a given fixed Galois module of any even dimension satisfying some mild conditions.
    \item Using geometric ergodicity of Markov operators $\mathbf{M}_L$, we compute the rate of convergence to the Poonen-Rains distribution, as well as guarantee that swapping the order of twists by prime ideals does not alter the asymptotic distribution.
\end{itemize}
By doing so, we prove Theorems \ref{thm:intro}, \ref{thm:intro_superelliptic} under assuming ERH, where $d$ is ordered by the norm of its radical of the ideal generated by $d$. Furthermore, we obtain Theorem \ref{thm:intro_hyperelliptic} under assuming ERH, without assuming any heuristic assumptions on the distribution of local Kummer maps defining 2-Selmer groups of quadratic twist families of hyperelliptic curves. 

An upcoming work by the authors will demonstrate that one can use analogous techniques to compute the distribution of 2-Selmer groups of quadratic twist families of Jacobians of hyperelliptic curves satisfying some mild conditions over global function fields $K = \mathbb{F}_q(t)$ of characteristic coprime to $6$. We expect that one can obtain the desired Poonen-Rains distribution (up to contributions from $\Sha^1(\mathrm{Jac}(C)[2])$) without any conditions on the size of the constant field of $K$. We hope such results will serve as a complement to the groundbreaking result by Aaron Landesman and Ishan Levy \cite{LL25}, whose work can be used to compute the distribution of odd Selmer groups of quadratic twist families of Jacobians of hyperelliptic curves satisfying some mild conditions over $K$, where the size of the constant field is at least some explicitly computable constant. Further generalizations the authors hope to understand include obtaining statements on rank growths of principally polarized abelian varieties over cyclic prime extensions of global fields, assuming appropriate versions of ERH.

\newpage

\nocite{*}
\bibliographystyle{alpha}
\bibliography{main}

\begin{thebibliography}{KMR14}

\bibitem[Deb19]{Debaene19}
Korneel Debaene.
\newblock Explicit counting of ideals and a {Brun-Titchmarsh} inequality for the {Chebotarev} density theorem.
\newblock {\em International Journal of Number Theory}, 15(5):883--905, 2019.

\bibitem[DGH21]{DGH21}
Vesselin Dimitrov, Ziyang Gao, and Philipp Habegger.
\newblock Uniformity in {Mordell–Lang} for curves.
\newblock {\em Annals of Mathematics}, 194:237--298, 2021.

\bibitem[EL23]{EL23}
Jordan Ellenberg and Aaron Landesman.
\newblock Homological stability for generalized {Hurwitz} spaces and {Selmer} groups in quadratic twist families over function fields, 2023.
\newblock arXiv:2310.16286.

\bibitem[GL22]{GarciaLee22}
Stephan~Ramon Garcia and Ethan~Simpson Lee.
\newblock Unconditional explicit {M}ertens' theorems for number fields and {D}edekind zeta residue bounds.
\newblock {\em Ramanujan J.}, 57(3):1169--1191, 2022.

\bibitem[HB93]{HeathBrown}
D.R. Heath-Brown.
\newblock The size of {S}elmer groups for the congruent number problem.
\newblock {\em Inventiones Mathematicae}, 111(1):171--196, 1993.

\bibitem[Kan13]{Kane2013}
Daniel Kane.
\newblock On the ranks of the 2-{S}elmer groups of twists of a given elliptic curve.
\newblock {\em Algebra Number Theory}, 7(5):1253--1279, 2013.

\bibitem[KMR13]{KMR13}
Zev Klagsbrun, Barry Mazur, and Karl Rubin.
\newblock Disparity in {S}elmer ranks of quadratic twists of elliptic curves.
\newblock {\em Annals of Mathematics}, 178:287--320, 2013.

\bibitem[KMR14]{KMR14}
Zev Klagsbrun, Barry Mazur, and Karl Rubin.
\newblock A {Markov model for Selmer ranks} in families of twists.
\newblock {\em Compositio Mathematica}, 150:1077--1106, 2014.

\bibitem[Kob06]{Kob06}
Emi Kobayashi.
\newblock A remark on the {Mordell-Weil} rank of elliptic curves over the maximal abelian extension of the rational number field.
\newblock {\em Tokyo Journal of Mathematics}, 29(2), 2006.

\bibitem[KP23]{KoymansPaganoFairCounting}
Peter Koymans and Carlo Pagano.
\newblock Malle's conjecture for fair counting functions, 2023.
\newblock To appear in ANT.

\bibitem[KP25]{KP25-LD}
Daniel Keliher and Sun~Woo Park.
\newblock Large deviation principles for abelian monoids, 2025.
\newblock arXiv:2507.08173. To appear at Electronic Communications in Probability.

\bibitem[Lan94]{LangANT94}
Serge Lang.
\newblock {\em Algebraic number theory}, volume 110 of {\em Graduate Texts in Mathematics}.
\newblock Springer-Verlag, New York, second edition, 1994.

\bibitem[Liu04a]{Li04}
Yu-Ru Liu.
\newblock A generalization of the {Erd\"os-Kac} theorem and its applications.
\newblock {\em Canadian Mathematical Bulletin}, 47(4):589--606, 2004.

\bibitem[Liu04b]{Li04-2}
Yu-Ru Liu.
\newblock A generalization of the {Turan} theorem and its applications.
\newblock {\em Canadian Mathematical Bulletin}, 47:573--588, 2004.

\bibitem[LL25]{LL25}
Aaron Landesman and Ishan Levy.
\newblock The stable homology of {Hurwitz} modules and applications, 2025.
\newblock arXiv:2510.02068.

\bibitem[McM13]{McMullen13}
Curtis McMullen.
\newblock Complex analysis on riemann surfaces, 2013.
\newblock Course notes available at https://math.berkeley.edu/~ianagol/complexriemann.pdf.

\bibitem[MR07]{MR07}
Barry Mazur and Karl Rubin.
\newblock Finding large {S}elmer rank via an arithmetic theory of local constants.
\newblock {\em Annals of Mathematics}, 166:579--612, 2007.

\bibitem[MT93]{MT93}
Sean Meyn and Richard Tweedie.
\newblock {\em Markov chains and stochastic stability}.
\newblock Springer-Verlag, Berlin, Germany, 1993.

\bibitem[MV06]{Montgomery&Vaughan06}
Hugh Montgomery and Robert Vaughan.
\newblock {\em Multiplicative Number theory 1. Classical Theory: Cambridge studies in advanced mathematics 97}.
\newblock Cambridge University Press, 2006.

\bibitem[MZ16]{MZ16}
Behzad Mehrdad and Lingjiong Zhu.
\newblock Moderate and large deviations for the {E}rd\"os-{K}ac theorem.
\newblock {\em The Quarterly Journal of Mathematics}, 67(1):147--160, 2016.

\bibitem[Mä93]{MakiDensity}
Sirpa Mäki.
\newblock The conductor density of abelian number fields.
\newblock {\em Journal of the London Mathematical Society}, s2-47(1):18--30, 1993.

\bibitem[Nor76]{Norton76}
Karl~K. Norton.
\newblock On the number of restricted prime factors of an integer {I}.
\newblock {\em Illinois Journal of Mathematics}, 20:681--705, 1976.

\bibitem[Par25a]{Park25-1}
Sun~Woo Park.
\newblock {Mordell--Lang} and disparate {Selmer} ranks of odd twists of some superelliptic curves over global function fields, 2025.
\newblock arXiv:2504.20594.

\bibitem[Par25b]{park2022prime}
Sun~Woo Park.
\newblock On the prime {S}elmer ranks of cyclic prime twist families of elliptic curves over global function fields.
\newblock {\em Compositio Mathematica}, 161(12):3277--3320, 2025.

\bibitem[PR12a]{PR12}
Bjorn Poonen and Eric Rains.
\newblock Random maximal isotropic subspaces and {S}elmer groups.
\newblock {\em Journal of the American Mathematical Society}, 25(1):245--269, 2012.

\bibitem[PR12b]{PR12-theta}
Bjorn Poonen and Eric Rains.
\newblock Self cup products and the theta characteristic torsor.
\newblock {\em Mathematical Research Letters}, 18(6):1305--1318, 2012.

\bibitem[PS99]{PS99}
Bjorn Poonen and Michael Stoll.
\newblock The {Cassels-Tate} pairing on polarized abelian varieties.
\newblock {\em Annals of Mathematics}, 150(3):1109--1149, 1999.

\bibitem[PS25]{PS25}
Sun~Woo Park and Efthymios Sofos.
\newblock On the application of large deviation estimates to local solubility in families of varieties, 2025.
\newblock arXiv:2507.08173.

\bibitem[PT25]{PT25}
Jinzhao Pan and Ye~Tian.
\newblock On the distribution of {2-Selmer} ranks of quadratic twists of elliptic curves over {Q}, 2025.
\newblock arXiv:2503.21462.

\bibitem[Ros99]{Rosen}
Michael Rosen.
\newblock A generalization of {M}ertens' theorem.
\newblock {\em J. Ramanujan Math. Soc.}, 14(1):1--19, 1999.

\bibitem[Sat53]{Sathe53}
L.G. Sathe.
\newblock On a problem of {Hardy} on the distribution of integers having a given number of prime factors.
\newblock {\em The Journal of the Indian Mathematical Society}, 17:63--141, 1953.

\bibitem[Sch98]{Sch98}
Edward Schaefer.
\newblock {Computing a Selmer group of a Jacobian} using functions on the curve.
\newblock {\em Mathematische Annalen}, 310:447--471, 1998.

\bibitem[SD08]{SDtwists}
Peter Swinnerton-Dyer.
\newblock The effect of twisting on the 2-{S}elmer group.
\newblock {\em Math. Proc. Cambridge Philos. Soc.}, 145(3):513--526, 2008.

\bibitem[Sel54]{Selberg54}
A.~Selberg.
\newblock Note on a paper by l.g.sathe.
\newblock {\em The Journal of the Indian Mathematical Society}, 18:83--87, 1954.

\bibitem[Ser81]{Serre81}
Jean-Pierre Serre.
\newblock Quelques applications du théorème de densité de {Chebotarev}.
\newblock {\em Publications Mathématiques de l'IHÉS}, 54:123--201, 1981.

\bibitem[Smi26a]{Smith1}
Alexander Smith.
\newblock The distribution of $\ell^{\infty}$-{S}elmer groups in degree $\ell$ twist families {I}.
\newblock {\em Journal of the American Mathematical Society}, 39(1):1--72, 2026.

\bibitem[Smi26b]{Smith2}
Alexander Smith.
\newblock The distribution of $\ell^{\infty}$-{S}elmer groups in degree $\ell$ twist families {II}.
\newblock {\em Journal of the American Mathematical Society}, 39(2):453--514, 2026.

\bibitem[Wu96]{Wu96}
Jie Wu.
\newblock A sharpening of effective formulas of {S}elberg–{D}elange type for some arithmetic functions on the semigroup {$G_K$}.
\newblock {\em Journal of Number Theory}, 59(1):1--19, 1996.

\bibitem[Yu16]{Yu16}
Myungjun Yu.
\newblock {Selmer ranks of twists of hyperelliptic curves and superelliptic curves}.
\newblock {\em Journal of Number Theory}, 160:148--185, 2016.

\bibitem[Yu19]{Yu19}
Myungjun Yu.
\newblock {The distribution of Selmer ranks of quadratic twists of Jacobians of hyperelliptic curves}.
\newblock {\em Mathematical Research Letters}, 26(4):1217--1250, 2019.

\end{thebibliography}

\end{document}